\documentclass[11pt,twoside]{amsart}
\usepackage{amssymb}
\usepackage{amsmath}
\usepackage{bbm}
\usepackage[dvips]{graphicx}
\usepackage{hyperref}
\usepackage[capitalize,nameinlink,noabbrev,compress]{cleveref}

\usepackage{todonotes}
\usepackage{enumitem}
\usepackage{comment}
\usepackage{mathdots}
\usepackage{color}

\usepackage{tikz}
\usepackage{calc}

\usepackage{accents}
\newcommand{\dbtilde}[1]{\accentset{\approx}{#1}}

\usetikzlibrary{arrows}
\usetikzlibrary{decorations.markings}
\usetikzlibrary{shapes.geometric}

\numberwithin{equation}{section}

\newtheorem{thm}{Theorem}[section]
\newtheorem{prop}[thm]{Proposition}
\newtheorem{cor}[thm]{Corollary}
\newtheorem{lem}[thm]{Lemma}

\newtheorem{rem}[thm]{Remark}

\newtheorem{assum}[thm]{Assumption}

\title{On the domino shuffle and matrix refactorizations}

\author{Sunil Chhita}
\address{Sunil Chhita, Department of Mathematical Sciences, 
Durham University, Durham, UK}
\email{sunil.chhita@durham.ac.uk}

\author{Maurice Duits}
\address{Maurice Duits, Department of Mathematics, 
KTH Royal Institute of Technology, Stockholm, Sweden.}
\email{duits@kth.se}


\date{\today}

\begin{document}

\begin{abstract}
	This paper is motivated by computing correlations for domino tilings of the Aztec diamond. It is inspired by two of the three distinct methods that have recently been used  in the simplest case of a doubly periodic weighting, that is the two-periodic Aztec diamond. One of the methods, powered by the domino shuffle, involves inverting the Kasteleyn matrix giving correlations through the local statistics formula. Another of the methods, driven by a Wiener-Hopf factorization for two-by-two matrix valued functions, involves the Eynard-Mehta Theorem. For arbitrary weights the Wiener-Hopf factorization can be replaced by an LU- and UL-decomposition, based on a  matrix refactorization, for the product of the transition matrices. This paper shows that, for arbitrary weightings of the Aztec diamond, the evolution of the face weights under the domino shuffle and the matrix refactorization is  the same. In particular, these dynamics can be used to find the inverse of the LGV matrix in the Eynard-Mehta Theorem. 

\end{abstract}

\maketitle
\tableofcontents

\section{Introduction} 
\label{sec:Introduction}
\subsection{Domino tilings of the Aztec diamond}
Random tiling models of bounded regions have been studied heavily in the past few decades; see \cite{Gor20} and references therein. 
One of the central examples of this area are domino tilings of the Aztec diamond, where an Aztec diamond of size $n$ is all the squares of the square grid whose centers satisfy the condition that $|x|+|y| \leq n$ and a domino tiling is a non-overlapping covering by two by one rectangles~\cite{EKLP92}. 
To obtain a random domino tiling of the Aztec diamond, one assigns weights to particular dominoes, which can be dependent on their location, picking each domino tiling with probability proportional to the product of the domino weights in that domino tiling. These models are often studied on the dual graph with a tile becoming a dimer, and the resulting random tiling probability measure is known as the dimer model.

These random tilings contain many fascinating asymptotic behaviors that should be apparent in other statistical mechanical models.  Indeed, for large random tilings limit shape curves emerge splitting the domain into different macroscopic regions, of which there are three types: frozen, where the configurations are deterministic; rough, where the correlations between tiles decay polynomially; smooth, where the correlations between tiles decay exponentially.  These phases were characterized for dimer models on bipartite graphs in~\cite{KOS03}.

To study these interesting asymptotic behaviors, one of the main approaches in recent years has been to find a non-intersecting path picture for the tiling. Using a combination of the Lindstr\"om-Gessel-Viennot theorem  and the Eynard-Mehta Theorem (e.g. see~\cite{RB04}),  the correlation kernel of the underlying determinantal point process of the particle system defined through the paths can be written in terms of the inverse of a particular principal submatrix of a product of transition matrices.  Finding an explicit expression for that inverse and thus the correlation kernel,  one that is amenable for asymptotic analysis, poses a serious challenge that has only been carried out in special situations.   For instance, it has been worked out for models that are Schur processes, such as uniformly random domino tilings of the Aztec diamond~\cite{Jo03}. For Schur processes,  the transition matrices are doubly infinite Toeplitz matrices and (in an appropriate limit) the inverse can be computed using a Wiener-Hopf factorization of the product of the symbols. As the symbols are scalar-valued, finding a Wiener-Hopf factorization is a mere reordering of the symbols in the product (see for example~\cite{Joh17}). In~\cite{BD19} the authors introduced a natural generalization of Schur process, one that includes doubly periodically weighted domino tiling of the Aztec diamond, by taking block Toeplitz matrices as transition matrices. In this case, the symbols are matrix-valued and this complicates a Wiener-Hopf factorization.    Still, it is possible to define a refactorization procedure that provides such a Wiener-Hopf factorization, and in special situations this Wiener-Hopf factorization is even explicit.  Once formulas for the correlation kernel of the determinantal point process have been found in a suitable form, fine asymptotic analysis unlocks the full asymptotic picture, which is often unavailable in more complicated models.

An alternate approach for random tiling models has been through the Kasteleyn matrix $K$ and its inverse.  The (Percus)-Kasteleyn matrix is a type of signed adjacency matrix whose rows and columns are indexed by the black and white vertices of the graph respectively. The appeal of the inverse Kasteleyn matrix is that it is the correlation kernel of the determinantal point process on the edges of the graph \cite{Ken97}.  Computations of this inverse are only known for periodic graphs and in certain special cases such as the Aztec diamond.  For instance, a procedure for computing the inverse Kasteleyn matrix for the Aztec diamond was given~\cite{CY14} in a certain setting, which showed that the entries could be computed using recurrence relations from an entry-wise expansion of the matrix equations $K.K^{-1}=K^{-1}.K=\mathbbm{I}$ and a boundary recurrence relation.  This boundary recurrence relation involved transformations of the entries of the inverse Kasteleyn matrix under the domino shuffle\footnote{We define the domino shuffle as applying the square move on all even faces followed by edge contraction of all two valent vertices and a shift; see \cref{sec:dynamics} for details.}, which is a particular graphical move special to 4-valent faces of the graph.

The {main} purpose of this paper is to show that both formulations for computing correlations for \emph{arbitrarily} weighted domino tilings of the Aztec diamond are equivalent, relying on the same amount of computational complexity.  {Along the way, we show that commuting transition matrices arising from the non-intersecting path picture for arbitrary weights is equivalent to the domino shuffle, thus providing an analog to the role of Yang-Baxter for the six-vertex model~\cite{Bax82}; see Corollary~\ref{cor:shufflematrixsame}. {This had only previously been noted for Schur processes~\cite{BCC:14}}. We next give an overview of the main results along with an outline of the paper.

\subsection{Outline of the main results}

Since our results hold for arbitrary weights and minimal assumptions, we start our discussion of the Aztec diamond from first principles. We therefore start in Section~\ref{sec:Preliminaries} by recalling the basics on the Kasteleyn approach, and the Eynard-Mehta Theorem for the non-intersecting paths in Section~\ref{sec:nonintersecting}.

Our first main result is a general expression in Theorem~\ref{thm:eqAztecLGV} for the inverse Kasteleyn matrix that involves the inverse of a matrix that counts the DR-paths on the Aztec diamond,    very similar to  the Eynard-Mehta Theorem for the non-intersecting paths process. In fact, when setting this up in the slightly larger domain, cf. Theorem~\ref{thm:AztecTowerLGV}, it is exactly the same matrix from the Eynard-Mehta Theorem that needs to be inverted. This shows that the two approaches ultimately boil down to the same question.

In Section~\ref{sec:dynamics} we then discuss and compare two fundamental discrete dynamical systems on the infinite underlying weighted graphs for the dimer model and the non-intersecting path process. For the dimer model, we apply the well-known square move on all even faces, as one does in the domino shuffle.  For the non-intersecting path processes,  we use a matrix refactorization by swapping all even transition matrices with their consecutive odd neighbor. We show that these two systems are the equivalent in the sense that each iteration changes the face weights of the underlying graphs identically. For special choices of doubly periodic weights (as we briefly discuss below), both systems have been used in the literature to compute the correlation functions, and we will show that this can be done for arbitrary weights.

The dynamics (and its reverse) provides us, up to a trivial shift, with an LU- and UL-decomposition  of the product of the transition matrices.   For the Eynard-Mehta Theorem we need to invert a particular submatrix of the product of transition matrices, and the LU- and UL decompositions do not immediately provide an inverse of this matrix. Inspired by a similar analysis for the special case of block Toeplitz matrices~\cite{Wid74}, we show in Section~\ref{sec:refactorizaion}  that it is possible to provide an explicit expression for an approximate inverse, with only  very minor assumptions on the weights (in particular, no periodicity is required). The approximation converges to the inverse when the size of the submatrix tends to infinity. Our analysis culminates in an expression of the correlation kernel that only involves the LU-decomposition and the transition matrices, cf. Theorem~\ref{thm:generalKernel}. Although the expression we obtain is not yet in a form that one can start an asymptotic study, it is valid under fairly weak conditions on the parameters, and we find it remarkable that it can be carried out in this generality. Moreover, it covers the result of~\cite{BD19} as a special case, including doubly periodic weights, and even provides an alternative more direct proof, which we included in Appendix~\ref{appendix:BD}. 

The results of Section~\ref{sec:periodic} are discussed in the next subsection.  Finally in Section~\ref{sec:inversesbyKas}, we show how to use the domino shuffle to compute the boundary recurrence relations, a method used in~\cite{CY14} and outline the steps needed to compute the inverse of the Kasteleyn matrix of the Aztec diamond.

 \subsection{Doubly periodic weightings}

In Section~\ref{sec:periodic} we discuss how our general procedure specializes to doubly periodic domino tilings of the Aztec diamond in which there has been significant progress in recent years. The attraction of these types of models is that they are currently the only statistical mechanical model with all three types of macroscopic regions present which also have explicit formulas for their correlations.  {Indeed, the original motivation for studying the two-periodic Aztec diamond was to study the probabilistic behavior at the rough-smooth boundary which is a transition between polynomially and exponentially decaying regions. This type of interface is believed to appear in other statistical mechanical models such as the six-vertex model with domain wall boundary conditions (with the associated parameter $\Delta<-1$) and low temperature 3D Ising models with certain boundary conditions. }

The starting point was in~\cite{CY14}, where a formula for the inverse Kasteleyn matrix was derived for the two-periodic Aztec diamond which was later simplified to a form suitable for asymptotic analysis in~\cite{CJ16}, leading to asymptotic results for the two-periodic Aztec diamond including the behavior at the rough-smooth boundary~\cite{BCJ18,BCJ22,JM21,Bai22} {which shows that the behavior is much more nuanced than the behavior at the frozen-rough boundary}.   An alternative approach is to find the correlation kernel of the determinantal point process for the particle system associated to the non-intersecting path picture which gives two different methods.  One of these methods used matrix orthogonal polynomials combined with Riemann Hilbert techniques~\cite{DK21} while the other method, introduced in~\cite{BD19}, is an important inspiration for the  general  refactorization  of the present paper.

For doubly periodic weights, the transition matrices are doubly infinite block Toeplitz.  The refactorization procedure in this case   is equivalent to a Wiener-Hopf factorization of the symbol corresponding to the product of the transition matrices. For various models, such as the two-periodic Aztec diamond and even a class of weightings with higher periodicity \cite{Ber19},  the dynamical system determined by the refactorization procedure is periodic and this allows for a very explicit double integral formulation for the correlation kernel that can be analyzed asymptotically. In general, tracing this dynamical system is not an easy task.  In a recent work~\cite{BD22}, the authors showed that for the biased two-periodic Aztec diamond the dynamical system from the refactorization is equivalent to a linear flow on an elliptic curve, and it is reasonable to expect that the general case can be linearized on the Jacobian of the spectral curve.

{It is important to note that the (matrix)-orthogonal method of~\cite{DK21} for doubly periodically weighted tilings, is essentially based on an LU-decomposition of the submatrix of the product of transition matrices, whereas the Wiener-Hopf factorization in~\cite{BD19} is an LU- and UL-decomposition for the entire matrix. The LU-decomposition is hiding in the orthogonality condition, but it was an important fact in~\cite{DK21}.  The benefit of the approach of~\cite{DK21} is that it also holds in more general situations. Moreover,  one can use tools from complex analysis, such as the Riemann-Hilbert problem, to study these polynomials. This has been  carried out for several interesting tiling models~\cite{DK21,Cflower,CDKL,GrKu21}. For the Aztec diamond the polynomials simplify significantly, and therefore one can circumvent this heavy machinery. It is also important to observe that the orthogonal polynomials only occur for weightings that have at least one direction in which they are periodic. }

Restricting our results to doubly-periodic weights, means that Wiener-Hopf factorization is equivalent to the domino shuffle. The dynamical system from domino shuffle, known as the dimer cluster integrable system introduced in~\cite{GK13}, has been studied extensively in various contexts~\cite{GSTV:16, KLRR:18, AGR:21, Izo21} for example, under the guise of the octahedron recurrence~\cite{Spe07, DFSG:14, DiFr14}. These dynamics also have a probabilistic interpretation~\cite{CT19,CT21} and when applied to the Aztec diamond, giving a powerful method for perfect simulation of domino tilings of the Aztec diamond with arbitrary weights~\cite{Pro03}. 
Finally, we mention that the connection between the dimer cluster integrable system and matrix refactorization is currently being investigated~\cite{BGR22}.

\subsection*{Acknowledgements}
{Both authors wish to thank the Galileo Galilei Institute for their hospitality and support during the scientific program on \lq Randomness, Integrability, and Universality\rq.  SC was supported by EPSRC EP\textbackslash T004290\textbackslash 1.  MD was partially supported by the Swedish Research Council (VR), grant no 2016-05450 and grant no. 2021-06015, and the European Research Council (ERC), Grant Agreement No. 101002013.

We thank C. Boutillier and T. Helmuth for helpful discussions.  
We would also like to thank Alexei Borodin, in particular for discussions on the proof of Theorem~\ref{thm:bd19}.  We are grateful to Tomas Berggren for fruitful discussions and  his comments on a preliminary version of the paper, especially for pointing out the content of Remark~\ref{rem:periodicitylost} to us. Finally, we would like to thank the anonymous referee for their detailed comments which led to a significant improvement to this manuscript.

\section{Preliminaries on the Aztec diamond and the Kasteleyn approach} \label{sec:Preliminaries}

In this section, we give the general setup, definitions of the Aztec diamond and tower Aztec diamonds graphs, and the Kasteleyn matrices associated to these graphs. We specify the importance of the inverse (of the) Kasteleyn matrix for computing correlations. Finally, we introduce the DR-path picture for domino tilings of the Aztec diamond and tower Aztec diamond graphs. 

We note that throughout the paper, we use the notation that for a matrix $M=(m_{i,j})_{1 \leq i,j \leq n}$, $M(i,j)=m_{i,j}$, depending on whichever is most convenient.   We will also use the notation $\mathbbm{I}_B$ to denote the indicator of a set $B$.

\subsection{General Setup} \label{subsec:generalsetup}
We consider the dimer model on (finite) planar bipartite graphs $G=(V,E)$.  A dimer configuration is a collection of edges such that each vertex is incident to exactly one edge of the collection.  To each edge of the graph, assign a positive number, that is $w:E \to \mathbb{R}$ with $w(e)>0$ for all $e \in E$.  The weight of a dimer configuration, $M$,  is equal to the product of the edge weights in that dimer configuration, that is $\prod_{e \in M} w(e)$.  The dimer model is the probability measure where each dimer configuration is
picked with probability proportional to the product of the edge weights. In other words, the probability of a dimer configuration $M$ is equal to 
\[
    \frac{\prod_{e \in M}w(e)}{\sum_{N \in \mathcal{M}}\prod_{e \in N}w(e)}
\]
where the sum is over all possible dimer configurations, $\mathcal{M}$, of the graph $G$. 

For this paper, we work only on the square grid and particular subgraphs of it.   Introduce the vertex sets
\[
\mathtt{W}=\{(i,j)\in \mathbb{Z}^2: i \hspace{-3mm}\mod 2=1, \, j\hspace{-3mm}\mod 2=0\} 
\]
and
\[
\mathtt{B}=\{(i,j)\in \mathbb{Z}^2: i\hspace{-3mm}\mod 2=0, \,j\hspace{-3mm}\mod 2=1\},
\]
which denote the white and black vertices.  The centers of the faces of the infinite graph are given by $(i,j) \in (2\mathbb{Z})^2$ or $(i,j) \in (2\mathbb{Z}+1)^2$.  

Although we introduced the edge weights above, as mentioned in~\cite{GK13}, it is in fact the face weights which  parameterize the dimer model\footnote{Put briefly, one can obtain any dimer configuration from another using pairwise flips of dimers around the faces of the graph.  Under each flip, the weight of the configuration  changes (from the weight of the configuration prior to that flip) by a multiplying or dividing through by a face weight.}.  The face weights are defined as the alternating product of the edge weights around each face viewed from a  clockwise orientation, such that the weights of the edges of the black to white vertices are in the numerator and the weights of the edges of the white to black vertices are in the denominator of this alternating product. 
  For $i,j \in \mathbb{Z}$, let the face weight of the face whose center is given by $(2i+1,2j+1)$ be equal to $F_{2i,j}$ and let the face weight of the face whose center is given by $(2i+2,2j+2)$ be equal to $F_{2i+1,j}$.

Without loss of generality, we fix a convention for the edge-weights used throughout this paper. For $\mathtt{b}=(2i,2j+1)$ with $i,j \in \mathbb{Z}$ we assert that the edge weights of the edges $(\mathtt{w},\mathtt{b})$ are given by
\begin{itemize}
\item 1 if $\mathtt{w}=(2i-1,2j+2)$ or $\mathtt{w}=(2i-1,2j)$
\item $a_{i,j}$ if $\mathtt{w}=(2i+1,2j+2)$, 
\item $b_{i,j}$ if $\mathtt{w}=(2i+1,2j)$
\end{itemize}
for $a_{i,j},b_{i,j}>0$ for all $i,j \in \mathbb{Z}$. With our conventions, each face $(2i+1,2j+1)$ has face weight $F_{2i,j}=\frac{a_{i,j}}{b_{i,j}}$ while each face $(2i+2,2j+2)$ has face weight $F_{2i+1,j}=\frac{b_{i+1,j+1}}{a_{i+1,j}}$. We will use this weighting throughout the paper.

\begin{rem} \label{rem:periodicitylostpre}
    Note that setting the weight of the edges  $((2i,2j+1),(2i-1,2j))$ equal to $1$ can be done without loss of generality. Indeed, in a general edge weighting one can always change all edge weights so that the edges  $((2i,2j+1),(2i-1,2j))$ have weights $1$ without changing the face weights. This can be achieved by a so-called successive application of gauge transformations, in which one multiplies each edge weight around a given vertex by a common factor. We note, however, that this can have an effect on the structure of the edge weights (see for instance Remark~\ref{rem:periodicitylost}).
\end{rem}

\subsection{The Aztec diamond}

We introduce the Aztec diamond graph of size $n$ denoted by $G_n^{\mathrm{Az}}$.  Let
\begin{equation}
\mathtt{W}_n^{\mathrm{Az}} =\{(2j+1,2k):0 \leq j \leq n-1 , 0 \leq k \leq n\} 
\end{equation}
and 
\begin{equation}
\mathtt{B}_n^{\mathrm{Az}} =\{(2j,2k+1):0 \leq j \leq n , 0 \leq k \leq n-1\}.
\end{equation}
The edges are given by
\begin{equation}
\begin{split}
&\mathtt{E}_n^{\mathrm{Az}} =\{((2j+1,2k),(2j+1\pm1,2k+1)):0 \leq j \leq n-1 , 0 \leq k \leq n-1\}\\
\cup&\{((2j+1,2k),(2j+1\pm1,2k-1)):0 \leq j \leq n-1 , 1 \leq k \leq n\}.
\end{split}
\end{equation}

Let $[n]=\{1,\dots,n\}$.   We assign a specific ordering to the white and black vertices, which are given by the functions $w_n^{\mathrm{Az}}:[n(n+1)] \to \mathtt{W}_n^{\mathrm{Az}}$ and $b_n^{\mathrm{Az}}:[n(n+1)] \to \mathtt{B}_n^{\mathrm{Az}}$ where
\begin{equation}\label{eq:wAz2}
    w_n^{\mathrm{Az}}(i) = \left\{\begin{array}{ll}
        (2i-1,0) &\mbox{if } 1 \leq i \leq n \\
    (2 [i-1]_n+1, 2n+2-2\lfloor \frac{i-1}{n} \rfloor)&  \mbox{otherwise} \end{array} \right.
\end{equation}
and
\begin{equation}\label{eq:bAz2}
b_n^{\mathrm{Az}}(i) = \left\{\begin{array}{ll}
(0,2i-1) &\mbox{if } 1 \leq i \leq n \\
(2 [i-1]_n+2, 2n+1-2\lfloor \frac{i-1}{n} \rfloor) &\mbox{otherwise} \end{array} \right.
\end{equation}
for $1 \leq i \leq n(n+1)$, where $[i]_n=i\mod n$.  See Fig.~\ref{fig:Azteclabels} for an example of these labels.  
\begin{center}
\begin{figure}
    \includegraphics[height=6cm]{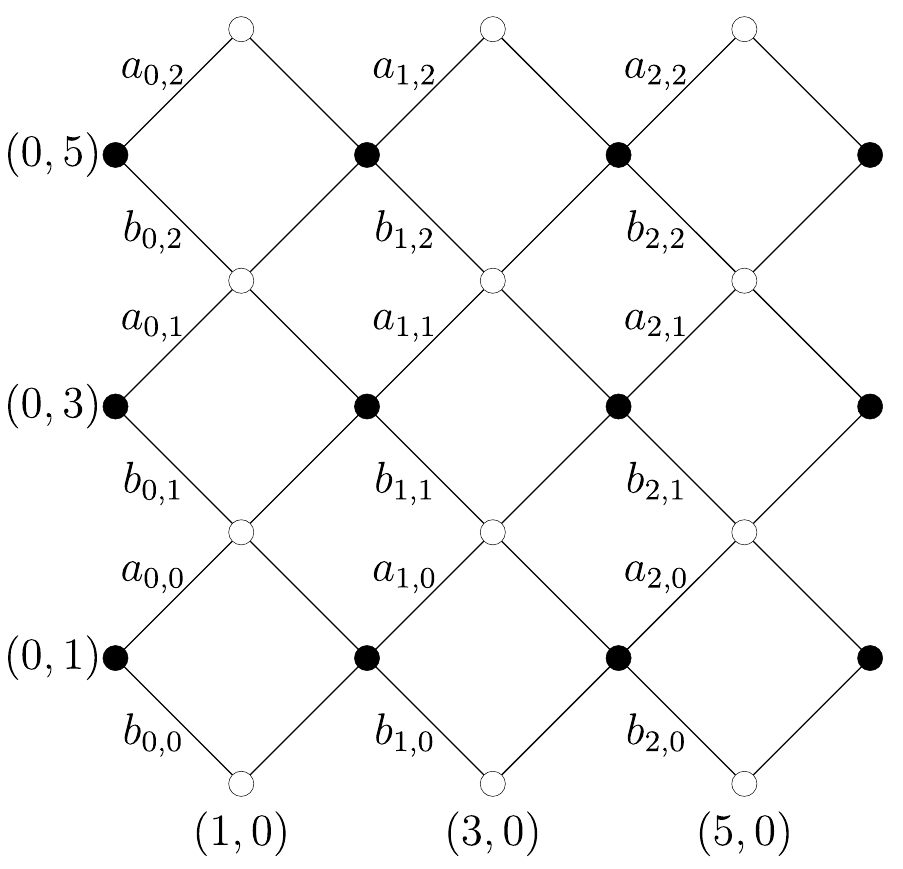}
\includegraphics[height=6cm]{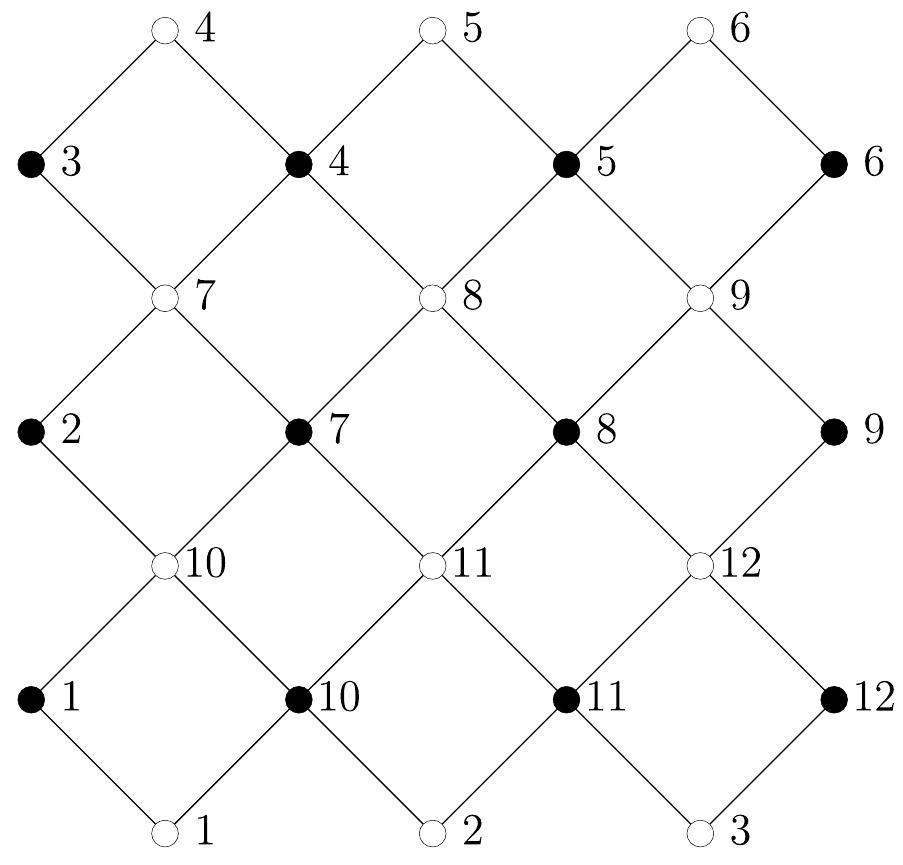}
    \caption{An Aztec diamond of size 3 with the Cartesian coordinates given on the left and the vertex labels on the right, including the edge weights with our conventions given in \cref{subsec:generalsetup}. The unmarked edges have weight 1.  }
    \label{fig:Azteclabels}
\end{figure}
\end{center}

The Kasteleyn(-Percus) matrix on $G_n^{\mathrm{Az}}$, $K_n^{\mathrm{Az}}:\mathtt{B}_n^{\mathrm{Az}} \times\mathtt{W}_n^{\mathrm{Az}} \to \mathbb{C}$ is given by 
\begin{equation}\label{eq:KnAz}
K_n^{\mathrm{Az}}(x,y)=\left\{ \begin{array}{ll}
1  & \mbox{if $y-x=(-1,-1)$} \\
\mathrm{i} & \mbox{if $y-x=(-1,1)$} \\
a_{r,s} & \mbox{if $y-x=(1,1)$} \\
b_{r,s} \mathrm{i} & \mbox{if $y-x=(1,-1)$} \\
0 & \mbox{otherwise,} \end{array} \right.
\end{equation} 
where $x=(x_1,x_2)$, $r=x_1/2$, and $s=(x_2-1)/2$. 
The significance of the Kasteleyn matrix is explained below.

\subsection{Tower Aztec diamond}
We next introduce the tower Aztec diamond of size $n$ and corridor of size $p$. Informally speaking, this is two Aztec diamonds of size $n$ and $n-1$ stitched together by a strip of the square grid of size $p$; see Fig.\ref{fig:AztecTowlabels}.  This model was introduced in~\cite{BD19}, however, it was not assigned a name. It takes little effort to see that it is not possible to have a dimer configuration of the tower Aztec diamond in which there is an edge with one vertex in the strip and the other in one of the Aztec diamonds. In fact, each dimer configuration consists of three independent dimer configurations: one for each of the two Aztec diamonds and a trivial configuration for the strip. The benefit of using the tower Aztec diamond is that it allows us to use infinite matrices in our analysis, after letting $p$ tend to infinity; see Section~\ref{sec:refactorizaion}. 
Let
\begin{equation}
    \mathtt{W}_{n,p}^{\mathrm{Tow}} =\{(2j+1,2k):0 \leq j \leq n-1 , -p-n+1 \leq k \leq n\} 
\end{equation}
and 
\begin{equation}
    \begin{split}
        &\mathtt{B}_{n,p}^{\mathrm{Tow}} =\{(2j,2k+1):0 \leq j \leq n , 0 \leq k \leq n-1 \\ &\mbox{ or  }0 \leq j \leq n-1 , -p \leq k \leq -1 \mbox{ or }1 \leq j \leq n-1 , -p-n \leq k \leq -1-p\}.
    \end{split}
\end{equation}
The edges here are given by
\begin{equation}
\begin{split}
    &\mathtt{E}_{n,p}^{\mathrm{Tow}} =\{((2j+1,2k),(2j+2,2k+1)):0 \leq j \leq n-1 , 0 \leq k \leq n-1  \\ \,\, &\mbox{ or }0\leq j \leq n-2,-p-n+1 \leq k \leq -1 \}\\
    &\cup \{((2j+1,2k),(2j,2k+1)):0 \leq j \leq n-1 , -p \leq k \leq n-1 \\ \,\, &\mbox{ or }0\leq j \leq n-2,-p-n+1 \leq k \leq -1-p \}\\
    &\cup \{((2j+1,2k),(2j+2,2k-1)):0 \leq j \leq n-1 , 1 \leq k \leq n\\ \,\, &\mbox{ or }0\leq j \leq n-2,1 -p-n\leq k \leq n \}\\
&\cup \{((2j+1,2k),(2j,2k-1)):0 \leq j \leq n-1 , 1-p \leq k \leq n\\ \,\, &\mbox{ or }1\leq j \leq n-1,1 -p-n\leq k \leq n \}.
\end{split}
\end{equation}
Label $G_{n,p}^{\mathrm{Tow}}=(\mathtt{W}_{n,p}^{\mathrm{Tow}}\cup\mathtt{B}_{n,p}^{\mathrm{Tow}},\mathtt{E}_{n,p}^{\mathrm{Tow}})$ to be the tower Aztec diamond of size $n$ with corridor of size $p$.  
We assign a specific ordering to the white and black vertices, which are given by the functions $w_n^{\mathrm{Tow}}:[n(2n+p)] \to \mathtt{W}_{n,p}^{\mathrm{Tow}}$ and $b_n^{\mathrm{Tow}}:[n(2n+p)] \to \mathtt{B}_{n,p}^{\mathrm{Tow}}$ where
\begin{equation}\label{eq:wTow}
    w_{n,p}^{\mathrm{Tow}}(i) =\left\{ \begin{array}{ll} 
        (2n-1,2-2i) &\hspace{-50mm} \mbox{for }1 \leq i \leq n+p \\
        \Bigg(2[i-(n+p+1)]_n+1, 2n-2\lfloor\frac{i-(n+p+1)}{n}\rfloor \Bigg) \\ &\hspace{-50mm} \mbox{for }   n+p+1 \leq i \leq 2n+p+n^2-1  \\
        \Bigg(2[i-(n^2+p+2n)]_{n-1}+1, -2-2\lfloor\frac{i-(n^2+p+2n)}{n-1}\rfloor \Bigg)\\ & \hspace{-50mm}\mbox{otherwise}
    \end{array} \right. 
\end{equation}
and
\begin{equation}\label{eq:bTow}
b_{n,p}^{\mathrm{Tow}}(i) =\left\{ \begin{array}{ll} 
        (0,2n+1-2i) &\hspace{-50mm} \mbox{for }1 \leq i \leq n+p \\
	\Big(2[i-(n+p+1)]_n+2, 2n-2\lfloor\frac{i-(n+p+1)}{n}\rfloor -1\Big)\\  &\hspace{-50mm} \mbox{for }   n+p+1 \leq i \leq n+p+n^2  \\
	\Big(2[i-(n^2+n+p+1)]_{n-1}+2, -2\lfloor\frac{i-(n^2+n+p+1)}{n-1}\rfloor -1\Big) \\ &\hspace{-50mm} \mbox{otherwise}   
\end{array}\right.
\end{equation}
where we recall that $[i]_n=i\mod n$.  See Fig.~\ref{fig:AztecTowlabels} for an example of these labels.  
\begin{center}
\begin{figure}
    \includegraphics[height=7cm]{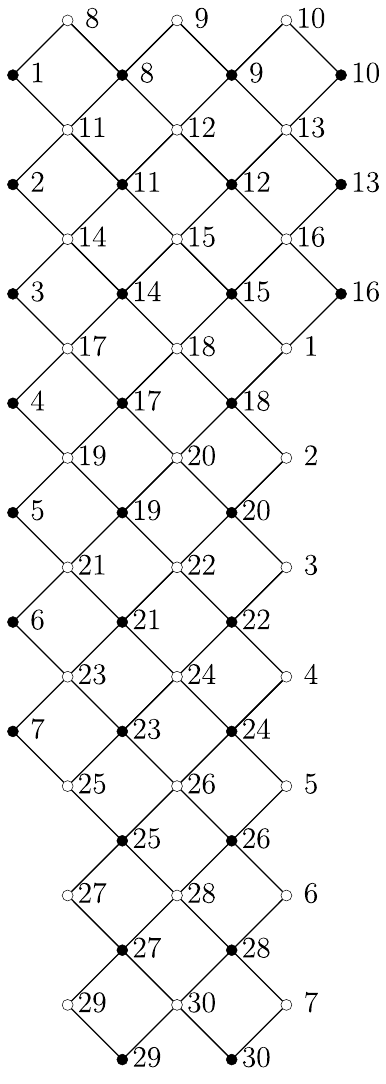}\,\,
\includegraphics[height=7cm]{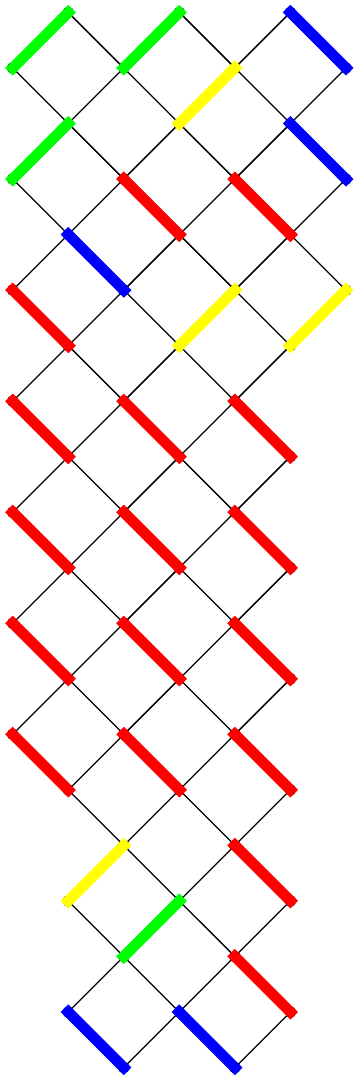}\,\,
\includegraphics[height=7cm]{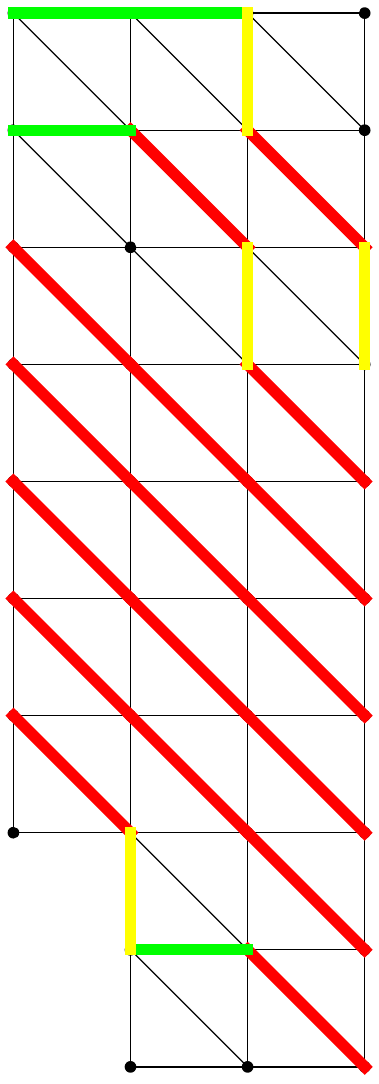}
    \caption{The left figure shows a tower Aztec diamond of size $3$ with corridor of size $4$ with the vertex labels described in~\eqref{eq:wTow} and~\eqref{eq:bTow}. The middle figure shows an example of a tiling and the right figure shows the corresponding non-intersecting lattice path picture.  }
    \label{fig:AztecTowlabels}
\end{figure}
\end{center}

The Kasteleyn matrix on $G_{n,p}^{\mathrm{Tow}}$, defined by $K_{n,p}^{\mathrm{Tow}}:\mathtt{B}_{n,p}^{\mathrm{Tow}} \times\mathtt{W}_{n,p}^{\mathrm{Tow}} \to \mathbb{C}$ is given by 
\begin{equation}\label{eq:KnTow}
K_{n,p}^{\mathrm{Tow}}(x,y)=\left\{ \begin{array}{ll}
1  & \mbox{if $y-x=(-1,-1)$} \\
\mathrm{i} & \mbox{if $y-x=(-1,1)$} \\
a_{r,s} & \mbox{if $y-x=(1,1)$} \\
b_{r,s} \mathrm{i} & \mbox{if $y-x=(1,-1)$} \\
0 & \mbox{otherwise,} \end{array} \right.
\end{equation} 
where $x=(x_1,x_2)$, $r=x_1/2$, and $s=(x_2-1)/2$.

\subsection{The Kasteleyn method}

Kasteleyn's theorem~\cite{Kas61,Kas63,TF61} gives that $|\det K_n^{\mathrm{Az}} |$ equals the number of weighted dimer coverings on $G_n^{\mathrm{Az}}$ while $|\det K_{n,p}^{\mathrm{Tow}} |$ equals the number of weighted dimer coverings on $G_{n,p}^{\mathrm{Tow}}$.  We have chosen the sign conventions in~\eqref{eq:KnAz} and~\eqref{eq:KnTow} so that the sign of $\det K_n^{\mathrm{Az}}$ equals $(-1)^{\lfloor (n+1)/2 \rfloor}$ whereas the sign in $\det K_{n,p}^{\mathrm{Tow}}$ equals 
\[
\frac{1}{2} (1+(-1)^p)(-1)^n + \frac{1}{2} (1-(-1)^p) \mathrm{i}^n.
\]
Both of these follow after a computation which we omit in this paper.

In what follows below, it is useful to define $K_n=K_n(w_n^{\mathrm{Az}}(i),b_n^{\mathrm{Az}}(j))_{1\leq i,j \leq n(n+1)}$ which is the Kasteleyn matrix for the Aztec diamond using the specific ordering of the white and black vertices as well as defining  $K_{n,p}=K_{n,p}(w_{n,p}^{\mathrm{Tow}}(i),b_{n,p}^{\mathrm{Tow}}(j))_{1\leq i,j \leq n(2n+p)}$ which is the Kasteleyn matrix for the tower Aztec diamond.  This gives a more compact notation for our Kasteleyn matrices and the subscript $p$ helps distinguish between the two.

The inverse of the Kasteleyn matrix can be used to compute statistics. We only state our result for the Aztec diamond graph and the formulation is analogous for the tower Aztec diamond as well as other graphs.  Suppose that $E=\{e_i\}_{i=1}^m$ with $e_i=(b_i,w_i)$  are a collection of distinct edges  with $b_i$ and $w_i$ denoting black and white vertices.  

\begin{thm}[\cite{Ken97,Joh17}]
The dimers form a determinantal point process on the edges of the Aztec diamond graph with correlation kernel given by $L(e_i,e_j)=K_n(b_i,w_i) K_n^{-1}(w_j,b_i)$, that is
$$
\mathbb{P}[e_1,\dots, e_m \mbox{ are covered by dimers}]=\det L(e_i,e_j)_{i,j=1}^m.
$$
\end{thm}

\subsection{DR paths for the Aztec diamond}

Associated to each dimer covering of the Aztec diamond of size $n$, there are DR-lattice paths~\cite{Jo03}.   The vertex set for the DR-lattice paths is given by
\begin{equation}
    \mathtt{V}_n^{\mathrm{Az},\mathrm{DR}}=\{(2j,2k-1):0 \leq j \leq n , 0 \leq k \leq n\},
\end{equation}
and the edge set 
\begin{equation}
\begin{split}
    \mathtt{E}_n^{\mathrm{Az},\mathrm{DR}}=
&\{((2j,2k-1),(2j+2,2k-1)) :0 \leq j \leq n-1 , 0 \leq k \leq n\} \\ & \cup\{((2j,2k-1),(2j,2k-3)) :0 \leq j \leq n , 1 \leq k \leq n\} \\ 
&\cup \{((2j,2k-1),(2j+2,2k-3)) :0 \leq j \leq n-1 , 1 \leq k \leq n\}.
\end{split}
\end{equation}
\begin{center}
\begin{figure}
    \includegraphics[height=5cm]{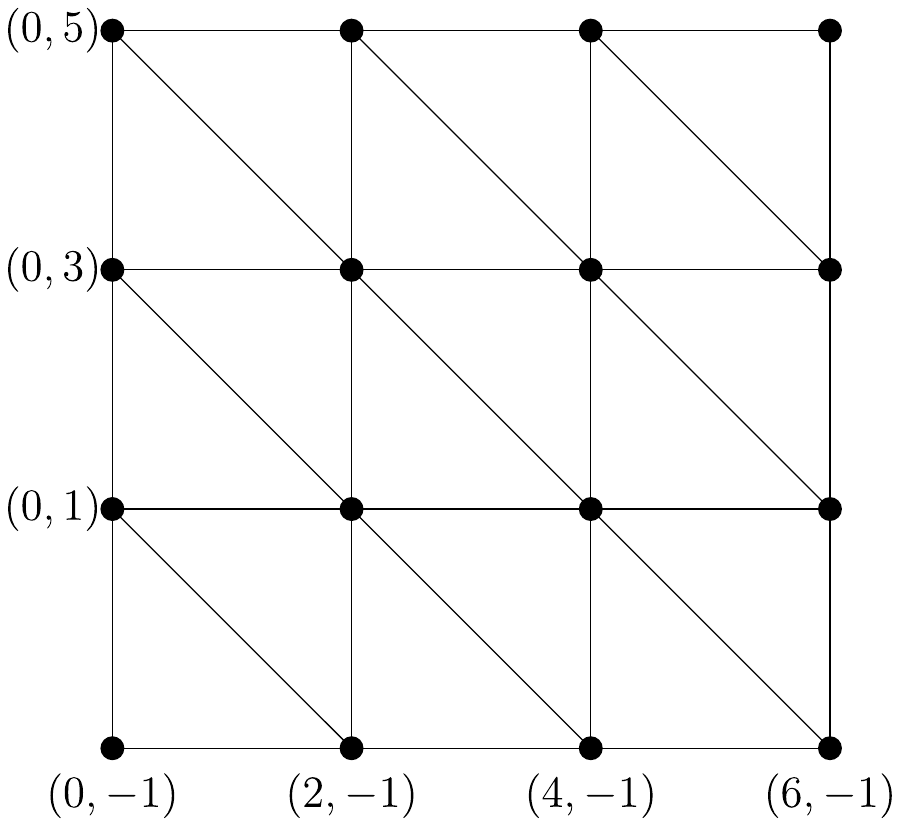}
    \caption{The DR graph corresponding to an Aztec diamond of size 3.}
    \label{fig:LGVgraph}
\end{figure}
\end{center}
Then, we write $G_n^{\mathrm{Az},\mathrm{DR}}=( \mathtt{V}_n^{\mathrm{Az},\mathrm{DR}}, \mathtt{E}_n^{\mathrm{Az},\mathrm{DR}})$ and label this graph the DR graph for the Aztec diamond. 
Fig.~\ref{fig:LGVgraph} shows an example of the DR graph for the Aztec diamond. 
The lattice paths start at the vertices $\{(0,2k-1), 0 \leq k \leq n\}$ and end at $\{(2j,-1), 0 \leq j \leq n\}$. We use the convention that we drop the path from $(0,-1)$ to $(0,-1)$ as this path is trivial. The paths are non-intersecting meaning that they cannot share a vertex.

The correspondence between dimers on the Aztec diamond graph and the DR lattice paths is given as follows:
\begin{itemize}
    \item if a dimer covers the edge $((2i,2j+1),(2i+1,2j))\in \mathtt{E}_n^{\mathrm{Az}}$ with $(2i,2j+1) \in \mathtt{B}_n^{\mathrm{Az}}$, then there is an edge $((2i,2j+1),(2i+2,2j-1))$ in $\mathtt{E}_n^{\mathrm{Az},\mathrm{DR}}$;
    \item if a dimer covers the edge $((2i,2j+1),(2i+1,2j+2))\in \mathtt{E}_n^{\mathrm{Az}}$ with $(2i,2j+1) \in \mathtt{B}_n^{\mathrm{Az}}$, then there is an edge $((2i,2j+1),(2i+2,2j+1))$ in $\mathtt{E}_n^{\mathrm{Az},\mathrm{DR}}$; 
    \item if a dimer covers the edge $((2i,2j+1),(2i-1,2j))\in \mathtt{E}_n^{\mathrm{Az}}$ with $(2i,2j+1) \in \mathtt{B}_n^{\mathrm{Az}}$, then there is an edge $((2i,2j+1),(2i,2j-1))$ in $\mathtt{E}_n^{\mathrm{Az},\mathrm{DR}}$. 
\end{itemize}
The edge weights of the dimers transfer directly to the edges associated to the lattice paths.  Recall that the edge weight of $((2i,2j+1),(2i-1,2j+2))\in \mathtt{E}_n^{\mathrm{Az}}$ for $(2i,2j+1) \in \mathtt{B}_n^{\mathrm{Az}}$ is equal to 1 and so we conclude that each weighted dimer covering is in one-to-one correspondence with each weighted lattice path configuration.
\begin{center}
\begin{figure}
    \includegraphics[height=5cm]{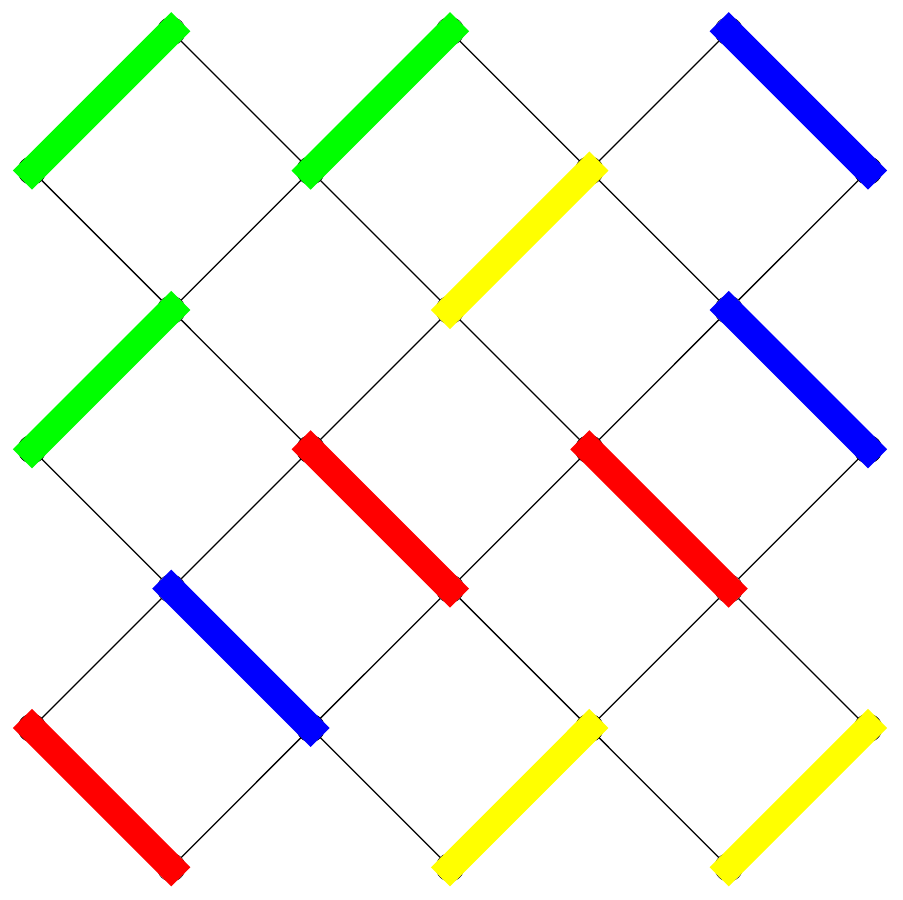}
\includegraphics[height=5cm]{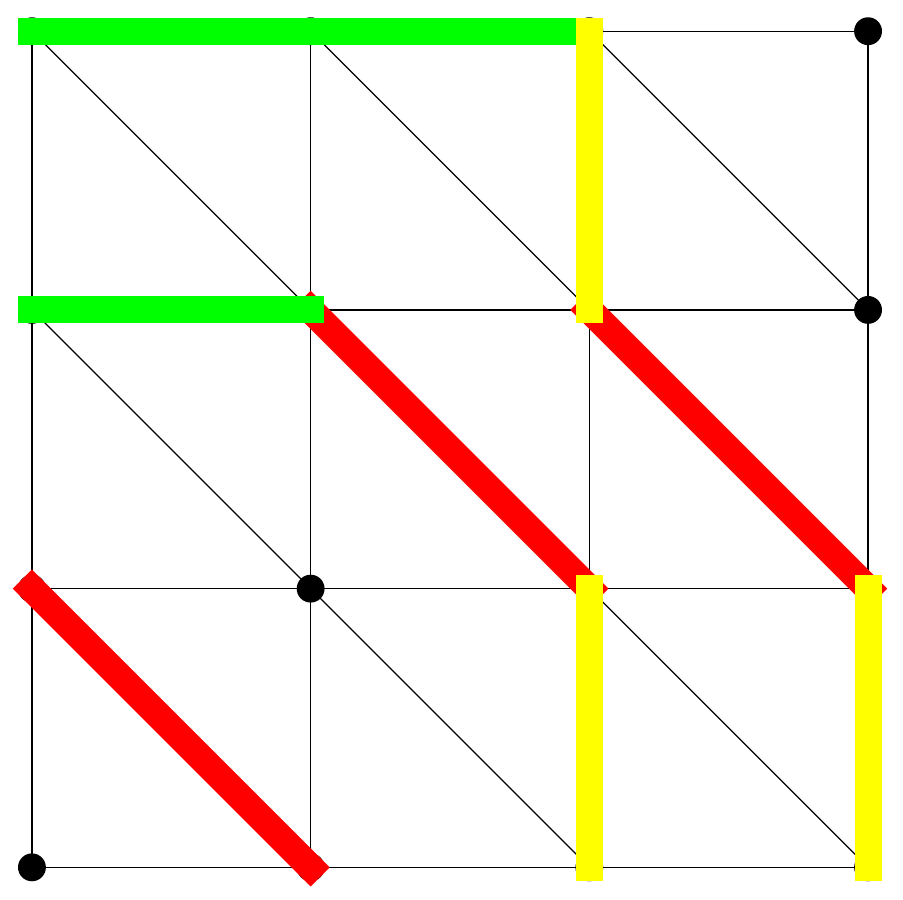}
    \caption{A dimer covering of an Aztec diamond of size 3 and the corresponding non-intersecting lattice paths}
    \label{fig:Aztec3dimerspath}
\end{figure}
\end{center}

\subsection{Tower Aztec diamond DR paths} \label{sec:DRtower}
As with the Aztec diamond case, we can also associate DR-lattice paths to the tower Aztec diamond.  The vertex set for the DR-lattice paths for the tower Aztec diamond of size $n$ and corridor of size $p$ is given by
\begin{equation}
\begin{split}
\mathtt{V}_{n,p}^{\mathrm{Tow},\mathrm{DR}}&=\{(2j,2k-1):0 \leq j \leq n , -p \leq k \leq n\\& \,\, \mbox{ or } 1 \leq j \leq n, -p-n+1\leq k \leq -p-1\},
\end{split}
\end{equation}
and the edge set 
\begin{equation}
\begin{split}
 &   \mathtt{E}_{n,p}^{\mathrm{Tow},\mathrm{DR}}=
\{((2j,2k-1),(2j+2,2k-1)) :0 \leq j \leq n-1 , -p \leq k \leq n\} \\ 
& \cup\{((2j,2k-1),(2j,2k-3)) :0 \leq j \leq n , -p+1 \leq k \leq n\} \\ 
&\cup \{((2j,2k-1),(2j+2,2k-3)) :0 \leq j \leq n-1 ,-p+ 1 \leq k \leq n\}\\ 
&\cup\{((2j,2k-1),(2j+2,2k-1)) :1 \leq j \leq n-1 , -n-p \leq k \leq -p-1\} \\ 
& \cup\{((2j,2k-1),(2j,2k-3)) :1 \leq j \leq n , -n-p+1 \leq k \leq -p-1\} \\ 
&\cup \{((2j,2k-1),(2j+2,2k-3)) :1 \leq j \leq n-1 ,-n-p+ 1 \leq k \leq -p-1\}.\\ 
\end{split}
\end{equation}
Let $G_{n,p}^{\mathrm{Tow},\mathrm{DR}}=(V_{n,p}^{\mathrm{Tow},\mathrm{DR}},E_{n,p}^{\mathrm{Tow},\mathrm{DR}})$ and label this graph to be the DR graph for the tower Aztec diamond. 
The lattice paths on this DR graph start at the vertices $\{(0,2k-1), -p+1 \leq k \leq n\}$ and end at $\{(2n,-1-2j), 0 \leq j \leq n+p-1\}$. The same correspondence between paths and dimers for the Aztec diamond  holds for the tower Aztec diamond; see Fig.~\ref{fig:AztecTowlabels}.

\section{Preliminaries on non-intersecting paths}\label{sec:nonintersecting}

We will now recall the  model of non-intersecting paths that is equivalent to the dimer model for the Aztec diamond.

We start with a directed graph $\mathcal G=(\mathtt V,\mathtt E)$, with  the vertex set $\mathtt V=\mathbb Z^2$ and (directed) edges
\begin{multline*}
 \mathtt E= \left\{  ((i,j), (i+1,j)) \mid i,j \in \mathbb Z\right\}
 \cup \left\{  ((2i,j),(2i+1,j-1)) \mid  i,j \in \mathbb Z\right\}\\
 \cup  \left\{  ((2i,j),(2i,j-1)) \mid  i,j \in \mathbb Z\right\}.
\end{multline*}
See also Fig.~\ref{fig:graphLG}. This graph can be thought of as gluing two types of strips in an alternating fashion. The two types of strips can be found in Fig.~\ref{fig:twotype}. Both strips consist of two columns of vertices. One strip has horizontal and diagonal edges all directed from left to the right. The other strip does not have the diagonal edges but instead has edges pointing down between consecutive vertices in the right column.  We then   cover  $\mathbb Z ^2$ by putting copies of the left strip in Fig.~\ref{fig:twotype} such that the  horizontal coordinates  of the vertices in the left column of the strip are even and  putting copies of the other strip such that the  horizontal coordinates of the left column of vertices are odd. 
\begin{figure}[t]
	\includegraphics[height=6cm]{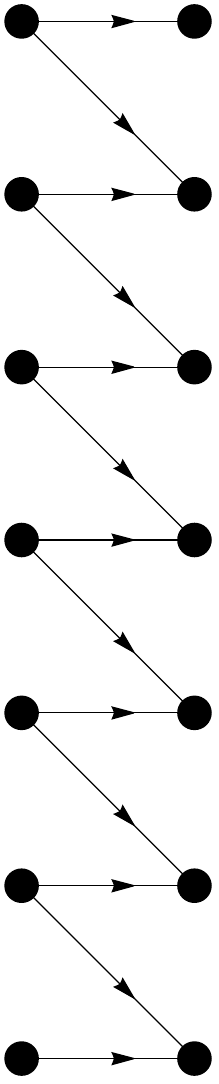}  \hspace*{1cm}
	\includegraphics[height=6cm]{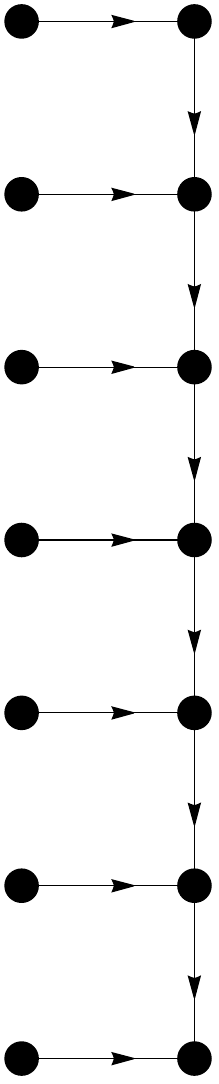}
	\caption{The figure shows the two strips that are the building blocks for the graph $\mathcal G$. Each strip has infinite length and consists of two vertical columns of vertices. The left strip has horizontal and diagonal edges directed from left to right. The right strip has downward edges between consecutive vertices in the right column instead of diagonal edges.} \label{fig:twotype}
\end{figure}

Next we introduce a function  $w:\mathtt V \to [0,\infty)$ that puts weights on all edges as follows
\begin{align*}
    &w((2i,j),(2i+1,j))=a_{i,j}\\
    &w((2i,j),(2i+1,j-1))=b_{i,j}\\
    &w((2i,j),(2i,j-1))=w((2i+1,j),(2i+2,j))=1, 
\end{align*}
for $i,j \in \mathbb Z$.
So only the edges $((2i,j),(2i+1,j))$ and $((2i,j),(2i+1,j-1))$ may have weights different from $1$. 

Based on the edge weights we also assign weights to the faces of the graph. Let $F_{2i,j}$ be the trapezoidal face defined by the vertices $(2i,j)$,  $(2i+2,j)$, $(2i+2,j-1)$ and $(2i+1,j-1)$.   Similarly, let $F_{2i+1,j}$ be the triangular face defined by the vertices $(2i+2,j)$, $(2i+2,j+1)$ and $(2i+3,j)$. We then define the weights of the faces by 
\begin{equation} \label{eq:evenfaceweightpaths}
    F_{2i,j}=\frac{a_{i,j}}{b_{i,j}}
\end{equation}
and 
\begin{equation} \label{eq:oddfaceweightpaths}
    F_{2i+1,j}=\frac{b_{i+1,j+1}}{a_{i+1,j}}.
\end{equation}
In Section~\ref{sec:dynamics} we will use the infinite graph $\mathcal G$, but for the connection with the Aztec diamond we will only need the subgraph that is obtained by gluing strips of the type in Fig.~\ref{fig:twotype} on  $\{0,\ldots,2n\} \times \mathbb Z$ in the same way as before, starting with a copy of the  type on the left of Fig.~\ref{fig:twotype}.

 {
The correspondence between dimers on the tower Aztec diamond graph and the non-intersecting paths on $\mathcal G$ is obtained by inserting trivial horizontal parts after every step in the DR-paths of Section~\ref{sec:DRtower} (albeit trivial, these parts help when applying the LGV-Theorem below, see also \cite{Jo03}).   More precisely,   for a given dimer configuration, construct a collection of paths in the following way:
\begin{itemize}
    \item if there is a dimer covering the edge $((2i,2j+1),(2i+1,2j))\in \mathtt{E}_n^{\mathrm{Tow,Az}}$ with $(2i,2j+1) \in \mathtt{B}_n^{\mathrm{Tow,Az}}$, then select the edges $((2i,j),(2i+1,j-1)), ((2i+1,j-1),(2i+2,j-1))$ in $\mathtt{E}$;
    \item if there is a dimer covering the edge $((2i,2j+1),(2i+1,2j+2))\in \mathtt{E}_n^{\mathrm{Tow,Az}}$ with $(2i,2j+1) \in \mathtt{B}_n^{\mathrm{Tow,Az}}$, then select  the edges $((2i,j),(2i+1,j)),((2i+1,j),(2i+2,j))$ in $\mathtt{E}$; 
    \item if there is a dimer covering the edge $((2i,2j+1),(2i-1,2j))\in \mathtt{E}_n^{\mathrm{Tow,Az}}$ with $(2i,2j+1) \in \mathtt{B}_n^{\mathrm{Az}}$, then select the edge $((2i,j),(2i,j-1))$ in $\mathtt{E}$. 
\end{itemize}
The selected edges in $E$ form non-intersecting paths  $(\pi_1,\ldots,\pi_{n+p})$ in $\mathcal G$   such that 
\begin{enumerate}
   \item for each $j$, the path $\pi_j$ starts in $(0,n-j)$ and ends in $(2n,-j)$,
   \item the paths are non-intersecting, i.e. $\pi_j \cap \pi_k = \varnothing$ for $j \neq k$. 
\end{enumerate}  See Fig.\ref{fig:LGVgraph2} for an example. 
}

Let us denote the set of all such collections of non-intersecting paths by $\Pi_{ni}$. Then we can define a probability measure on $ \Pi_{ni}$ by setting the probability of a given collection  $(\pi_1,\ldots,\pi_{n+p})$ to be proportional to 

 \begin{equation} \label{eq:nonintersectingpaths}
\mathbb P\left((\pi_1,\ldots,\pi_N)\right) \sim \prod_{j=1}^n \prod_{e \in \pi_j} w(e)
 \end{equation}
where $w$ is the weight function on the edges.

For each $i \in \mathbb Z$ we define a $\mathbb Z\times \mathbb Z$ matrix $M_i(j,k)$ and refer to these as transition matrices.   For even indices  the matrices $M_i$  are only non-zero on the diagonal and the subdiagonal right below the main diagonal, and have the values
\begin{equation}\label{eq:eventransitionmatrices}
M_{2i}(j,k)= \begin{cases}
    a_{i,j}, & k=j,\\
    b_{i,j}, & k=j-1,\\
    0, & \textrm{ otherwise.}
\end{cases}
\end{equation}
 For odd indices the matrices $M_i$ are lower triangular, with values
 \begin{equation}\label{eq:oddtransitionmatrices}
 M_{2i+1}(j,k)= \begin{cases}
    1& k  \leq j,\\
     0, & k >j.
 \end{cases}
\end{equation}
Note that the transition matrices $M_{2i}$ and $M_{2i+1}$ correspond to the left and right columns respectively in Fig.~\ref{fig:twotype}.

\begin{figure}[t]
    \begin{center}
    \includegraphics[height=8cm]{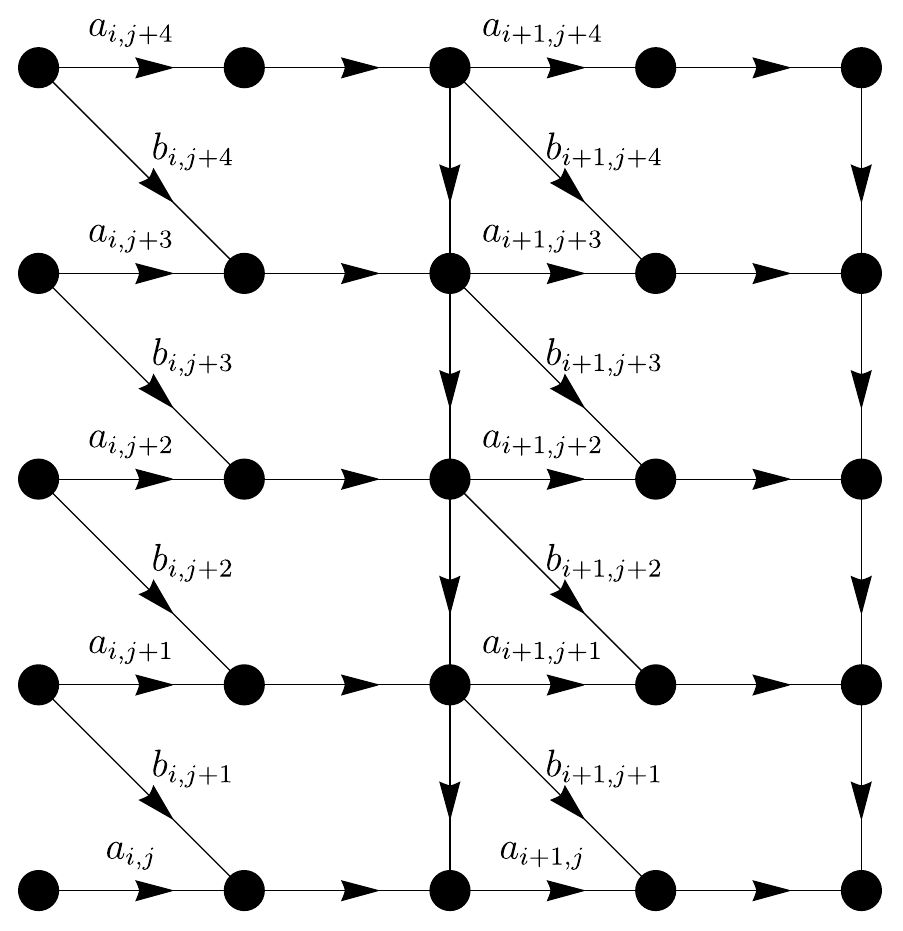}
    \caption{Part of the LGV graphs along with the edge weights.  The unmarked edges all have weight 1.} \label{fig:graphLG}
\end{center}
\end{figure}

\begin{figure}[t]
    \begin{center}
\includegraphics[height=8cm]{AztecTow32paths.pdf}
    \includegraphics[height=8cm]{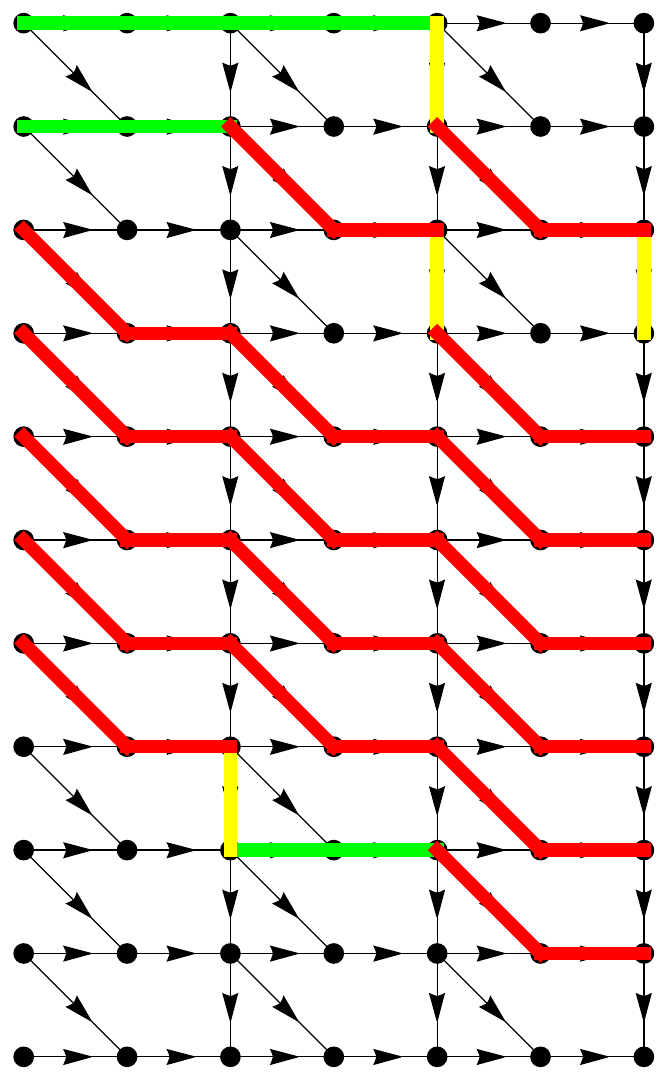}
    \caption{ The DR paths along with the non-intersecting paths. }
    \label{fig:LGVgraph2}
\end{center}
\end{figure}
{The random configuration of paths induces a natural  point process $\{(i,x_j^i)\}_{i=0,j=1}^{2n,n+p}$ where $(i,x_j^i)$ is the lowest vertex in $\pi_j$ at the vertical section with horizontal coordinate $i$.  By applying a celebrated theorem of Lindström-Gessel-Viennot~\cite{GesVie85,Lin73}, which we will henceforth simply refer to as the LGV Theorem, (for completeness, we included this theorem in Appendix~\ref{sec:app:lgv}) the point process has the probability distribution  proportional to a product of determinants
$$
 \prod_{i=0}^{2n-1} \det \left(M_{i}(x_j^i,x_k^{i+1})\right)_{j,k=1}^{n+p}.
$$
The Eynard-Mehta Theorem~\cite{EM97} then tells us that  the process is a determinantal point process:}

\begin{thm}\label{thm:em}\cite{EM97}
Let $W$ be the matrix 
\begin{equation} \label{eq:defWEM}
    W_{rs}= \left( M_{0}\cdots M_{2n-1}\right)(n-r,-s), \qquad r,s=1,\ldots,n+p,
\end{equation}
and set 
\begin{multline} \label{eq:kernelEM}
K(m_1,x_1,m_2,x_2)=- \mathbbm{I}[m_1 >m_2] M_{m_2+1}\cdots M_{m_1}(x_1,x_2)\\
+ \sum_{r,s=1}^{n+p}  (M_{m_1}\cdots M_{2n-1})(x_1,-s) (W^{-1})_{s,r}(M_{0}\cdots M_{m_2-1})(n-r,x_2).
\end{multline}
Then for any $(m_\ell,x_\ell)$ for $\ell=1, \ldots,M$ we have 
\begin{multline}
\mathbb P\left(\text{ paths go through } (m_\ell,x_\ell) \text{ for }\ell=1, \ldots,M\right)\\
= \det \left(K(m_\ell,x_\ell,m_k,x_k)\right)_{\ell,k=1}^M.
\end{multline}
\end{thm}

{The power of this result is that studying the asymptotic behavior of the point process, as $n \to \infty$, boils down to studying the limiting behavior of the kernel $K$. Obviously, a major obstacle remains: the expression for the kernel~\eqref{eq:kernelEM} involves the inverse of the matrix $W$ in~\eqref{eq:defWEM} that is growing in size. In the literature, one typically restricts to  special situations, in which more workable expressions for the inverse can be found. In Section~\ref{sec:refactorizaion} we will show how, with only very minor restriction on the weights, the inverse can be computed using a dynamical system defined by refactorizing the matrices $M_i$, giving an LU- and UL-decomposition of the doubly infinite matrix $M_0 \cdots M_{2n-1}$ (corrected by a shift matrix). Since we are after the inverse of a submatrix, these decompositions do not immediately provide the inverse of $W$. Here the auxiliary variable $p$ comes to the rescue: from these LU and UL-decompositions one can construct an approximate inverse for $p$ large enough, and, by taking the limit $p\to \infty$,  this can be used to find an expression for $K$.}

{But before we explain this inverse, we first return to the approach with the inverse Kasteleyn matrix and show that there is an analogue of Theorem~\ref{thm:em} for the inverse Kasteleyn matrix that involves the inverse of the same matrix $W$ in~\eqref{eq:defWEM} (up to a trivial sign change in the entries).  }

\section{Relation between the Kasteleyn Approach and the LGV theorem}

In this section, we show a relation between the Kasteleyn matrix and the DR-lattice paths via the LGV theorem.  This also gives a relation for computing the inverse of the Kasteleyn matrix using the inverse of the (complete) LGV matrix.  We prove our result for the Aztec diamond and the tower Aztec diamond graph.  This result holds more generally, but we specialize to the case we are interested in this paper.  

\subsection{The relation for the Aztec diamond graph}
\label{subsec:KasttoLGVAz}

Let $W_n^{\mathrm{Az}}=(w_{ij})_{1 \leq i,j \leq n}$ with $w_{ij}$ equal to the number of weighted paths from $(0,2i-1)$ to $(2j,-1)$ for $1 \leq i,j \leq n$ on $G_n^{\mathrm{Az},\mathrm{DR}}$.  The LGV Theorem, Theorem~\ref{thm:lgv}  asserts that  the number of non-intersecting weighted lattice paths on $G_n^{\mathrm{Az},\mathrm{DR}}$ whose start points are given by $\{(0,2k-1), 0 \leq k \leq n\}$ and end at $\{(2j,-1), 0 \leq j \leq n\}$ is equal to  $\det W_n^{\mathrm{Az}}$.  Due to the one-to-one correspondence with dimer coverings, $\det W_n^{\mathrm{Az}}$ is also equal to the weighted number of dimer coverings on $G_n^{\mathrm{Az}}$.   

For a matrix $M$, denote $M[i;j,k;l]$ to be the submatrix of $M$ restricted to rows $i$ through to $j$ and columns $k$ through to $l$.

 \begin{thm}\label{thm:AztecLGV}
Let $A_n=K_n[1;n,1;n]$, $B_n=K_n[1;n,n+1;n(n+1)]$,$C_n=K_n[n+1;n(n+1),1;n]$ and $D_n=K_n[n+1;n(n+1),n+1;n(n+1)]$. 
    For $1 \leq i,j \leq n$, let
 $$
     \tilde{w}_{ij}= (A_n-B_n D_n^{-1}C_n)(i,j)
    $$
and $\tilde{W}_n^{\mathrm{Az}}=(\tilde{w}_{ij})_{1 \leq i,j \leq n}$.
Then, we have that $w_{ij}=| \tilde{w}_{ij}|$ for $1 \leq i,j \leq n$ and $\det W_n^{\mathrm{Az}}=|\det \tilde{W}_n^{\mathrm{Az}}|$.  
    We also have that 
    \begin{equation}\label{thm:eqAztecLGV}
        K_n^{-1} =\left( \begin{matrix}
            (\tilde{W}_n^{\mathrm{Az}})^{-1} & - (\tilde{W}_n^{\mathrm{Az}})^{-1}B_n D_n^{-1} \\
            -D_n^{-1}C_n (\tilde{W}_n^{\mathrm{Az}})^{-1}& D_n^{-1} + D_n^{-1} C_n (\tilde{W}_n^{\mathrm{Az}})^{-1} B_n D_n^{-1} 
        \end{matrix}\right)
     \end{equation}
\end{thm}
{
	
\begin{rem}
\begin{enumerate}
\item {A similar assertion for the first statement has been made for the square-grid on a cylinder; see~\cite{AGR:21}[Section 4]}.
	\item The signs for the entries in  $(\tilde{W}_n^{\mathrm{Az}})^{-1}$ can be computed explicitly. In fact, by~\cite{CY14}[Lemma 3.6], we have that $\tilde{w}_{i,j}=\mathrm{i}^{i+j-1}w_{i,j}$ for $1 \leq i,j\leq n$. We omit this computation. 
	\item {Here, $D_n$ is a triangular matrix while $B_n$ and $C_n$ are very sparse matrices.  }
\end{enumerate}
	\end{rem}
		}

\begin{proof}
    We first show that $|\det D_n|=1$. Notice that $D_n$ is the Kasteleyn matrix of removing the vertices $\{(2k-1,0):1\leq k \leq n \} \in \mathtt{W}_n^{\mathrm{Az}}$ and  $\{(0,2k-1):1\leq k \leq n \} \in \mathtt{B}_n^{\mathrm{Az}}$ and their incident edges from $G_n^{\mathrm{Az}}$, that is removing the bottom row and leftmost column of vertices and their incident edges from $G_n^{\mathrm{Az}}$.  This is tile-able and so $\det D_n$ is non-zero.  Moreover, it is easy to see that there is in fact exactly one dimer configuration on this graph; see Fig.~\ref{fig:Aztec3frozen} for an example.  This configuration is precisely all dimers of the form  $((i,j),(i-1,j+1))\in \mathtt{E}_n^{\mathrm{Az}}$ for $(i,j) \in \mathtt{B}_n^{\mathrm{Az}}\backslash \{(0,2k+1):0\leq k \leq n-1
\}$, which all have weight 1.
\begin{center}
\begin{figure}
    \includegraphics[height=5cm]{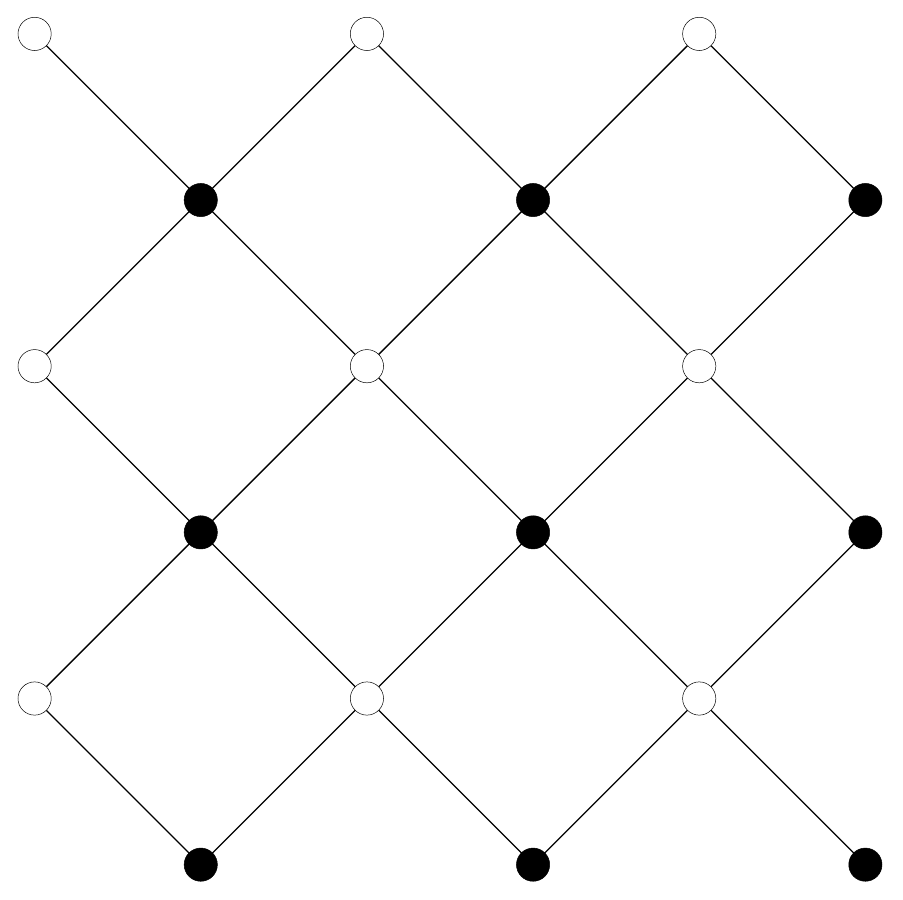}
  \includegraphics[height=5cm]{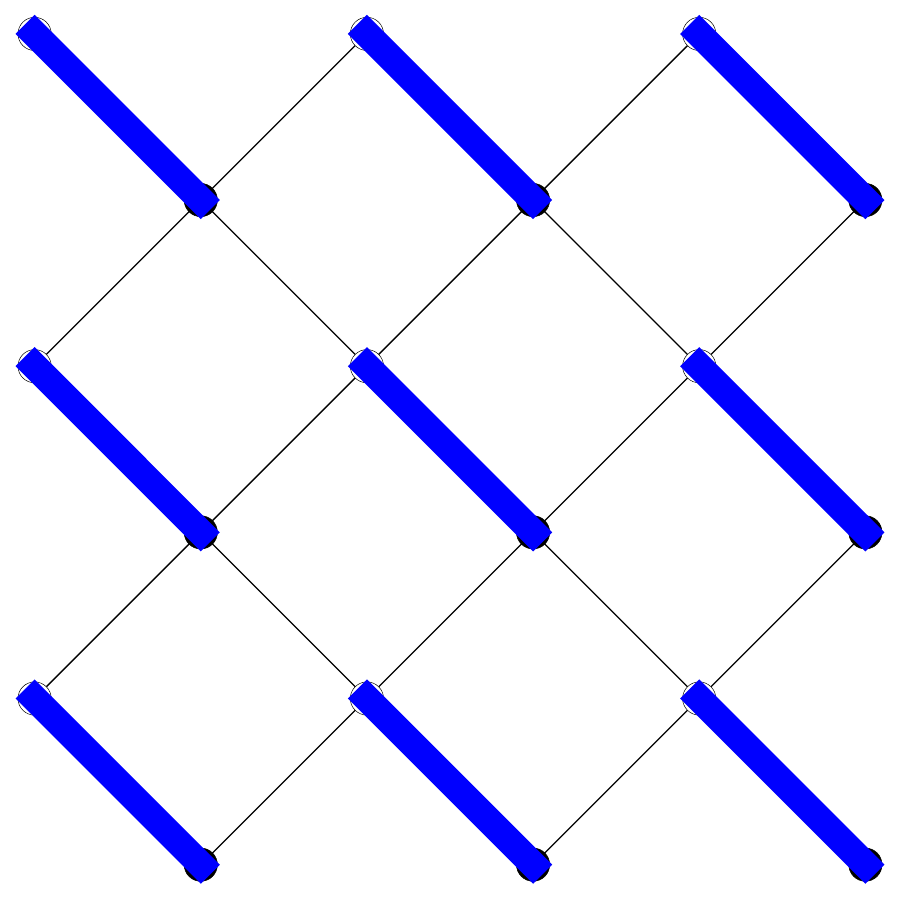}
    \caption{The left figure shows removing the leftmost column and bottom row of vertices from an Aztec diamond of size 3 and the right figure shows the single possible dimer configuration.  }
    \label{fig:Aztec3frozen}
\end{figure}
\end{center}
Kasteleyn's theorem, the formula for determinants of 2 by 2 block matrices and the evaluation of $\det D_n$ above give
\begin{equation}
\begin{split}
&\det W_n^{\mathrm{Az}}=|\det K_n|= |\det( A_n - B_n D_n^{-1} C_n) \det D_n| \\
  &= | \det (A_n -B_n D_n^{-1} C_n)|=|\det \tilde{W}_n^{\mathrm{Az}}|.
\end{split}
\end{equation} 
    Observe that $(A_n)_{i,j}=\mathbbm{I}[i=j=1] b_{0,0}$.  Next, we expand out $B_n D_n^{-1} C_n$. We have that
\begin{equation}
\begin{split}
&(B_n D_n^{-1} C_n)_{i,j}=\sum_{k,l=1}^{n^2} ( B_n)_{i,k} (D_n^{-1})_{k,l} (C_n)_{l,j}\\ 
&=\sum_{\varepsilon_1 , \varepsilon_2 \in \{0,1\} } ( B_n)_{i,n^2+1-(i-\varepsilon_1)n} ( D_n^{-1})_{n^2+1-(i-\varepsilon_1)n,n^2-n+j-\varepsilon_2}
 ( C_n)_{n^2-n+j-\varepsilon_2,j} \\
& \,\,\, \times \mathbbm{I}[i-\varepsilon_1> 0] \mathbbm{I}[j-\varepsilon_2> 0] \\
&= \sum_{\varepsilon_1=i-1}^{i} \sum_{\varepsilon_2=j-1}^{j}
 ( B_n)_{i,n^2+1-\varepsilon_1n} ( D_n^{-1})_{n^2+1-\varepsilon_1n, n^2-n+\varepsilon_2}
 ( C_n)_{n^2-n+\varepsilon_2,j} \\
& \,\,\, \times\mathbbm{I}[\varepsilon_1> 0] \mathbbm{I}[\varepsilon_2> 0]
\end{split}
\end{equation}
where the penultimate line follows from only considering the non-zero entries of $B_n$ and $C_n$, and the last line is just a rearrangement of the sum. 

The entries of $D_n^{-1}$ in the last line of the above formula are those on the boundary of the graph induced from the Kasteleyn matrix $D_n$.  Since these vertices are on the boundary, they represent, up to sign, the ratio between the number of weighted dimer coverings 
on the graph induced by $D_n$ with these vertices removed and the number of weighted dimer coverings on the graph induced by $D_n$. 
 We next show that $\mathrm{sgn}( D_n^{-1})_{n^2+1-\varepsilon_1n, n^2-n+\varepsilon_2}=\mathrm{i}^{\varepsilon_1+\varepsilon_2+1}$. To see this, observe that $\mathrm{sgn}(\det D_n)=\mathrm{i}^{n^2}$, since there are $n^2$ edges each having weight $\mathrm{i}$ and the only configuration corresponds to the identity permutation in the expansion of the determinant.  It follows that $\mathrm{sgn}( D_n^{-1})_{n^2+1-n, n^2-n+1}=\mathrm{i}^{n^2-1}/\mathrm{i}^{n^2}=\mathrm{i}^{3}$ since the numerator corresponds to the sign of the relevant entry of the adjugate matrix.  We can then sequentially increase $\varepsilon_1$ which removes a factor of  $\mathrm{i}$ from the product of entries of the Kasteleyn matrix and flips a sign in the relevant entry of the adjugate matrix; similar for $\varepsilon_2$. 

We have that for $1 \leq i,j \leq n$, 
\begin{equation}
\begin{split}
\label{eq:AnBDC}
&(A_n)_{i,j} -\sum_{\varepsilon_1=i-1}^{i} \sum_{\varepsilon_2=j-1}^{j}
 ( B_n)_{i,n^2+1-\varepsilon_1n} ( D_n^{-1})_{n^2+1-\varepsilon_1n, n^2-n+\varepsilon_2}
 ( C_n)_{n^2-n+\varepsilon_2,j}  \\
& \,\,\, \times\mathbbm{I}[\varepsilon_1> 0]\mathbbm{I}[\varepsilon_2> 0]
\end{split}
 \end{equation}
is equal to, up to sign, the weighted number of dimer coverings on $G_n^{\mathrm{Az}}\backslash (\{(0,2k-1):k\not = i \} \cap  \{(2k-1,0):k\not = j \}$.  This indeed follows because the prefactor and postfactor multiplication by entries of $( B_n)$ and $(C_n)$ respectively are in fact edge weights due to their specific entries, and the signs from each of the terms combine in such a way that the expansion of each term has the same sign.  
\begin{center}
\begin{figure}
    \includegraphics[height=5cm]{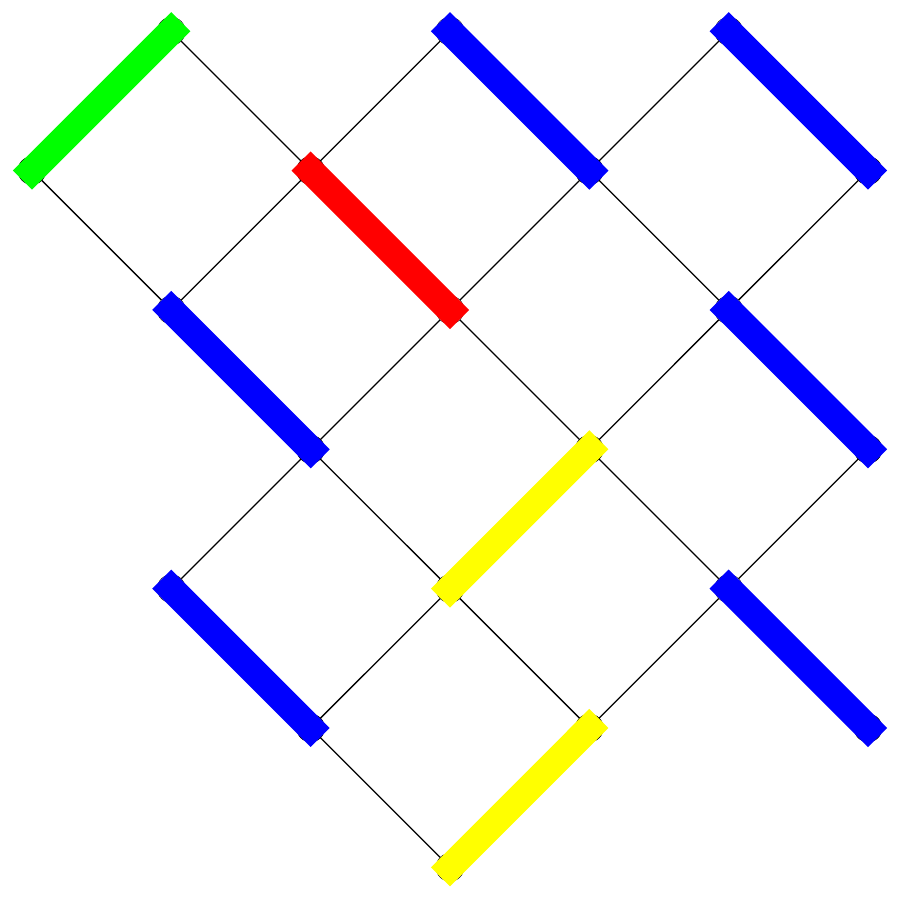}\,\,\,
  \includegraphics[height=5cm]{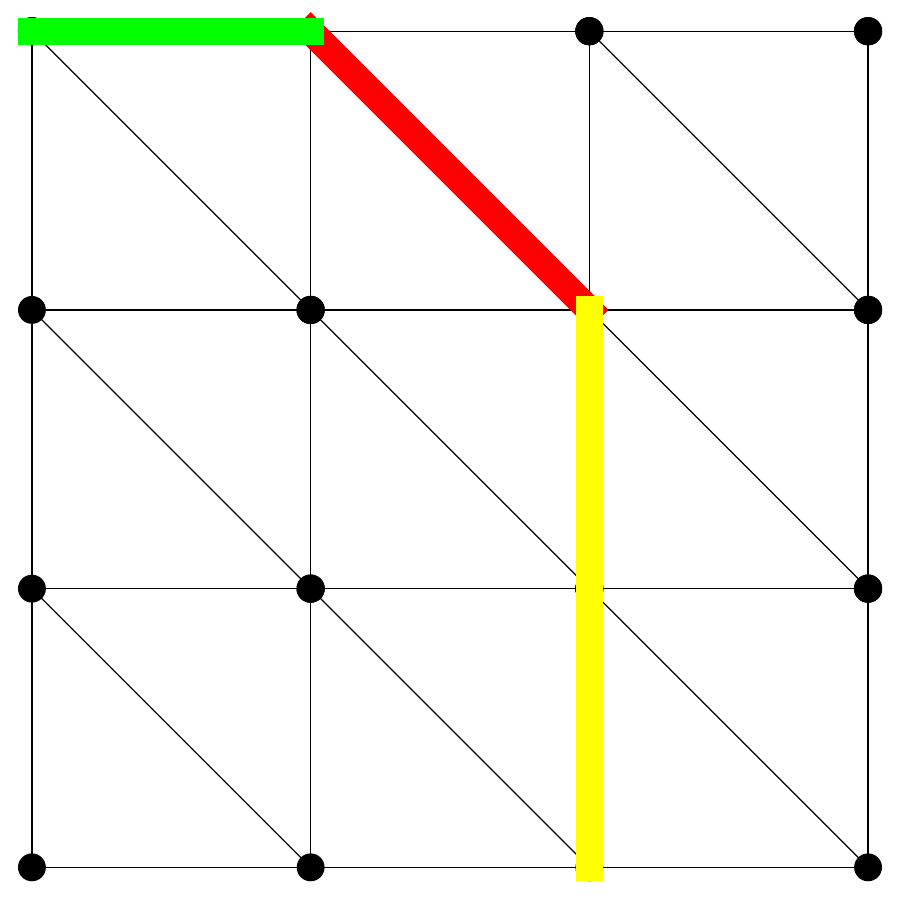}
    \caption{The left figure shows an example of a dimer covering on $G_3^{\mathrm{Az}}\backslash (\{(0,2k-1):k\not = 3 \} \cap  \{(2k-1,0):k\not = 2 \}$  }
    \label{fig:Aztec3frozenonepath}
\end{figure}
\end{center}

Next notice that the graph $G_n^{\mathrm{Az}}\backslash (\{(0,2k-1):k\not = i \} \cap  \{(2k-1,0):k\not = j \})$ induces  a single DR lattice path from $(0,2i-1)$ to $(2j,-1)$ on $G_n^{\mathrm{Az},\mathrm{DR}}$ for $1 \leq i,j\leq n$; see Fig.~\ref{fig:Aztec3frozenonepath} for an example.  Therefore, the above expression, up to sign, is also equal to the number of DR lattice paths from $(0,2i-1)$ to $(2j,-1)$ for $1 \leq i,j\leq n$.  

The final assertion is a consequence of the first assertion in the statement of the theorem and the Schur complement formula.

\end{proof}

\subsection{The relation for the tower Aztec diamond}

We next give an analogous theorem to \cref{thm:AztecLGV} for the tower Aztec diamond of size $n$ with corridor $p$.  

 \begin{thm}\label{thm:AztecTowerLGV}
Let $B_{n,p}=K_{n,p}[1;n+p,n+p+1;n(2n+p)]$,$C_{n,p}=K_{n,p}[n+p+1;n(2n+p),1;n+p]$ and $D_{n,p}=K_{n,p}[n+p+1;n(2n+p),n+p+1;n(2n+p)]$. 
    For $1 \leq i,j \leq n+p$, let
    $$
    \tilde{w}_{ij}= (-B_{n,p} D_{n,p}^{-1}C_{n,p})(i,j).
    $$
and $\tilde{W}_{n,p}^{\mathrm{Tow}}=(\tilde{w}_{ij})_{i,j=1}^{2n^2+np}$. Then, up to sign, $\tilde{W}_{n,p}^{\mathrm{Tow}}$ is equal to the LGV matrix for the tower Aztec diamond of size $n$ with corridor $p$ and 
\begin{equation}\label{thm:eqTowAzLGV0}
|\det \tilde{W}_{n,p}^{\mathrm{Tow}}(i,j)|_{1 \leq i,j \leq n+p} = |\det K_{n,p}(i,j)|_{1 \leq i,j \leq 2n^2+np}
\end{equation}

    We also have that 
    \begin{equation}\label{thm:eqTowAzLGV}
        K_{n,p}^{-1} =\left( \begin{matrix}
            (\tilde{W}_{n,p}^{\mathrm{Tow}})^{-1} & - (\tilde{W}_{n,p}^{\mathrm{Tow}})^{-1}B_{n,p} D_{n,p}^{-1} \\
            -D_{n,p}^{-1}C_{n,p} (\tilde{W}_{n,p}^{\mathrm{Tow}})^{-1}& D_{n,p}^{-1} + D_{n,p}^{-1} C_{n,p} (\tilde{W}_{n,p}^{\mathrm{Tow}})^{-1} B_{n,p} D_{n,p}^{-1} 
        \end{matrix}\right)
     \end{equation}
\end{thm}

\begin{rem}\label{rem:signs}
	\begin{enumerate}
	\item The signs for the entries in  $(\tilde{W}_{n,p}^{\mathrm{Tow}})^{-1}$ can be computed explicitly. Following the computation given in~\cite{CY14}[Lemma 3.6], we have that $\mathrm{sgn}\tilde{w}_{ij}=\mathrm{i}^{i+3j+1}(-1)^n$ for $1 \leq i,j\leq n+p$. We omit this computation. 
	\item The matrix $D_n$ is a triangular matrix and so its inverse can easily be computed while $B_n$ and $C_n$ are very sparse matrices. Thus, the complicated step to finding  a formula for the asymptotic inverse of the Kasteleyn matrix as $p$ tends to infinity is to find the  asymptotic inverse of $(\tilde{W}_{n,p}^{\mathrm{Tow}})^{-1}$ as $p$ tends to infinity.  We will discuss this in Section \ref{sec:refactorizaion}.  In certain special cases,  such as doubly periodic weights, we expect using our results that the asymptotic inverse of the Kasteleyn matrix entries are given by double contour integral formulas, but we will not work this out here.
	\end{enumerate}
\end{rem}
		
\begin{center}
    \begin{figure}[t]
        \includegraphics[height=7cm]{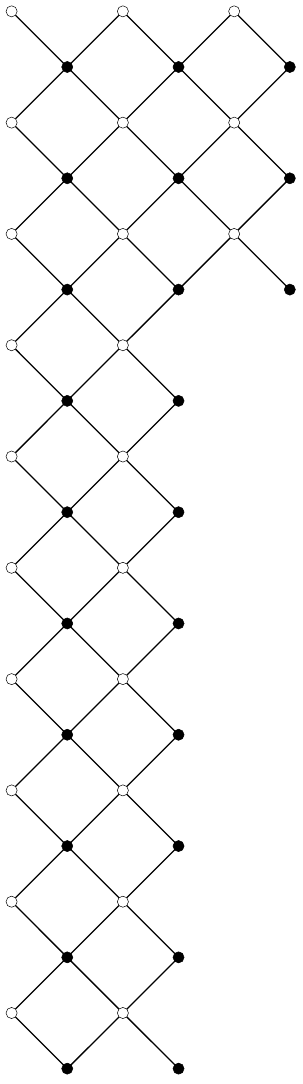}\,\,
      \includegraphics[height=7cm]{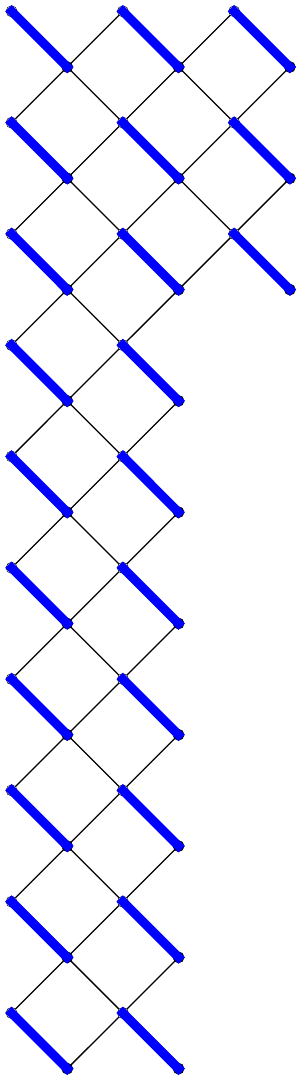}\,\,
    \includegraphics[height=7cm]{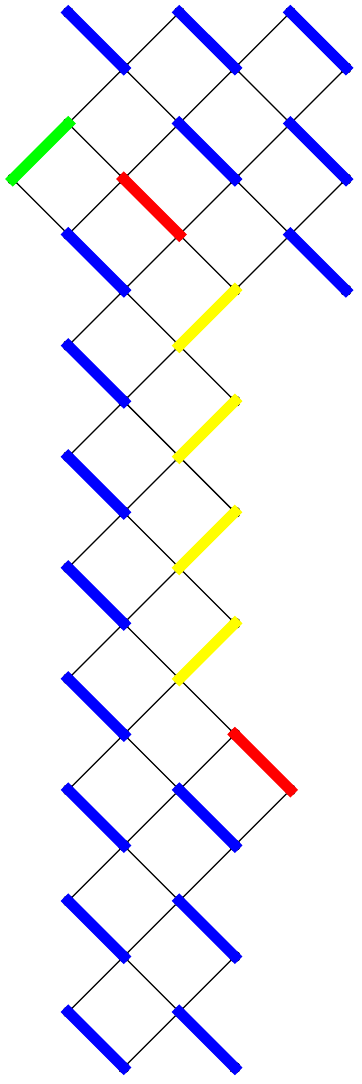}\,\,
    \includegraphics[height=7cm]{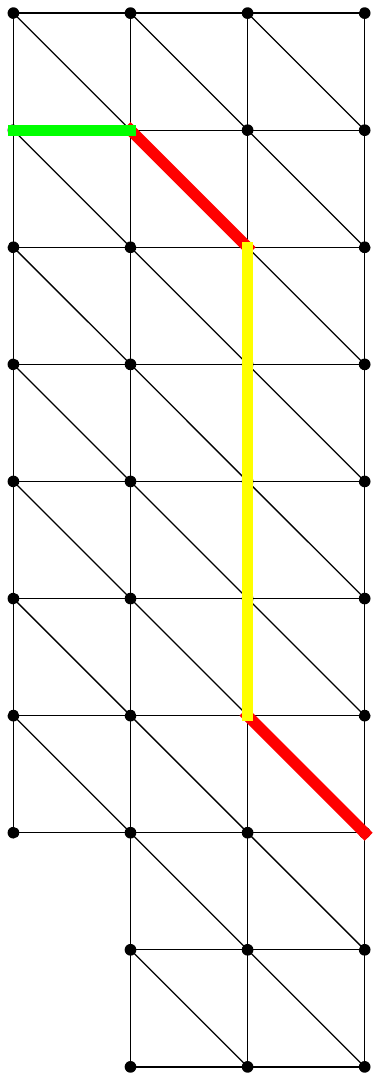}
        \caption{The two left figures show the graph induced by $D_{3,4}$ and the only possible dimer covering. The right two figures show a dimer covering on $G_{3,4}^{\mathrm{Tow}}\backslash (\{(0,2n+1-2k):k\not = 2 \} \cap  \{(2n-1,2-2k):k\not = 5 \}$ and the corresponding path $(0,3)$ to $(6,-7)$ on  $G_{3,4}^{\mathrm{Tow},\mathrm{LGV}}$.
    }
    \label{fig:AztecTowFrozen}
    \end{figure}
    \end{center}

\begin{proof}
The proof proceeds similar to the proof of \cref{thm:eqAztecLGV}. 

First observe that $D_{n,p}$ is the Kasteleyn matrix of removing the vertices $\{(2k,2n-1):-p-n+1\leq k \leq 0 \} \in \mathtt{W}_n^{\mathrm{Tow}}$ and  $\{(0,2k+1):-p\leq k \leq n-1 \} \in \mathtt{B}_{n,p}^{\mathrm{Tow}}$ and their incident edges from $G_{n,p}^{\mathrm{Tow}}$. This is tileable, has exactly one dimer configuration with all edges in the configuration having weight 1, and so $| \det D_{n,p}|=1$; see Fig.~\ref{fig:AztecTowFrozen} for an example.  
From the formula for determinants of 2 by 2 block matrices and that $D_{n,p}$ is invertible, we have 
\begin{equation}
\det K_{n,p} = \det(  - B_{n,p} D_{n,p}^{-1} C_{n,p}) \det D_{n,p} 
\end{equation} 
since $K_{n,p}[1;n+p,1;n+p]$ has all entries equal to 0.  This gives~\eqref{thm:eqTowAzLGV0}. 
 
The entries of $D_{n,p}^{-1}$ in the last line of the above formula are those on the boundary of the graph induced from the Kasteleyn matrix $D_{n,p}$.  Since these vertices are on the boundary, they represent, up to sign, the ratio between the number of weighted dimer coverings on the graph induced by $D_{n,p}$ with these vertices removed and the number of weighted dimer coverings on the graph induced by $D_{n,p}$. Similar to the proof of Theorem~\ref{thm:AztecLGV}, we can compute explicitly the sign of the boundary entries of $D_{n,p}^{-1}$ -- we omit this computation. 

We have that for $1 \leq i,j \leq n+p$, 
\begin{equation}
\begin{split}
&-\sum_{k,l=1}^{2n^2+pn-p-n}(B_{n,p})_{i,k} (D_{n,p}^{-1})_{k,l}(C_{n,p})_{l,j} \\
	&=-\sum_{\varepsilon_1=i-1}^i \sum_{\varepsilon_2=j-1}^j B_{i,n\varepsilon_1+1 (\varepsilon_1-n)\mathbbm{I}[\varepsilon_1>n]} D_{n\varepsilon_1+1 (\varepsilon_1-n)\mathbbm{I}[\varepsilon_1>n],n^2+(n-1)\varepsilon_2-\mathbbm{I}[\varepsilon_2=0]} \\ &\hspace{5mm}\times C_{n^2+(n-1)\varepsilon_2-\mathbbm{I}[\varepsilon_2=0],j}
\end{split}\label{eq:AnBDCTow}
\end{equation}
	where we have simplified by only recording the non-zero entries of $B_{n,p}$ and $D_{n,p}$.  This is equal, up to sign, to the number of weighted dimer coverings on $G_{n,p}^{\mathrm{Tow}}\backslash (\{(0,2n+1-2k):k\not = i \} \cap  \{(2n-1,2-2j):k\not = j \}$.  This indeed follows because the prefactor and postfactor multiplication by entries of $( B_{n,p})$ and $(C_{n,p})$ respectively are in fact edge weights due to their specific entries, and the signs from each of the terms combine in such a way that the expansion of each term has the same sign.

Next notice that the graph $G_{n,p}^{\mathrm{Tow}}\backslash (\{(0,2n+1-2k):k\not = i \} \cap  \{(2n-1,2-2j):k\not = j \}$ induces a single DR lattice path from $(0,2n-1-2i)$ to $(2n-1,1-2j)$ for  $1 \leq i,j\leq n+p$ on $G_{n,p}^{\mathrm{Tow},\mathrm{LGV}}$; see Fig.~\ref{fig:AztecTowFrozen} for an example.  Therefore, the above expression, up to sign, is also equal to the weighted number of DR lattice paths from $(0,2n-1-2i)$ to $(2n-1,1-2j)$.  

The final statement of the theorem follows directly from the Schur complement formula.

\end{proof}

\section{Equivalence of the domino shuffle and matrix refactorizations} \label{sec:dynamics}

\subsection{Domino shuffle for face weights}\label{subsec:shuffle}

Introduce the vertex sets
\[
\mathtt{V}_1=\{(i,j)\in \mathbb{Z}^2: i\mod 2=1, j\mod 2=0\} 
\]
and
\[
\mathtt{V}_2=\{(i,j)\in \mathbb{Z}^2: i\mod 2=0, j\mod 2=1\}.
\]
We initially set $\mathtt{W}=\mathtt{V}_1$ and $\mathtt{B}=\mathtt{V}_2$, that is the white and black vertices.

\begin{figure}[t]
    \begin{center}
  \includegraphics[height=3cm]{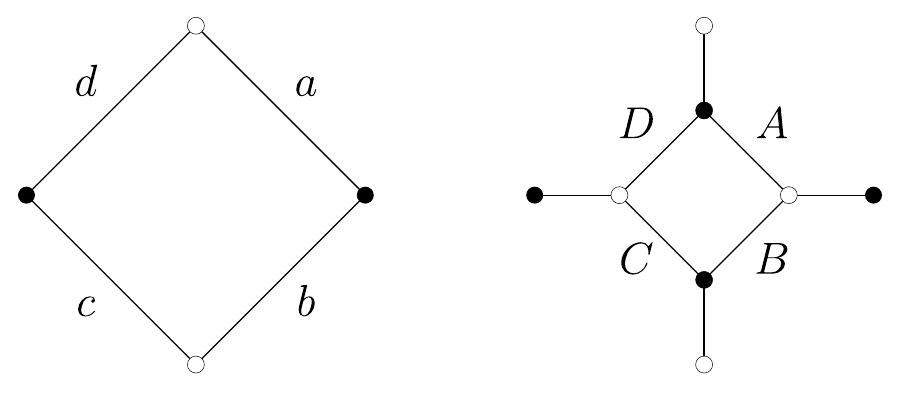}
    \caption{ The square move and its effect on the edge weights. The left figure shows a square with edge weights $a, b, c$ and $d$ while the right figure shows an application of the square move to that single face. 
	Here, we have $A=c/\Delta$, $B=d/\Delta$, $C=a/\Delta$, and $D=b/\Delta$ where
        $\Delta=(ac+bd)$.  }
\label{fig:Aztec4shufflesquare}
\end{center}
\end{figure}

We need the following two graph transformations.
\begin{enumerate}
    \item (Square Move) Suppose the edge weights around a square with vertices $(0,1),(1,0),(0,-1)$ and, $(-1,0)$ are given by $a,b,c$, and $d$ where the labelling is done clockwise around the face starting with the NE edge.  We can replace the square by a smaller square with edge weights $A,B,C$, and $D$ (with the same labelling convention) and add an edge, with edge-weight equal to 1, between each vertex of the smaller square and its original vertex.  Then, set $A=c/\Delta$, $B=d/\Delta$, $C=a/\Delta$, and $D=b/\Delta$ where
        $\Delta=(ac+bd)$. This transformation is called the \emph{square move}; see Fig. \ref{fig:Aztec4shufflesquare}.
    \item (Edge contraction) For any two-valent vertex in the graph with incident edges having weight 1, contract the two incident edges.  This is called \emph{edge contraction}.
\end{enumerate}
When we apply the above two moves to all the even (or odd) faces, we recover $\mathbb{Z}^2$ but with different face weights and the black and white vertices interchanged, that is after applying these two moves, we have that $\mathtt{W}=\mathtt{V}_2$ and $\mathtt{B}=\mathtt{V}_1$. To counter this interchanging of vertex colors, we translate the square grid by $(-1,1)$, that is, the face $(2i+1,2j+1)$ becomes the face $(2i,2j+2)$.  We call the application of the two moves above and the shift \emph{the domino shuffle}. To simplify conventions, we label the face $(2i+1,2j+1)$ to be \emph{even} faces.  Due to the shift, we only need to consider the domino shuffle applied to even faces; it is not hard to see that applying the square move twice to the same face gives the original graph and its original face weights.  

\begin{center}
\begin{figure}
\includegraphics[height=5cm]{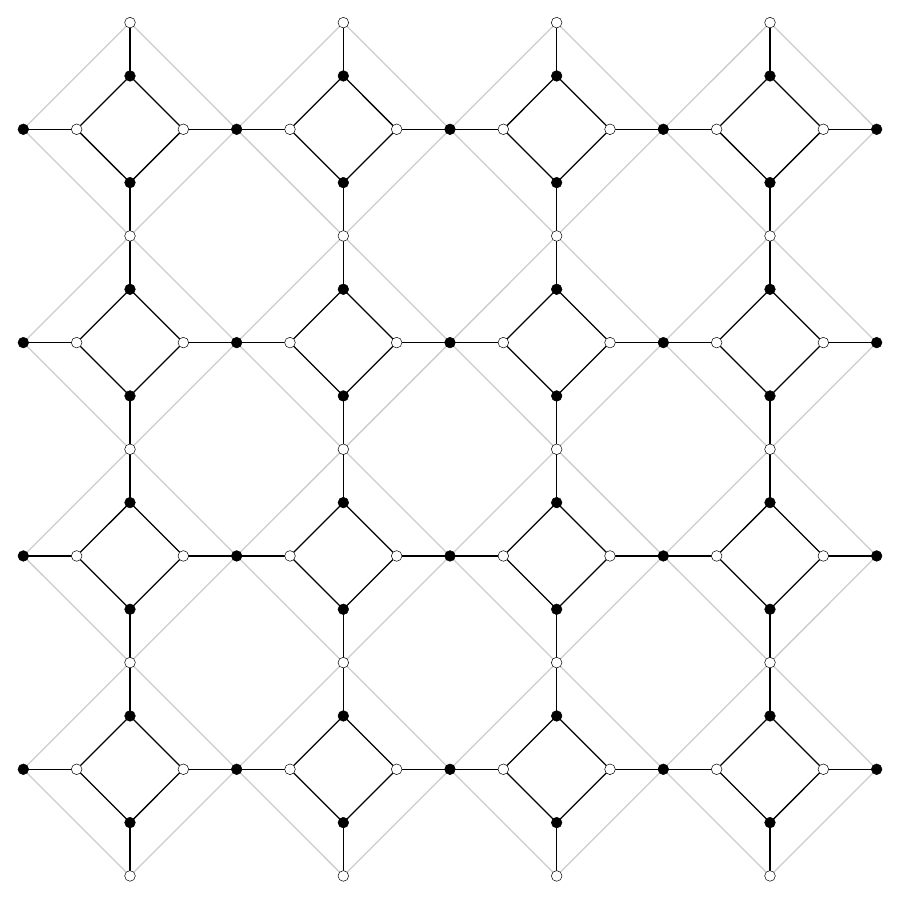}\,\,\,
    \includegraphics[height=5cm]{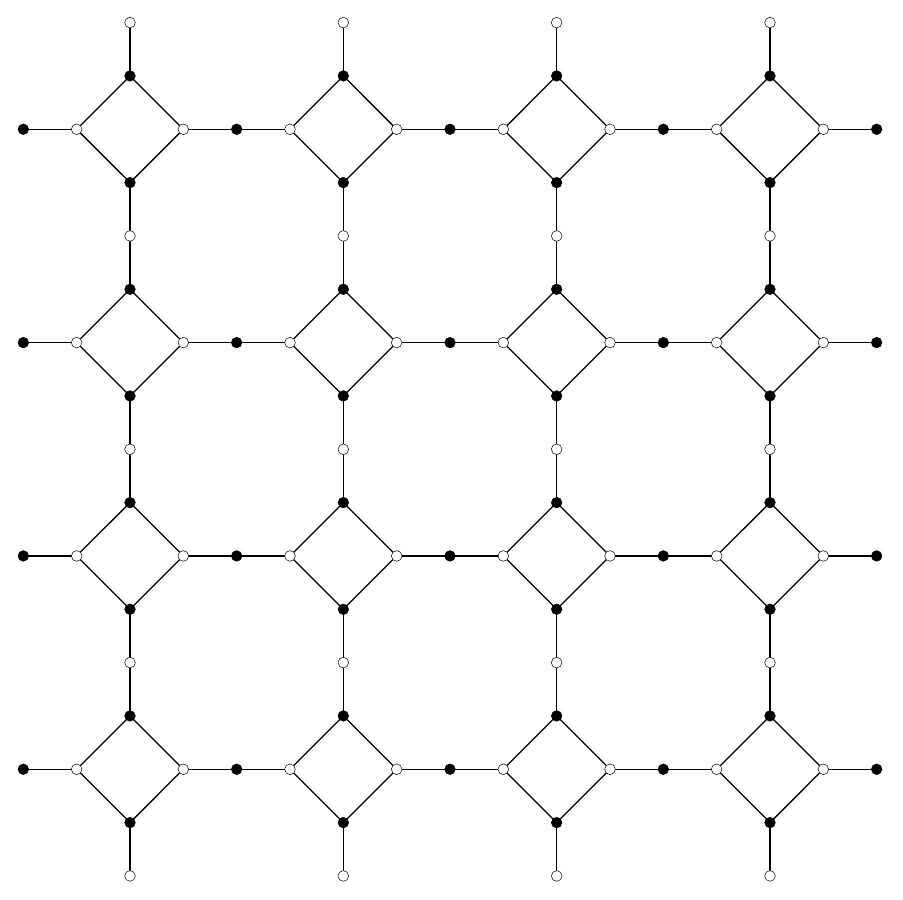}    
	\caption{An example of the graph transformation after applying the square move on all even faces of an Aztec diamond of size 4. The left figure shows where the square move applied. The figure on the right shows the actual graph.  Applying edge contraction to the two-valent vertices and removing pendant edges gives an Aztec diamond of size 3; see Section \ref{sec:inversesbyKas} for an explanation on pendant edges.  }
    \label{fig:Aztec4shuffle}
\end{figure}
    \end{center}

 \begin{center}
\begin{figure}
\includegraphics[height=3cm]{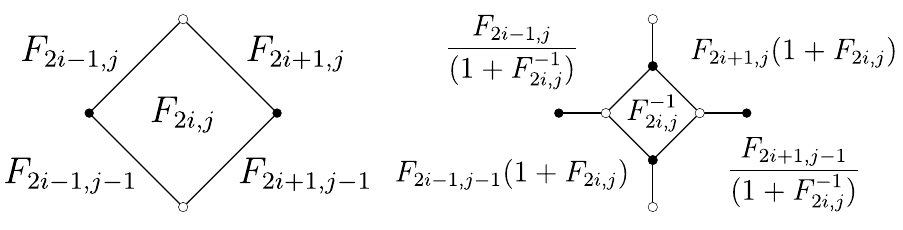}
    \caption{ The square move and its effect on the face weights.  The left figure shows the original  even face with center $(2i+1,2j+1)$, whose face weight is given by $F_{2i,j}$, and the neighboring face weights given by $ (F_{2i+1,j}, F_{2i+1,j-1}, F_{2i-1,j-1}, F_{2i-1,j})$ for the faces with centers $(2i+2,2j+2), (2i+2,2j), (2i,2j),$ and $(2i,2j+2)$ respectively. 
	}
\label{fig:Aztec4shufflesquare2}
\end{figure}
\end{center}
The next two propositions indicate the transformation of the face weights after the square move is applied to a single even face and when the domino shuffle is performed on all even faces. The first of which is well-known see e.g.~\cite{GK13}, but we include it in our presentation to keep the paper self-contained; see Fig. \ref{fig:Aztec4shufflesquare2}.  

\begin{prop}\label{prop:squaremoveonce}
    Consider the even face $(2i+1,2j+1)$ whose face weight is given by $F_{2i,j}$ and neighboring faces have face weights 
    \[
    (F_{2i+1,j}, F_{2i+1,j-1}, F_{2i-1,j-1}, F_{2i-1,j})
    \] 
    for the faces with centers $(2i+2,2j+2), (2i+2,2j), (2i,2j),$ and $(2i,2j+2)$ respectively. Applying the square move to the face $(2i+1,2j+1)$ changes the face weight at $(2i+1,2j+1)$ to $1/F_{2i,j}$ and the face weights of the neighboring faces to 
    \[\Bigg(F_{2i+1,j}(1+F_{2i,j}), \frac{F_{2i+1,j-1}}{1+F_{2i,j}^{-1}}, F_{2i-1,j-1}(1+F_{2i,j}),\frac{F_{2i-1,j}}{1+F_{2i,j}^{-1}}\Bigg)\] for the faces whose original centers are given by $(2i+2,2j+2), (2i+2,2j), (2i,2j),$ and $(2i,2j+2)$ respectively. 
\end{prop}
\begin{proof}
    Let the edge weights around the face whose center is $(2i+1,2j+1)$ be given by $a, b,c,$ and $d$ where the labelling is done clockwise around the face starting with the NE edge.  By definition of the face weights and the fact that $(2i+1,2j+1)$ is even, we have
    \[
        F_{2i,j}= \frac{bd}{ac}.
        \]
    After applying the square move to $(2i+1,2j+1)$, the small square has the opposite parity to the original square and so the new face at $(2i+1,2j+1)$ is given by
    \[
        \frac{  \frac{a}{\Delta} \frac{c}{\Delta} }{\frac{b}{\Delta} \frac{d}{\Delta} }=\frac{ac}{bd}=\frac{1}{F_{2i,j}},
    \]
where $\Delta=ac+bd$. 
    We now need to compute the face weights around the neighboring faces to the face whose center is at $(2i,j)$.   The face weight of the face whose center is at $(2i+2,2j+2)$ now has weight
    \[
        F_{2i+1,j} \cdot \frac{1}{a} \cdot \frac{1}{\frac{c}{\Delta}}= F_{2i+1,j} \frac{\Delta}{ac}=F_{2i+1,j} \bigg(1+\frac{bd}{ac}\bigg)=F_{2i+1,j}(1+F_{2i,j}).
    \]
	To see this equation, the first term on the left side is the original face weight, the second term is removing the contribution from $a$ from $F_{2i+1,j}$ while the third term is the weight of the new edge, which is oriented from white to black around the face (clockwise).  The computation for the new face weight at the face $(2i,2j)$ which has original face weight $F_{2i-1,j-1}$ is similar.  
The face weight of the face whose center is at $(2i+2,j)$ now has weight
    \[
        F_{2i+1,j-1} \cdot b \cdot {\frac{d}{\Delta}}= F_{2i+1,j-1} \frac{bd}{\Delta}=F_{2i+1,j-1} \frac{1}{1+\frac{ac}{bd}}=F_{2i+1,j-1}\frac{1}{1+F_{2i,j}^{-1}}.
    \]
To see this equation, the first term on the left side is the original face weight, the second term is removing the contribution from $b$ from $F_{2i+1,j-1}$ while the third term is the weight of the new edge, which is oriented from black to white around the face (clockwise).  The computation for the new face weight at the face $(2i,2j+2)$ which has original face weight $F_{2i-1,j}$ is similar. 
\end{proof}

We next consider the effect of applying the shuffle to all even faces. 

\begin{prop}\label{prop:squaremoveall}
	Let the face weights of the even faces whose centers are given by $(2i+1,2j+1)$ be equal to $F_{2i,j} > 0$ for all $i,j \in \mathbb{Z}^2$ and the face weights of the odd faces whose centers are given by $(2i+2,2j+2)$ be equal to $F_{2i+1,j} > 0$ for all $i,j \in \mathbb{Z}^2$.  Applying the domino shuffle to all the even faces of the graph, then 
    \begin{enumerate}
	    \item the face weight at $(2i+2,2j+2)$ is given by $1/F_{2i+2,j}$,   
	    \item the face weight at $(2i+1,2j+1)$ is given by 
		    $$ F_{2i+1,j-1} \frac{1+F_{2i+2,j}}{1+1/F_{2i+2,j-1}}\frac{1+F_{2i,j-1}}{1+1/F_{2i,j}}$$
    \end{enumerate}
	for all $i,j \in \mathbb{Z}$. 
\end{prop}
\begin{proof}
	The proposition follows by applying \cref{prop:squaremoveonce} to all even faces and noting the shift by $(-1,1)$ in our conventions of the domino shuffle.   

\end{proof}

\subsection{Dynamics on the transitions matrices}
Now let $\mathcal{G}$ be the weighted directed graph from Section~\ref{sec:nonintersecting}, i.e. the underlying graph for the non-intersecting path model~\eqref{eq:nonintersectingpaths}. We recall that the weights are determined by the transition matrices~\eqref{eq:eventransitionmatrices} and~\eqref{eq:oddtransitionmatrices}. We will now define a dynamics on the set of transition matrices, that is equivalent to the domino shuffle for the corresponding dimer model. The dynamics will be based on a commutation relation between transitions matrices, that we will discuss first. 

It will be convenient to use the notations
\begin{equation} \label{eq:defShift}
S_{ij}=\begin{cases}
    1, & i=j+1\\
    0, & \textrm{otherwise,}
\end{cases}
\end{equation}
and
\begin{equation}\label{eq:defpsi}
{\Psi}= I+S^1+\ldots=\sum_{k=0}^\infty S^k.
\end{equation}
Note that this series converges entrywise, but not in matrix norm. For two sequences $\mathtt a=(\mathtt a_j)_{j\in \mathbb Z}, \mathtt b =(\mathtt b_j)_{j\in \mathbb Z}$ of non-negative real numbers, we define 
$$
\Phi(\mathtt a,\mathtt b)=D(\mathtt a)+D(\mathtt b)S,
$$
where $D(\mathtt a)$ is the diagonal matrix $(D(\mathtt a))_{jj}= a_j$.

Then we can rewrite~\eqref{eq:eventransitionmatrices} as
$$
M_{2i}=\Phi(\mathtt a_i , \mathtt b_i)=D(\mathtt a_i)+D(\mathtt b_i)S,
$$
with $\mathtt a_i=(a_{i,j})_{j \in \mathbb Z},$
and~\eqref{eq:oddtransitionmatrices} as
$$
M_{2i+1}=\Psi.
$$
The following lemma is the main ingredient for the dynamics.
\begin{lem} Let $\sigma (\mathtt b)=(b_{i-1})_{i \in \mathbb Z}$. Then
   $$ \Phi(\mathtt a,\mathtt b) \Psi=D(\mathtt a+ \mathtt b)\Psi \Phi(\mathtt a,\sigma(\mathtt b)) D(\mathtt a+ \mathtt b)^{-1}.$$
\end{lem}
\begin{proof}
    First we write 
    \begin{multline}
        \Phi(\mathtt a,\mathtt b) \Psi=(D(\mathtt a)+D(\mathtt b)S)(I+S+\ldots )\\=D(\mathtt a )+D(\mathtt a+\mathtt b) \left(S+S^2+\ldots\right),
    \end{multline}
and then
$$
\Phi(\mathtt a,\mathtt b) \Psi=D(\mathtt a+\mathtt b) \left(D(\mathtt a)+\left(S+S^2+\ldots\right)D(\mathtt a+\mathtt b) \right)D(\mathtt a+\mathtt b)^{-1},
$$
which can then be turned into
\begin{multline}
\Phi(\mathtt a,\mathtt b) \Psi=D(\mathtt a+\mathtt b) \left(I+S+S^2+\ldots\right)(D(\mathtt a)+ S D(\mathtt b))D(\mathtt a+\mathtt b)^{-1}.
\end{multline}
Then using
\begin{multline}
D(\mathtt a) +S D(\mathtt b)=D(\mathtt a)+D(\sigma (\mathtt b))S,
\end{multline}
we  find 
\begin{multline}
    \Phi(\mathtt a,\mathtt b) \Psi
        =D(\mathtt a+\mathtt b) \left(I+S+S^2+\ldots\right)(D(\mathtt a)+  D(\sigma (\mathtt b))S)D(\mathtt a+\mathtt b)^{-1}\\
        =D(\mathtt a+ \mathtt b)\Psi \Phi(\mathtt a,\sigma(\mathtt b)) D(\mathtt a+ \mathtt b)^{-1},
\end{multline}
and we have proved the statement.
\end{proof}
Given transition matrices $M_i$ on the graph $\mathcal G$, it follows from the above lemma that 
\begin{equation}  \label{eq:baxter}
    M_{2i} M_{2i+1}=X_{i} \tilde M_{2i+1} \tilde M_{2i}X_{i}^{-1},\qquad i \in  \mathbb Z,
\end{equation}
where $\tilde M_{2i+1}=M_{2i+1}=\Psi$, $X_i=D(\mathtt a_i +\mathtt b_i)$, and $\tilde M_{2i}= \Phi(\mathtt a_i,\sigma(\mathtt b_i))$
We define a new weighting $\{\hat M_i\}_{i \in \mathbb Z}$ on the graph $\mathcal G$ such that the transition matrices are the following:
$$
\hat M_{2i}=\tilde M_{2i} X_{i}^{-1} X_{i+1},  
$$
and $\hat M_{2i+1}=\Psi$. Then~\eqref{eq:baxter} can be written as
\begin{equation}  \label{eq:baxter11}
    M_{2i} M_{2i+1}=X_{i} \hat M_{2i+1} \hat M_{2i}X_{i+1}^{-1},\qquad i \in  \mathbb Z.
\end{equation}
The parameters $(\hat{\mathtt a}_i)_{i \in \mathbb Z}$ and $(\hat{\mathtt b}_i)_{i \in \mathbb Z}$ for $\hat M_{2i}$ can be obtained from the parameters $(\mathtt a_i)_{i \in \mathbb Z}$ and  $(\mathtt b_i)_{i \in \mathbb Z}$ for $M_{2i}$, using the maps
\begin{equation}\label{eq:hata}
    \hat { a}_{i,j}=a_{i,j} \frac{ a_{i+1,j}+ b_{i+1,j}}{ a_{i,j}+ b_{i,j}},
\end{equation}
and 
\begin{equation} \label{eq:hatb}
    \hat { b}_{i,j}=  b_{i,j-1} \frac{ a_{i+1,j-1}+ b_{i+1,j-1}}{ a_{i,j-1}+ b_{i,j-1}}.
\end{equation}
The map 
$$\{M_i\}_{i}\mapsto \{\hat M_i\}_{i}$$
defines a discrete dynamical system that we will be interested in. In the following theorem we show how the face weights change under this map.

\begin{thm} \label{thm:dynamics1}
Under the map $\{M_i\}_{i\in \mathbb{Z}}\mapsto \{\hat M_i\}_{i\in \mathbb{Z}}$ the face weights change as:
%After the transformation the odd face weights become the reciprocal even face weights
$$
    \hat F_{2i+1,j}=1/F_{2i+2,j},
$$
%The even face weights change as follows:
$$
    \hat F_{2i,j}= F_{2i+1,j-1}  \frac{1+F_{2i+2,j}}{1+1/F_{2i+2,j-1}}\frac{1+F_{2i,j-1}}{1+1/F_{2i,j}}.
$$
\end{thm}
\begin{proof}[Proof of Theorem~\ref{thm:dynamics1}]
Substituting~\eqref{eq:hata} and~\eqref{eq:hatb} into the face weights~\eqref{eq:evenfaceweightpaths}  gives
    \begin{multline*}
    \hat F_{2i,j}= \frac{\hat { a}_{i,j}}{\hat { b}_{i,j}}=\frac{ a_{i,j}}{ b_{i,j-1}} \frac{ a_{i+1,j}+ b_{i+1,j}}{ a_{i,j}+ b_{i,j}} \frac{ a_{i,j-1}+ b_{i,j-1}}{ a_{i+1,j-1}+ b_{i+1,j-1}}\\
    = \frac{ a_{i,j}}{ b_{i,j-1}} \frac{ a_{i+1,j}+ b_{i+1,j}}{ a_{i+1,j-1}+ b_{i+1,j-1}} \frac{ a_{i,j-1}+ b_{i,j-1}}{ a_{i,j}+ b_{i,j}}.
    \end{multline*}
This can be written as 
$$
    \hat F_{2i,j}=\frac{ a_{i+1,j}+ b_{i+1,j}}{ a_{i+1,j-1}+ b_{i+1,j-1}} \frac{\frac{ a_{i,j-1}}{ b_{i,j-1}}+1}{1+\frac{ b_{i,j}}{ a_{i,j}}}= \frac{ a_{i+1,j}+ b_{i+1,j}}{ a_{i+1,j-1}+ b_{i+1,j-1}}\frac{1+F_{2i,j-1}}{1+1/F_{2i,j}}.
$$
By putting a factor  $\frac{ b_{i+1,j}}{ a_{i+1,j-1}}$ on front we get
\begin{multline*}
\hat F_{2i,j}=\frac{ b_{i+1,j}}{ a_{i+1,j-1}}\frac{\frac{ a_{i+1,j}}{ b_{i+1,j}}+1}{1+\frac{ b_{i+1,j-1}}{ a_{i+1,j-1}}}\frac{1+F_{2i,j-1}}{1+1/F_{2i,j}}\\
=F_{2i+1,j-1}  \frac{1+F_{2i+2,j}}{1+1/F_{2i+2,j-1}}\frac{1+F_{2i,j-1}}{1+1/F_{2i,j}}.
\end{multline*}
This proves the statement for the even face weights.

For the odd face weights, we substitute~\eqref{eq:hata} and~\eqref{eq:hatb} into~\eqref{eq:oddfaceweightpaths}  giving
$$
\hat F_{2i+1,j}=\frac{\hat { b}_{i+1,j+1}}{\hat {  a}_{i+1,j}}=\frac{ b_{i,j}}{ a_{i,j}}=\frac{1}{F_{2i+2,j}},
$$
and this finishes the proof.
\end{proof}
We have the following corollary. 
\begin{cor}\label{cor:shufflematrixsame}
The evolution of the face weights under domino shuffle and the map $\{M_i\}_{i\in \mathbb{Z}}\mapsto \{\hat M_i\}_{i\in \mathbb{Z}}$ are the same. 
\end{cor}
\begin{proof}
This is immediate from comparing the evolution of the face weights in Proposition~\ref{prop:squaremoveall} and Theorem~\ref{thm:dynamics1}.
\end{proof}

It is also possible to define a reverse flow, which we will discuss now. 

Given transition matrices $M_i$ on the graph $\mathcal G$, it follows from the above lemma that 
$$
 M_{2i-1} M_{2i} = Y_i^{-1} \dbtilde M_{2i} \dbtilde M_{2i-1} Y_i\
$$
where $M_{2i-1}=\dbtilde  M_{2i-1}=\Psi$, $Y_i=D(\mathtt a+\sigma^{-1}(\mathtt b_i))$, and $\dbtilde M_{2i}= \Phi(\mathtt a_i,\sigma^{-1} (\mathtt b_i))$. We can define now a new weighting $\{\check M_{i}\}_i$ by setting  
$$
 \check M_{2i}= Y_{i-1} Y_i^{-1}M_{2i},
$$
and $\check M_{2i+1}=\Psi$. With this definition we can rewrite~\eqref{eq:baxter} as
\begin{equation}
    \label{eq:baxter22}
    M_{2i-1} M_{2i} = Y_{i-1}^{-1} \check M_{2i} \check M_{2i-1} Y_i\
\end{equation}
 Now the parameters  $(\check{\mathtt a}_i)_{ i\in \mathbb Z }$ and $(\check{\mathtt b}_i)_{ i\in \mathbb Z }$ for $\check M_i$ can be obtained from the parameters $(\mathtt a_i)_{ i\in \mathbb Z }$ and $(\mathtt b_i)_{ i\in \mathbb Z }$ as follows:

\begin{equation} \label{eq:defchecka}
    \check{ a}_{i,j}=  a_{i,j} \frac{ a_{i-1,j}+ b_{i-1,j+1}}{ a_{i,j}+ b_{i,j+1}},
\end{equation}
and 
\begin{equation}\label{eq:defcheckb}
    \check{ b}_{i,j}=  b_{i,j+1} \frac{ a_{i-1,j}+ b_{i-1,j+1}}{ a_{i,j}+ b_{i,j+1}}.
\end{equation}
The following theorem explains how the face weights change under the reverse dynamics  $\{M_i\}_{i \in \mathbb Z}\mapsto \{\check M_i\}_{i \in \mathbb Z}$.
\begin{thm}
    Under the map $\{M_i\}_{i \in \mathbb Z}\mapsto \{\check M_i\}_{i \in \mathbb Z}$ the face weights change as:
    \begin{align}
        \check F_{2i,j}&=\frac{1}{F_{2i-1,j}},\\
        \check F_{2i+1,j}
        &=F_{2i,j+1}\frac{1+F_{2i-1,j+1}}{1+1/F_{2i+1,j+1}} \frac{1+F_{2i+1,j}}{1+1/F_{2i-1,j}},
    \end{align}
\end{thm}
\begin{proof}
We start with $\check F_{2i,j}$. By~\eqref{eq:evenfaceweightpaths} and using~\eqref{eq:defchecka} and~\eqref{eq:defcheckb} we find
\begin{equation}
    \check F_{2i,j}= \frac{\check{ a}_{i,j}}{\check{ b}_{i,j}}
    =\frac{{ a}_{i,j}}{{ b}_{i,j+1}}=\frac{1}{F_{2i-1,j}}.
\end{equation}
and we arrive at the statement for $\check F_{2i,j}$ Then, by~\eqref{eq:oddfaceweightpaths} and using~\eqref{eq:defchecka} and~\eqref{eq:defcheckb} we find
\begin{multline}
    \check F_{2i+1,j}=\frac{ \check { b}_{i,j+1}}{\check{ a}_{i+1,j}}
    =\frac{{b}_{i+1,j+2}}{ a_{i+1,j}} \frac{ a_{i,j+1}+  b_{i,j+2}}{ a_{i+1,j+1}+  b_{i+1,j+2}} \frac{ a_{i+1,j}+  b_{i+1,j+1}}{ a_{i,j}+  b_{i,j+1}}
\end{multline}
This can be rewritten using~\eqref{eq:oddfaceweightpaths} to
\begin{multline}
    \check F_{2i+1,j}
    = \frac{ a_{i,j+1}+  b_{i,j+2}}{1+1/F_{2i+1,j+1}} \frac{1+F_{2i+1,j}}{ a_{i,j}+  b_{i,j+1}}
    =\frac{ a_{i,j+1}+  b_{i,j+2}}{1+1/F_{2i+1,j+1}} \frac{1+F_{2i+1,j}}{ a_{i,j}+  b_{i,j+1}}.
\end{multline}
Then we can put a factor $F_{2i,j+1}=\frac{ a_{i,j+1}}{ b_{i,j+1}}$ in front and obtain the statement for $\check F_{2i+1,j}$.
\end{proof}

\section{Inverse from matrix refactorization} \label{sec:refactorizaion}

In this section we return to the Eynard-Mehta Theorem~\ref{thm:em} for the non-intersecting paths,  and we show how to compute $W^{-1}$ from the matrix refactorization~\eqref{eq:baxter11} and its reverse~\eqref{eq:baxter22}. 
We recall~\eqref{eq:defWEM}
$$
 W_{r,s}=V_{n-r,-s}, \qquad r,s=1,\ldots,n+p,
$$
where  the doubly infinite matrix $(V_{i,j})_{i,j=-\infty}^\infty$ is defined by
\begin{equation} \label{eq:defV}
    V=M_0 \Psi M_2 \Psi M_4  \Psi \cdots  M_{2n-2} \Psi.
\end{equation}
Now $W$ is a submatrix of $V$, but it is not a principal submatrix. It will be convenient to write 
$$
 W_{r,s}=(G)_{n-r,n-s}, \qquad r,s=1,\ldots,n+p,
$$
where $G=V S^{-n}$ and $S$ is the shift matrix~\eqref{eq:defShift}.

Although we will focus on the non-intersecting paths, the analysis in this section can also be applied to find the inverse of the Kasteleyn matrix through Theorem~\ref{thm:AztecTowerLGV}.
 Indeed, we have that $W_{r,s}=(-1)^n {\mathrm i}^{i+3j+1}\tilde{W}_{n,p}^{\mathrm{Tow}}(r,s)$ for $1 \leq r,s\leq n+p$ where $\tilde{W}_{n,p}^{\mathrm{Tow}}$ is defined in Theorem~\ref{thm:AztecTowerLGV}; see  also Remark~\ref{rem:signs} for a statement on the exact sign. 

The heart of the matter is that the map~\eqref{eq:baxter11} and its reverse~\eqref{eq:baxter22} can be iteratively used to find  LU-decomposition and UL-decomposition of the matrix $G$. From these decompositions it will be easy to find the inverse of $G$. However, we are after the inverse of a particular submatrix of $G$ of size $(n+p) \times (n+p)$. We will show how  an \textit{approximate} inverse for this submatrix can be computed using the LU- and UL-decomposition for the  doubly infinite matrix $G$. This  construction  is inspired by the formula for the approximate inverse of submatrices of block Toeplitz matrices, introduced by Widom~\cite{Wid74}. With the approximate inverse at hand, it will be easy to take the limit $p \to \infty$ and give a general expression for the correlation function $K$ of~\eqref{eq:kernelEM}. 

Before we come to the arguments we stress that the relevance of the final result is the following: if one is able to track and comprehend the flows defined by iterating the maps~\eqref{eq:hata}--\eqref{eq:hatb} and~\eqref{eq:defchecka}--\eqref{eq:defcheckb}  (for instance,  by finding closed  expressions) then the final result of the procedure in this section will give an explicit expression for correlation kernel. Understanding these flows is not a trivial matter, and one typically has to resort to weightings with special structures. For instance, uniform weights or doubly periodic weights which we will discuss briefly in Section~\ref{sec:periodic}.

\subsection{Inverse from LU- and UL-decomposition} \label{sec:wid}
We start with some basics facts on LU-decompositions.

Let $(G_{i,j})_{i,j=-\infty}^{\infty}$ be an infinite matrix that has an LU-decomposition and a UL-decomposition.  That is, we assume that, for $j=1,2$, there exist  lower triangular matrices $L^{(j)}$ and upper triangular matrices $U^{(j)}$ such that 
$$
    G=L^{(1)}U^{(1)}=U^{(2)}L^{(2)}.
$$
Suppose now that $L^{(j)}$ and $U^{(j)}$ are invertible and denote the inverses, for $j=1,2$, by the lower triangular matrices $\Lambda^{(j)}$ and upper triangular matrices $\Upsilon^{(j)}$. Then 
$$
    G^{-1}=\Upsilon^{(1)}\Lambda^{(1)}=\Lambda^{(2)}\Upsilon^{(2)}.
$$
We decompose all matrices in nine blocks:
\begin{equation}\label{eq:decompblocks}
 A=\begin{pmatrix}A_{11} & A_{12} & A_{13}\\
 A_{21}& A_{22} & A_{23}\\
 A_{31} & A_{32} & A_{33}\end{pmatrix},
 \end{equation}
where $A_{22}$ is a finite $(n+p)\times (n+p)$ matrix such that $(A_{22})_{i,j}=A_{i,j}$ for $i,j=-p,\ldots,n-1$. Note this also fixed the dimensions of the other blocks. In particular $A_{11}$ and $A_{33}$ are square infinite matrices.

We also need  the diagonal matrix $P_2$ defined by
$$P_2=\begin{pmatrix}
    0 & 0 & 0\\
    0 & I_2 & 0 \\
    0 & 0 & 0
\end{pmatrix},
$$
where $I_2$ is the identity matrix of size $n+p$ and is placed at columns at rows/columns of indices between $-p$ and $n-1$. We let $I$ denote the doubly infinite identity matrix. We will need the following formula 
\begin{equation}\label{eq:Lblockinverse}
\begin{pmatrix}
    L_{11} & 0 & 0\\
    L_{21} & L_{22} & 0 \\
    L_{31} & L_{32} & L_{33}
\end{pmatrix}^{-1}=
\begin{pmatrix}
    L_{11}^{-1} & 0 & 0\\
   -L_{22}^{-1} L_{21} L_{11}^{-1} & L_{22}^{-1} & 0 \\
    (L_{33}^{-1}L_{32}L_{22}^{-1}L_{21}-L_{33}^{-1}L_{31}^{-1})L_{11}^{-1} & -L_{33}^{-1}L_{32}L_{22}^{-1} & L_{33}^{-1}
\end{pmatrix}.
\end{equation}

The following result is a key step in inverting the matrix $W$. 

\begin{lem}
    We have
    \begin{multline} \label{eq:approximateinverse}
     \left(G_{22}\right)\left(\left(G^{-1}\right)_{22}-\Upsilon^{(1)}_{23} \Lambda^{(1)}_{32} -\Lambda^{(2)}_{21} \Upsilon^{(2)}_{12}\right)\\
     =I_2+L^{(1)}_{21} U^{(1)}_{11} \Upsilon^{(1)}_{13} \Lambda^{(1)}_{32}+U_{23}^{(2)}L^{(2)}_{33}\Lambda^{(2)}_{31} \Upsilon^{(2)}_{12}.
    \end{multline}
\end{lem}
\begin{proof}
   
   From $G^{-1}=\Lambda^{(2)} \Upsilon^{(2)}=\Upsilon^{(1)}\Lambda^{(1)}$ and the fact that $\Lambda^{(j)}$ and $\Upsilon^{(j)}$ are lower and upper triangular respectively, we find
   \begin{equation} \label{eq:G22inverse12}
        (G^{-1})_{12}=\Lambda^{(2)}_{11} \Upsilon_{12}^{(2)}.
   \end{equation} 
  and
   \begin{equation} \label{eq:G22inverse32}
       (G^{-1})_{32}= \Upsilon_{33}^{(1)}\Lambda^{(1)}_{32}.
   \end{equation}
   Now
   \begin{multline} \label{eq:G22G22inverse}
   \begin{pmatrix} 0 & 0 & 0\\
    0& G_{22}\left(G^{-1}\right)_{22}&0\\
    0 & 0 &0
   \end{pmatrix}=P_2 G P_2 G^{-1} P_2=
   P_2-P_2 G(I-P_2) G^{-1}P_2\\=P_2-
   P_2 G \begin{pmatrix} 0& (G^{-1})_{12} &0 \\
    0& 0 & 0 \\
    0& (G^{-1})_{32}& 0
   \end{pmatrix} 
   =P_2-
   P_2 G \begin{pmatrix} 0& \Lambda^{(2)}_{11} \Upsilon_{12}^{(2)} &0 \\
    0& 0 & 0 \\
    0& \Upsilon_{33}^{(1)}\Lambda^{(1)}_{32}& 0
   \end{pmatrix} 
\end{multline}
where we inserted~\eqref{eq:G22inverse12} and~\eqref{eq:G22inverse32} in the last step. Note also that 
\begin{equation}\label{eq:G22rest}
    \begin{pmatrix} 0 &0 &0 \\
        0 & G_{22} \left( \Upsilon^{(1)}_{23} \Lambda^{(1)}_{32} +\Lambda^{(2)}_{21} \Upsilon^{(2)}_{12} \right)& 0 \\
        0 & 0 &0 
    \end{pmatrix}=P_2 G
    \begin{pmatrix} 0&0 &0 \\
        0&  \Upsilon^{(1)}_{23} \Lambda^{(1)}_{32} +\Lambda^{(2)}_{21} \Upsilon^{(2)}_{12} & 0 \\
        0& 0& 0
       \end{pmatrix}.
    \end{equation}
    Combining~\eqref{eq:G22G22inverse} and~\eqref{eq:G22rest} gives
\begin{multline}\label{eq:G22G22inverserest}
\begin{pmatrix}
    0&0&0\\
    0& G_{22}\left(\left(G^{-1}\right)_{22}-\Upsilon^{(1)}_{23} \Lambda^{(1)}_{32} -\Lambda^{(2)}_{21} \Upsilon^{(2)}_{12} \right) & 0 \\
    0 &0 & 0
\end{pmatrix}
=P_2-P_2G\begin{pmatrix} 0&\Lambda^{(2)}_{11} \Upsilon_{12}^{(2)} &0 \\
    0&  \Upsilon^{(1)}_{23} \Lambda^{(1)}_{32} +\Lambda^{(2)}_{21} \Upsilon^{(2)}_{12} & 0 \\
    0& \Upsilon_{33}^{(1)}\Lambda^{(1)}_{32}& 0
   \end{pmatrix}\\
    =P_2-P_2G\begin{pmatrix} 0&\Lambda^{(2)}_{11} \Upsilon_{12}^{(2)} &0 \\
    0& \Lambda^{(2)}_{21} \Upsilon^{(2)}_{12} & 0 \\
    0&0& 0
   \end{pmatrix}-P_2G\begin{pmatrix} 0&0&0 \\
    0&  \Upsilon^{(1)}_{23} \Lambda^{(1)}_{32} & 0 \\
    0& \Upsilon_{33}^{(1)}\Lambda^{(1)}_{32}& 0
   \end{pmatrix}\\
   =P_2-P_2U^{(2)} L^{(2)} \begin{pmatrix} 0&\Lambda^{(2)}_{11} \Upsilon_{12}^{(2)} &0 \\
    0& \Lambda^{(2)}_{21} \Upsilon^{(2)}_{12} & 0 \\
    0&0& 0
   \end{pmatrix}-P_2 L^{(1)} U^{(1)} \begin{pmatrix} 0&0&0 \\
    0&  \Upsilon^{(1)}_{23} \Lambda^{(1)}_{32} & 0 \\
    0& \Upsilon_{33}^{(1)}\Lambda^{(1)}_{32}& 0
   \end{pmatrix}.
\end{multline}
Now using that $\Lambda^{(2)}$ is the inverse of $L^{(2)}$ we can write {using~\eqref{eq:Lblockinverse}}
$$L^{(2)} \begin{pmatrix} 0&\Lambda^{(2)}_{11} \Upsilon_{12}^{(2)} &0 \\
    0& \Lambda^{(2)}_{21} \Upsilon^{(2)}_{12} & 0 \\
    0&0& 0
   \end{pmatrix}=\begin{pmatrix} 0&\Upsilon_{12}^{(2)} &0 \\
    0& 0 & 0 \\
    0&-L^{(2)}_{33}\Lambda^{(2)}_{31} \Upsilon^{(2)}_{12}& 0
   \end{pmatrix},
   $$
   and thus, using that fact that $U^{(2)}$ is upper triangular
\begin{equation} \label{eq:P2U2L2}
P_2U^{(2)} L^{(2)} \begin{pmatrix} 0&\Lambda^{(2)}_{11} \Upsilon_{12}^{(2)} &0 \\
    0& \Lambda^{(2)}_{21} \Upsilon^{(2)}_{12} & 0 \\
    0&0& 0
   \end{pmatrix}=-U_{23}^{(2)}L^{(2)}_{33}\Lambda^{(2)}_{31} \Upsilon^{(2)}_{12}
\end{equation}
Similarly, 
$$
U^{(1)} \begin{pmatrix} 0&0&0 \\
    0&  \Upsilon^{(1)}_{23} \Lambda^{(1)}_{32} & 0 \\
    0& \Upsilon_{33}^{(1)}\Lambda^{(1)}_{32}& 0
   \end{pmatrix}
   =\begin{pmatrix} 0& -U^{(1)}_{11} \Upsilon^{(1)}_{13} \Lambda^{(1)}_{32} &0 \\
    0& 0 & 0 \\
    0& \Lambda^{(1)}_{32}& 0
   \end{pmatrix}
   $$
   and thus
   \begin{equation}\label{eq:P2L1U1}
   P_2 L^{(1)}U^{(1)} \begin{pmatrix} 0&0&0 \\
    0&  \Upsilon^{(1)}_{23} \Lambda^{(1)}_{32} & 0 \\
    0& \Upsilon_{33}^{(1)}\Lambda^{(1)}_{32}& 0
   \end{pmatrix}
   =
-L^{(1)}_{21} U^{(1)}_{11} \Upsilon^{(1)}_{13} \Lambda^{(1)}_{32}.
\end{equation}
   We obtain the statement after inserting~\eqref{eq:P2U2L2} and~\eqref{eq:P2L1U1} into~\eqref{eq:G22G22inverserest}. 
\end{proof}
The intuition behind~\eqref{eq:approximateinverse} is the following: in special cases, the matrices that we are interested in are diagonally dominant and the values on the $k$-th subdiagonals above and below the main diagonal  decrease rapidly with $k$. This means that the entries of the matrices $\Upsilon^{(1)}_{13}$ and $\Lambda^{(2)}_{31}$ are small. This can be used to show that the right-hand side of~\eqref{eq:approximateinverse} equals $I_2$ plus a small correction.  Thus 
\begin{equation}\label{eq:approxiinv}
 \left(G^{-1}\right)_{22}-\Upsilon^{(1)}_{23} \Lambda^{(1)}_{32} -\Lambda^{(2)}_{21} \Upsilon^{(2)}_{12}
\end{equation}
 is an \emph{approximate} inverse to $G_{22}$.

Note that 
\begin{equation}\label{eq:G22iinv1}
    (G^{-1})_{22}=\Lambda^{(2)}_{22} \Upsilon^{(2)}_{22}+\Lambda^{(2)}_{21} \Upsilon^{(2)}_{12},
\end{equation}
but also 
\begin{equation}\label{eq:G22iinv2}
    (G^{-1})_{22}= \Upsilon^{(1)}_{22} \Lambda^{(1)}_{22}+\Upsilon^{(1)}_{23} \Lambda^{(1)}_{32}.
\end{equation}
In the upper left corner of~\eqref{eq:approxiinv}, the matrix $\Upsilon^{(1)}_{23} \Lambda^{(1)}_{32}$ is small and thus by~\eqref{eq:approxiinv} and~\eqref{eq:G22iinv1} we find that
$$
 (G_{22})^{-1}\approx \Lambda^{(2)}_{22} \Upsilon^{(2)}_{22}.
$$
Similarly, in the lower right corner we find 
$$
 (G_{22})^{-1}\approx \Upsilon^{(1)}_{22} \Lambda^{(1)}_{22}.
$$
We formalize this discussion in the following proposition. 
\begin{prop}\label{prop:asymptoticinverse}
    Assume that there exists an $R>0$ and $0 < \rho <1$ such that, for $i,j \in \mathbb Z$,
    \begin{align} \label{eq:bounds_giving_essentialbanded1}
     |(L_{21}^{(1)}U_{11}^{(1)})_{i,j}|&\leq R, \\
      |(U_{23}^{(2)}L^{(2)}_{33})_{i,j}| &\leq R, \\
      |\Lambda_{i,j}^{(1,2)}|, |\Upsilon_{i,j}^{(1,2)}|& \leq R \rho^{-|i-j|}. \label{eq:bounds_giving_essentialbanded3}
    \end{align}

    Then,
    \begin{equation} \label{eq:approximateinverse2}
        \left(G_{22}\right)\left(\left(G^{-1}\right)_{22}-\Upsilon^{(1)}_{23} \Lambda^{(1)}_{32} -\Lambda^{(2)}_{21} \Upsilon^{(2)}_{12}\right)\\
        =I_2+\mathcal O(p^2 \rho^{p}),
       \end{equation}
       as $p\to \infty$. The error term is with respect to the standard matrix norm. 
       
       Moreover, for $i,j=1,\ldots,n+p$,
       \begin{equation} \label{eq:asymptoticinverseleftercorner}
       \left((G_{22})^{-1}\right)_{-p+i-1,-p+j-1}=(\Lambda^{(2)}_{22} \Upsilon^{(2)}_{22})_{-p+i,-p+j}+ \mathcal O(p^2\rho^{p}),
       \end{equation}
       and 
       \begin{equation} \label{eq:asymptoticinverserightcorner}
       \left((G_{22})^{-1}\right)_{n-i,n-j}=(\Upsilon^{(1)}_{22} \Lambda^{(1)}_{22})_{n-i,n-j}+ \mathcal O(p^2 \rho^{p}),
       \end{equation}
       as $p\to \infty$.
\end{prop}

\begin{proof}
    By expanding the product
    $$
     \left(L_{21}^{(1)} U_{11}^{(1)} \Upsilon^{(1)}_{13} \Lambda_{32}^{(1)}\right)_{r,s}
     =\sum_{j_1=-\infty}^{-p-1}  \sum_{j_2=n}^\infty  \left(L_{21}^{(1)} U_{11}^{(1)}\right)_{r,j_1} \left( \Upsilon^{(1)}_{13}\right)_{j_1,j_2}\left( \Lambda_{32}^{(1)}\right)_{j_2,s}
    $$
    and using the bounds~\eqref{eq:bounds_giving_essentialbanded1} and~\eqref{eq:bounds_giving_essentialbanded3} we find, for $-p \leq r,s \leq n-1$,
    $$
    |\left(L_{21}^{(1)} U_{11}^{(1)} \Upsilon^{(1)}_{13} \Lambda_{32}^{(1)}\right)_{r,s}|\leq R^3 \sum_{j_1=-\infty}^{-p-1}  \sum_{j_2=n}^\infty  \rho^{j_2-j_1} \rho^{j_2-s}
    =  C \rho^{p},
    $$
    for some constant $C$ independent of $r$ and $s$. By further using the standard inequality $\|A\| \leq \sum_{i,j}|A_{i,j}|$ we thus find
    \begin{equation} \label{eq:approximateinversenorm1}
        \left\|L_{21}^{(1)} U_{11}^{(1)} \Upsilon^{(1)}_{13} \Lambda_{32}^{(1)}\right\| \leq C (n+p)^2 \rho^p.
    \end{equation}
    By a similar argument, we find that (we can choose the constant $C>0$ large enough such that also)
    \begin{equation}\label{eq:approximateinversenorm2}
    \left\|U_{23}^{(2)}L^{(2)}_{33}\Lambda^{(2)}_{31} \Upsilon^{(2)}_{12}\right\| \leq C (n+p)^2 \rho^p.
    \end{equation} 
    Then~\eqref{eq:approximateinversenorm1} and~\eqref{eq:approximateinversenorm2} prove~\eqref{eq:approximateinverse2}. It also shows that $G_{22}$ is invertible for $p$ sufficiently large and 
    $$
     \left(G_{22}\right)^{-1}=\left(\left(G^{-1}\right)_{22}-\Upsilon^{(1)}_{23} \Lambda^{(1)}_{32} -\Lambda^{(2)}_{21} \Upsilon^{(2)}_{12}\right)+ \mathcal O(\rho^{p}),
     $$
     as $p \to \infty$. 
     
      Finally,~\eqref{eq:asymptoticinverseleftercorner} and~\eqref{eq:asymptoticinverserightcorner} follow by applying~\eqref{eq:bounds_giving_essentialbanded3} to~\eqref{eq:G22iinv1} and~\eqref{eq:G22iinv2} respectively and inserting the result in~\eqref{eq:approxiinv}.
\end{proof}

In the special case that $G$ is a block Toeplitz matrix, this result was already proved by Widom in~\cite{Wid74}. We will also come back to this in Appendix~\ref{appendix:BD}.
\subsection{LU-decomposition} \label{sec:generalLU}
We now show how we can obtain an LU-decomposition for the doubly infinite matrix $VS^{-n}$ with $V$ as in~\eqref{eq:defV}. The idea is to use the commutation relation~\eqref{eq:baxter11}  and shift  all $M_{2i}$ to the right and all $M_{2i+1}$ to the left. 

Repeatedly applying~\eqref{eq:baxter11} gives
$$
V=X_0  \Psi \hat M_0 \Psi   \hat M_2 \Psi \cdots  \Psi \hat M_{2n-2} X_{n}^{-1}. 
$$
Care should be take here since $X_n^{-1}$ requires the parameters of $M_{2n}$ and this matrix is not necessarily defined. Instead, we will work with the assumption that $X_{n}=I$, giving
$$
V=X_0  \Psi \hat M_0 \Psi   \hat M_2 \Psi \cdots  \Psi \hat M_{2n-2} 
$$ 
Since $X_0 \Psi$ and $\hat M_{2n-2}$ are at the desired locations already, we drop these factors and continue with 
$$
\hat M_0 \Psi   \hat M_2 \Psi \cdots  \hat M_{2n-4} \Psi,
$$ 
and refactorize $\hat M_{2j} \Psi$ using the same principles.  We iterate this procedure in total  $n$ times so that all factors are on the desired place. 

This iteration is described by the following algorithm:

For  $j=0$ we set, 
$$
({\mathtt a}^{(0)}_i,{\mathtt b}^{(0)}_i)_{i=0}^{n-1}=({\mathtt a}^{(j)}_i,{\mathtt b}^{(j)}_i)_{i=0}^{n-1}
$$
and then, for $j=1,\ldots, n$,
$$
({\mathtt a}^{(j)}_i,{\mathtt b}^{(j)}_i)_{i=0}^{n-j}=(\hat {\mathtt a}^{(j-1)}_i,\hat {\mathtt b}^{(j-1)}_i)_{i=0}^{n-j}
$$
where, for $i=0,\ldots,n-j-1$, we define $\hat {\mathtt a}^{(j-1)}_i$ and $\hat {\mathtt b}^{(j-1)}_i$ as in~\eqref{eq:hata} and~\eqref{eq:hatb} respectively,  and, for $i=n-j$ we set, for $k \in \mathbb Z$,
\begin{equation}\label{eq:hatalast}
    \hat { a}^{(j-1)}_{n-j,k}=a^{(j-1)}_{n-j,k} \frac{1}{ a^{(j-1)}_{n-j,k}+ b^{(j-1)}_{n-j,k}},
\end{equation}
and 
\begin{equation} \label{eq:hatblast}
    \hat { b}^{(j-1)}_{n-j,k}=  b_{n-j,k-1}^{(j-1)} \frac{1}{ a^{(j-1)}_{n-j,k-1}+ b^{(j-1)}_{n-j,k-1}}.
\end{equation}
The difference in the last step is explained by the fact that~\eqref{eq:hata} and~\eqref{eq:hatb} are defined for infinite products, while we are working with a finite product and thus care should be taken for the last matrix in the product.

We then set 
$$ 
M_{2i}^{(j)}=D(\mathtt a^{(j)}_i)+D(\mathtt b^{(j)}_i) S,
$$
Note also that from~\eqref{eq:baxter11} we have 
$$
M_{2i}^{(j-1)}\Psi =X_i^{(j-1)} \Psi M_{2i}^{(j)} (X_{i+1}^{(j-1)})^{-1}, \quad i=0,\ldots,n-j,
$$
where 
$$
X_i^{(j-1)}=\begin{cases}
    D({\mathtt a}_i^{(j-1)}+\text b_i^{(j-1)}), & i=0,\ldots,n-j\\
    I, & i=n-j+1.
\end{cases}
$$
The conclusion of this procedure is that 
$$ 
    V=X_0^{(0)} \Psi X_0^{(1)} \Psi X_0^{(2)} \Psi \cdots X_0^{(n-1)} \Psi M_0^{(n)}M_2^{(n-1)}\cdots M_{2(n-1)}^{(1)}.
$$
After multiplying by  $S^{-n}$ from the right, this gives an LU-decomposition for $G$.

\begin{lem}\label{lem:LUd}
    Let 
\begin{equation}
\label{eq:defL1}
L^{(1)}=X_0^{(0)} \Psi X_0^{(1)} \Psi X_0^{(2)} \Psi \cdots X_0^{(n-1)}\Psi
\end{equation}
 and
\begin{equation} \label{eq:defU1}
U^{(1)}= M_0^{(n)}M_2^{(n-1)}\cdots M_{2(n-1)}^{(1)}S^{-n}.
\end{equation}
The $L^{(1)}$ is lower triangular, $U^{(1)}$ is  upper triangular matrix and $G=VS^{-n}=L^{(1)}U^{(1)}$.
\end{lem}
\begin{proof}
    We have already seen that  $VS^{-n}=L^{(1)}U^{(1)}$. It remains to prove the triangular structure of $L^{(1)}$ and $U^{(1)}$. From the definition~\eqref{eq:defpsi} it is clear that $\Psi$ is lower triangular. Combining this with the fact that each $X_0^{(j)}$ is diagonal, it follows that $L^{(1)}$ is lower triangular. To see that $U^{(1)}$ is upper triangular, we write 
    $$
     U^{(1)}= M_0^{(n)}S^{-1} S M_2^{(n-1)}S^{-2} S^2 \cdots S^{n-1} M_{2(n-1)}^{(1)}S^{-n}.
    $$

    Each $S^{i} M_{2i}^{(n-i)} S^{(-i-1)}$, for $i=0,\ldots,n-1$, is upper triangular and thus the product $U^{(1)}$ is upper triangular. 
\end{proof}
\subsection{UL-decomposition}

To obtain a UL-decomposition we use the reverse dynamics but with all the $M_{2i}$ at the left and all $M_{2i+1}$ to the right. Since the outermost factors are already at the desired place, we drop them in the first step and start with 
$$
\Psi M_2 \Psi M_4  \Psi \cdots  M_{2n-2}.
$$
Repeatedly applying~\eqref{eq:baxter22} gives
$$
\Psi M_2 \Psi M_4  \Psi \cdots  M_{2n-2} =Y_0^{-1} \check M_2  \Psi\check M_4 \Psi \cdots   \check M_{2n-2}\Psi Y_{n-1}. 
$$
In this factorization we have the freedom to choose $Y_0$ as we wish (as long as it is consistent for the definition of $\check M_2$), and we choose it to be $Y_0=I$ so that it can be removed from the product. 
Now we drop the first and last factor  and obtain
$$
\Psi \check M_4 \Psi   \check M_6 \Psi \cdots  \check M_{2n-2},
$$ 
and refactorize $\Psi \check M_{2j} $ using the same principles.  We then iterate this procedure in total $n-1$  times until all factors are on the desired place. 

The result can be presented as follows. 

We start by defining 
$$
 (\mathtt a^{[0]}_i, \mathtt b^{[0]}_i)_{i=0}^{n-1} =  (\mathtt a_i,\mathtt b_i)_{i=0}^{n-1}
$$
and then set, for $j=1,\ldots,n-1$, 
$$
 (\mathtt a^{[j]}_i, \mathtt b^{[j]}_i)_{i=j}^{n-1} =  (\check {\mathtt a}_i^{[j-1]},\check{ \mathtt b}_i^{[j-1]})_{i=j}^{n-1}
$$
where, for $i=j+1,\ldots,n-1$, we define $\check {\mathtt a}_i^{[j-1]}$ and $\check{ \mathtt b}_i^{[j-1]}$ as in~\eqref{eq:defchecka} and~\eqref{eq:defcheckb} respectively, and $i=j$ we set 
$$
 \check a_{j,k}^{[j-1]}=a_{j,k} ^{[j-1]} \frac{1}{ a_{j,k}^{[j-1]}+ b_{j,k+1}^{[j-1]}},
$$
and
$$
\check{ b}_{j,k}^{[j-1]}=  b_{j,k+1}^{[j-1]} \frac{ 1}{ a_{j,k}^{[j-1]}+ b_{j,k+1}^{[j-1]}}.
$$
We then set
$$M_{2i}^{[j]}=D(\mathtt a^{[j]}_i)+D(\mathtt b^{[j]}_i)S,$$
for $j=1,\ldots, n-1$ and $i=j,\ldots, n-1$.

It is important to observe that from~\eqref{eq:baxter22} we find that 
$$
\Psi M_{2i}^{[j-1]}=(Y_{i-1}^{[j-1]})^{-1}  \check M_{2i}^{[j]} \Psi Y_i^{[j-1]} ,  \quad i=j,\ldots,n-1,
$$
where
$$Y_{i-1}^{[j-1]}=\begin{cases}
    D(\check{\mathtt a}_i^{[j-1]}+\sigma^{-1}(\check {\mathtt b}_i^{[j-1]})), & i=j+1,\ldots,n-1\\
    I, & i=j,
\end{cases}
$$
for $j=1,\ldots, n-1$ and $i=j,\ldots, n-1$.

The conclusion of this procedure is that 
$$ 
    V=M_0^{[0]} M_2^{[1]} \cdots M_{2n-2}^{[n-1]}  \Psi Y_{n-1}^{[n-2]} \Psi \cdots Y_{n-1}^{[0]}\Psi.
$$
The following lemma is the analogue of Lemma~\ref{lem:LUd} for the reverse dynamics.

\begin{lem}
    Let 
\begin{equation}\label{eq:defL2}
L^{(2)}=S^{n}\Psi Y_{n-1}^{[n-2]} \Psi \cdots Y_{n-1}^{[0]}\Psi S^{-n}
\end{equation}
 and
\begin{equation} \label{eq:defU2}
U^{(2)}= M_0^{[0]} M_2^{[1]} \cdots M_{2n-2}^{[n-1]}S^{-n}.
\end{equation}
The $L^{(2)}$ is a lower triangular matrix, $U^{(2)}$ is  upper triangular matrix and $G=VS^{-n}=U^{(2)}L^{(2)}.$
\end{lem}
\begin{proof}
    As the proof is similar to the proof of Lemma~\ref{lem:LUd} it will be omitted.
\end{proof}

\subsection{The correlation kernel}

Now that we have both an LU- and UL-decomposition: 
$$
    G=L^{(1)}U^{(1)}=U^{(2)}L^{(2)},
$$
with $L^{(1)}$ and $U^{(1)}$ as in~\eqref{eq:defL1} and~\eqref{eq:defU1} respectively, and $L^{(2)}$ and $U^{(2)}$ as in~\eqref{eq:defL2} and~\eqref{eq:defU2} respectively, we can try the ideas of Section~\ref{sec:wid} to compute the inverse of the matrix $W$ in~\eqref{eq:kernelEM}. To this end we define 
$$
 \Lambda^{(j)}= (L^{(j)})^{-1}, \qquad \Upsilon^{(j)}= (U^{(j)})^{-1},
$$
and decompose all matrices into blocks as in~\eqref{eq:decompblocks}.

Note that the inverses $\Lambda^{(1,2)}$ of $L^{(1,2)}$ are rather easy to compute, as they are products that alternate between diagonal matrices (with trivial inverses) and the matrix $\Psi$, which has inverse
 $$
    (\Psi^{-1})_{i,j}=
    \begin{cases}1 & \text{if } i=j \\
    	-1 &\text{if } i=j+1,\\
        0 & \text{otherwise.}
    \end{cases}
 $$
 Observe also that this means that $\Lambda^{(1,2)}$ are banded matrices where the width of the band depends on $n$, but not on $p$.  

 The inverses $\Upsilon^{(1,2)}$ of $U^{(1,2)}$ are also easy to compute, but now $\Upsilon^{(1,2)}$ are not banded.

 In order to take the limit $p \to \infty$, we also need $\Lambda^{(1,2)}$ and $\Upsilon^{(1,2)}$  to satisfy~\eqref{eq:bounds_giving_essentialbanded3}. This requires a condition on the parameters $a_{i,j}$ and $b_{i,j}$. 
\begin{assum}\label{assume}
    We assume that there exists $R>0$ and $0< \rho<1$ such that, for $1\leq i \leq n$ and $k\in \mathbb N$,
    \begin{equation}\label{eq:assumption1}
        \sup_{j} \frac{a_{i,j}a_{i,j+1} \cdots a_{i,j+k-1}}{b_{i,j}b_{i,j+1} \cdots b_{i,j+k-1}} \leq R \rho^k.
    \end{equation}
    Furthermore, we assume that there exists $0<\delta_1< \delta_2$ such that, for $j \in \mathbb Z$ and $i=1,\ldots,n$,
    \begin{equation} \label{eq:assumption2}
      \delta_1 \leq  a_{i,j}+ b_{i,j} \leq \delta_2.
    \end{equation}
\end{assum}
\begin{lem}
    Under Assumption~\ref{assume}  we have that for each $i=1,\ldots, n$ the inverse of $D(\mathtt a_i)+D(\mathtt b_i)S$ satisfies~\eqref{eq:bounds_giving_essentialbanded3}. 
\end{lem}
\begin{proof}
By using the notation $(\mathtt a_i/ \mathtt b_i)=(a_{i,j}/b_{i,j})_{j \in \mathbb Z}$ and using the rule $S^{-1}D(\mathtt a)=D(\sigma^{-1} \mathtt a) S^{-1}$, we can  write 
    \begin{multline}
    \left(D(\mathtt a_i)S^{-1} +D(\mathtt b_i)\right)^{-1}=\left( \sum_{k=0} ^\infty (-1)^k (D(\mathtt a_i/ \mathtt b_i)S^{-1})^k \right)D(\mathtt b_i)^{-1}\\
    =\left( \sum_{k=0}^\infty (-1)^kD\left(\prod_{\ell=0}^{k-1} \sigma^{-\ell}(\mathtt a_i/ \mathtt b_i)\right)S^{-k} \right)D(\mathtt b_m)^{-1}.
    \end{multline}
By~\eqref{eq:assumption2} we see that the values of the diagonal matrix  $D(\mathtt b_m)^{-1}$ are bounded from above and below. By~\eqref{eq:assumption1} we have that the values of the diagonal matrix  $D\left(\prod_{\ell=0}^{k-1} \sigma^{-\ell}(\mathtt a_m/ \mathtt b_m)\right)$ are bounded by $R \rho^k$. Combining these facts gives that $ \left(D(\mathtt a_i)S^{-1} +D(\mathtt b_i)\right)^{-1}$ satisfies~\eqref{eq:bounds_giving_essentialbanded3}.
\end{proof}
\begin{lem} \label{lem:hatmappreserving}
    Assume that $( \mathtt a_i, \mathtt b_i)_{i=1,\ldots,n-1}$ satisfies  Assumption~\ref{assume}. Then the parameters $(\hat {\mathtt a}_i,\hat {\mathtt b} _i)_{i=1,\ldots,n-1}$ also satisfy assumptions~\eqref{eq:assumption1} and~\eqref{eq:assumption2}, possibly with different values of $R, \delta_1$ and $\delta_2$, but with the same $\rho$. Consequently, each inverse of $D(\hat {\mathtt a}_i)+D(\hat{\mathtt b}_i)S$, for  $i=1,\ldots, n-1$, also satisfies~\eqref{eq:bounds_giving_essentialbanded3}. 
\end{lem}
\begin{proof}
    It is clear from~\eqref{eq:assumption2} that also 
    \begin{equation*}
        \delta_1' \leq  \hat  a_{i,j}+ \hat b_{i,j} \leq \delta_2',
      \end{equation*}
      for some $0<\delta_1'<\delta_2$. Then 
     \begin{multline*}
      \frac{\hat a_{i,j} \hat a_{i,j+1} \cdots \hat a_{i,j+k-1}}{\hat b_{i,j}b_{i,j+1} \cdots \hat b_{i,j+k-1}}\\ = \frac{ a_{i,j}  a_{i,j+1} \cdots  a_{i,j+k-1}}{ b_{i,j-1}b_{i,j} \cdots  b_{i,j+k-2}} \frac{  a_{i+1,j-1}+ b_{i+1,j-1}}{  a_{i,j-1}+ b_{i,j-1}}\frac{ a_{i,j+k-1}+ b_{i,j+k-1}}{  a_{i+1,j+k-1}+  b_{i+1,j+k-1}}\\=\frac{ a_{i,j}  a_{i,j+1} \cdots  a_{i,j+k-1}}{ b_{i,j}  b_{i,j+1} \cdots b_{i,j+k-1}}\frac{b_{i,j+k-1}}{b_{i,j-1}} \frac{  a_{i+1,j-1}+ b_{i+1,j-1}}{  a_{i,j-1}+ b_{i,j-1}}\frac{ a_{i,j+k-1}+ b_{i,j+k-1}}{  a_{i+1,j+k-1}+  b_{i+1,j+k-1}}.
     \end{multline*}
     By~\eqref{eq:assumption1} and~\eqref{eq:assumption2} we thus find 
     $$
     \frac{\hat a_{i,j} \hat a_{i,j+1} \cdots \hat a_{i,j+k-1}}{\hat b_{i,j}b_{i,j+1} \cdots \hat b_{i,j+k-1}} \leq  R \rho^k \frac{\delta_2^3}{\delta_1^3},
     $$
     and thus $(\hat {\mathtt a}_i, \hat {\mathtt b}_i)_{i=0}^{n-1} $ also satisfy both~\eqref{eq:assumption1} and~\eqref{eq:assumption2}, which by the above means that $ \left(D(\hat{\mathtt a_i})S^{-1} +D(\hat{\mathtt b}_i)\right)^{-1}$ satisfies~\eqref{eq:bounds_giving_essentialbanded3}.
\end{proof}
\begin{lem}
    Under assumptions~\eqref{eq:assumption1} and~\eqref{eq:assumption2} it holds that both  $\Lambda^{(1,2)}$ and $\Upsilon^{(1,2)}$ satisfy~\eqref{eq:bounds_giving_essentialbanded3}.
\end{lem}
\begin{proof}
    Let us start by mentioning that if $A$ and $B$ are matrices for which there exists $R>0$ and $0< \rho<1$, such that, for $i,j$ we have
    $$
    |A_{ij}|, |B_{i,j}|\leq R \rho^{-|i-j|},
    $$
    then also 
    \begin{equation} \label{eq:AB}
    |(AB)_{i,j}|\leq \tilde R \rho^{-|i-j|},
    \end{equation}
    for some $\tilde R$.

    As in the proof of Lemma~\ref{lem:LUd} we write $U^{(1)}$ as
    $$
        U^{(1)}=M_0^{(n)}S^{-1} S M_2^{(n-1)}S^{-2} S^2 \cdots S^{n-1} M_{2(n-1)}^{(1)}S^{-n}.
    $$
    By the principle in~\eqref{eq:AB} it is sufficient to show that the inverse of each  $S^{i} M_{2i}^{(n-m)}S^{-i-1}$ satisfies~\eqref{eq:bounds_giving_essentialbanded3}. Conjugation by $S^{\pm 1}$ only moves the values on the diagonal up or down by 1, and thus it is sufficient  to show that the inverse of
    $$ 
       M_{2i}^{(n-i)}S^{-1}=D(\mathtt a^{(n-i)}_i)S^{-1} +D(\mathtt a^{(n-i)}_i),
    $$
    satisfies~\eqref{eq:bounds_giving_essentialbanded3}. But since $({\mathtt a}^{(n-i)}_i,{\mathtt b}^{(n-i)}_i)$ are obtained by iterating the maps~\eqref{eq:hata} and~\eqref{eq:hatb}, this fact follows from Lemma~\ref{lem:hatmappreserving} and we have thus proved the statement for $\Upsilon^{(1)}$. 

    The claim for $\Upsilon^{(2)}$ follows similarly.
    
    The claim for $\Lambda^{(1,2)}$ is easier. Indeed, both matrices are banded with a bandwidth independent of $p$, and it is not hard to verify that under assumption~\eqref{eq:assumption2} there is a uniform bound for all the entries.
\end{proof}
\begin{thm} \label{thm:generalKernel}
    With $L$ and $U$ as in~\eqref{eq:defL1} and~\eqref{eq:defU1} respectively, set $\Upsilon=U^{-1}$ and $\Lambda=L^{-1}$.  Then, for $x_1,x_2\geq -1$, we have   
     \begin{multline} \label{eq:kernelEMtopa}
        \lim_{p \to \infty}  K(m_1,x_1,m_2,x_2)=- \mathbbm{I}[m_1 >m_2] M_{m_2+1}\cdots M_{m_1}(x_1,x_2)\\
            + \sum_{\ell=1}^{\infty}  \left(M_{m_1}\cdots M_{2n-1}S^{-n}\Upsilon\right)_{x_1,n-\ell} (\Lambda M_{0}\cdots M_{m_2-1})_{n-\ell,x_2}.
     \end{multline}
\end{thm}

\begin{proof}
    The starting point is the expression for $K$ in~\eqref{eq:kernelEM}. Then we note that 
    $$
       (W^{-1})_{s,r}=((G_{22})^{-1})_{n-s,n-r}.
    $$
    Now note that there exists a constant $C>0$ such that
    \begin{equation} \label{eq:boundsonallentries}
    |(M_{m_1}\cdots M_{2n-1})(x_1,-s) |,|(M_{0}\cdots M_{m_2-1})(n-r,x_2)| \leq C,
    \end{equation}
    for all $r,s\geq 1$, and thus, by~\eqref{eq:approximateinverse2}, 
    \begin{multline}
    K(m_1,x_1,m_2,x_2)=- \mathbbm{I}[m_1 >m_2] M_{m_2+1}\cdots M_{m_1}(x_1,x_2)\\
            + \sum_{r,s=1}^{n+p}  (M_{m_1}\cdots M_{2n-1})(x_1,-s) ((G^{-1})_{22}- \Upsilon_{23}^{(1)} \Lambda_{32}^{(1)}-\Lambda_{21}^{(2)} \Upsilon_{12}^{(2)})_{n-s,n-r}\\
            \times  (M_{0}\cdots M_{m_2-1})(n-r,x_2)+\mathcal O(p^2 \rho^p),
    \end{multline}
    as $p \to \infty$.
    We then use~\eqref{eq:G22iinv2}  to write
    \begin{multline}
        K(m_1,x_1,m_2,x_2)=- \mathbbm{I}[m_1 >m_2] M_{m_2+1}\cdots M_{m_1}(x_1,x_2)\\
                + \sum_{r,s=1}^{n+p}  (M_{m_1}\cdots M_{2n-1})(x_1,-s) (\Upsilon_{22}^{(1)} \Lambda_{22}^{(1)}-\Lambda_{21}^{(2)} \Upsilon_{12}^{(2)})_{n-s,n-r} \\ 
                \times (M_{0}\cdots M_{m_2-1})(n-r,x_2)+\mathcal O(p^2 \rho^p ),
        \end{multline}
        as $p \to \infty$. It remains to show that the term $\Lambda_{21}^{(2)} \Upsilon_{12}^{(2)}$ can be ignored. 
    To this end, note that
    $$
    (M_{0}\cdots M_{m_2-1})(n-r,x_2)=0, \qquad \text{ for } n-r< x_2.
    $$
    Since, by assumption, $x_2\geq -1$ we can restrict the sum over $r$ to range from $r=1,\ldots,n+1$. Observe that this range is independent of $p$. Using~\eqref{eq:bounds_giving_essentialbanded3} we then find that, for some constant $C$ independent of $p$
    $$
    \left(\Lambda^{(2)}_{21} \Upsilon^{(2)}_{12}\right)_{n-s,n-r} \leq C \rho^p,
    $$
    for $s=1,\ldots, n+p$ and $r=1,\ldots, n$. Together with~\eqref{eq:boundsonallentries} this implies that 
    \begin{multline}\label{eq:dontknow}
        K(m_1,x_1,m_2,x_2)=- \mathbbm{I}[m_1 >m_2] M_{m_2+1}\cdots M_{m_1}(x_1,x_2)\\
                + \sum_{r,s=1}^{n,n+p}  (M_{m_1}\cdots M_{2n-1})(x_1,-s) (\Upsilon_{22}^{(1)} \Lambda_{22}^{(1)})_{n-s,n-r}(M_{0}\cdots M_{m_2-1})(n-r,x_2)\\+\mathcal O(p^2\rho^p ),
        \end{multline}
        as $p \to \infty$. 
        
        The last step is to take the limit $p \to \infty$. First observe that the series converges because of~\eqref{eq:bounds_giving_essentialbanded3} and~\eqref{eq:boundsonallentries}. Then, 
      $$  
            (\Upsilon_{22}^{(1)} \Lambda_{22}^{(1)})_{n-s,n-r}
            =\sum_{l=-p}^{n-1} \Upsilon^{(1)}_{n-s,l} \Lambda^{(1)}_{l,n-r}\\=\sum_{l=1}^{n+p} \Upsilon^{(1)}_{n-s,n-l} \Lambda^{(1)}_{n-l,n-r}, 
     $$   
        for $r,s=1,2, \ldots$. Inserting this back into~\eqref{eq:dontknow} and taking the limit $p \to \infty$ gives \begin{multline} \label{eq:kernelEMtop}
            \lim_{p \to \infty}  K(m_1,x_1,m_2,x_2)=- \mathbbm{I}[m_1 >m_2] M_{m_2+1}\cdots M_{m_1}(x_1,x_2)\\
                + \sum_{\ell,r,s=1}^{\infty}  (M_{m_1}\cdots M_{2n-1})(x_1,-s)\Upsilon^{(1)}_{n-s,n-\ell} \Lambda^{(1)}_{n-\ell,n-r} (M_{0}\cdots M_{m_2-1})(n-r,x_2).
         \end{multline}
         Since $\Upsilon^{(1)}_{n-s,n-\ell}=0$ for $s \leq \ell$ and $\ell\geq 1$, we can let the sum over $s$ range from $-\infty$ to $\infty$. Similarly, we can let the sum over $r$ range from $-\infty$ to $\infty$, since $r\geq 1$ and $\Lambda^{(1)}_{n-\ell,n-r}=0$ for $r<\ell$. By doing so, and writing 
         $$(M_{m_1}\cdots M_{2n-1})(x_1,-s)=(M_{m_1}\cdots M_{2n-1}S^{-n})(x_1,n-s),$$
         we prove~\eqref{eq:kernelEMtopa}.
    \end{proof}

\section{Periodic weights} \label{sec:periodic}

\subsection{Preliminaries on block Toeplitz matrices}

Let $A(z)$ be $p\times p$ matrix-valued function whose entries are rational functions in $z$. Then the doubly infinite block Toeplitz matrix $M(A(z))$ is defined as the doubly infinite matrix
$$
\left(M(A(z))_{pj+r,pk+s}\right)_{r,s=1}^{p}= \frac{1}{2 \pi \mathrm{i}} \oint_{|z|=1+\varepsilon} A(z) \frac{dz}{z^{k-j+1}},
$$
for $j,k \in \mathbb Z$. Typically, one assumes that $A(z)$ has no poles on the unit circle. In our situation, we will consider matrices with poles on the unit circle, and we integrate over a circle centered at the origin with a radius $1+\varepsilon$ where $\varepsilon$ is sufficiently small so  that all the poles of $A(z)$ and $(A(z))^{-1}$ that are outside the unit circle are also on the outside the circle of radius $1+\varepsilon$.

The values on the diagonals in the upper triangular part decay exponentially with the distance to the main diagonal. The values on the diagonals in the lower triangular part decay exponentially with the distance to the main diagonal if there is no poles on the unit circle, and remain bounded if there is a pole on the unit circle. Indeed,  by  deforming contours it is straightforward to check that
\begin{equation}\label{eq:expontialdecay1}
\left(M(A(z))_{pj+r,pk+s}\right)_{r,s=1}^{p}=\mathcal O( r_*^{j-k}),
\end{equation}
for $k-j \to \infty$,
and
\begin{equation} \label{eq:expontialdecay2}
\left(M(A(z))_{pj+r,pk+s}\right)_{r,s=1}^{p}=\mathcal O( r_{**}^{j-k})
\end{equation}
for $j-k \to \infty$, where $r_*>1$  is the radius of the pole of $A$ outside the unit circle with the smallest radius, and $r_{**} \leq 1$  is the radius of the pole of $A$ inside or on the unit circle with the largest radius.

Doubly infinite block Toeplitz matrices have the convenient property  that 
\begin{equation} \label{eq:factortoeplitz}
    M(A_1(z))M(A_2(z))=M(A_1(z)A_2(z)),
\end{equation}
for any two $p\times p$ matrix-valued functions $A_{1,2}(z)$ with rational entries.

In the upcoming discussion we will need block LU- and UL-decompositions of doubly infinite Toeplitz matrices. 

Note that if $A(z)$ has no poles inside or on the unit circle, then $M(A(z))$ is an upper triangular block matrix, and if $A(z)$ has no poles outside the unit circle then $M(A(z))$ is a lower triangular block matrix. This, together with \eqref{eq:factortoeplitz} means that finding a block UL-decomposition  for a matrix $M(A(z))$ amounts to finding a factorization 
$$
 A(z)=A_+(z) A_-(z),
$$
of the matrix-valued symbol $A(z)$ such that $A_+(z)$ has no poles inside or on the unit circle and $A_-(z)$ has no poles outside the unit circle. Similarly, a block LU-decomposition  for a matrix $M(A(z))$ amounts to finding a factorization $$
A(z)=\tilde A_-(z) \tilde A_+(z),
$$
of the matrix-valued symbol $A(z)$ such that $\tilde A_+(z)$ has no poles inside or on the unit circle and $\tilde A_-(z)$ has no poles outside the unit circle. Note that in the scalar case $p=1$, we can take $\tilde A_\pm =A_\pm$.

In a block LU- or UL-decomposition we will also want that $L^{-1}$ and $U^{-1}$ are lower and upper block triangular matrices respectively. To this end, we will also need that $A_+^{-1}$ and $\tilde A_+^{-1}$  have no pole inside or on the unit circle, and $A_-^{-1}$ and $\tilde A_-^{-1}$ have no poles outside the unit circle. By Cramer's rule this is equivalent to require that $\det A_+(z)$ has no pole inside or on the unit circle and $\det A_-(z)$ has no pole outside the unit circle.

Concluding, finding a block LU- and UL-decomposition of a doubly infinite block Toeplitz matrix with symbol $A(z)$ such that the lower and upper triangular matrices remain lower and upper triangular after taking inverses, is equivalent to finding factorizations
$$
    A(z)=A_+(z)A_-(z)=\tilde A_-(z) \tilde A_+(z),
$$
where $A_+^{\pm 1}$ and $\tilde A_+^{\pm 1}$ have no pole inside or on the unit circle, and $A_-^{-1}$ and $\tilde A_-^{-1}$ have no poles outside the unit circle. Such factorizations are called Wiener-Hopf factorizations and have been studied extensively in the literature.

\subsection{Vertically periodic transition matrices}

In this section we assume that there exists a $p$ such that, for $i=0,\ldots,n-1$ and $j \in \mathbb Z$,
\begin{equation} \label{eq:periodicparameters}
a_{i,j+p}=a_{i,j}, \quad b_{i,j+p}=b_{i,j}.
\end{equation}
In other words, the parameters are $p$–periodic in the vertical direction. Before we continue, we mention that Assumption \ref{assume} takes a simpler form. Indeed, \eqref{eq:assumption2} just means that all parameters are positive, and \eqref{eq:assumption1} is equivalent to requiring 
\begin{equation} \label{eq:assumption1periodic}
\frac{a_{i,1}\cdots a_{i,p}}{b_{i,1}\cdots b_{i,p}}<1.
\end{equation}
We will see shortly why this assumption is relevant.

With $p$-periodic parameters \eqref{eq:periodicparameters}, the transition matrices $M_{i}$ becomes block Toeplitz matrices. Indeed, for each $i\in \mathbb Z$, we have $M_i=M(A_i(z))$ with  
$$
A_{2i}(z)=\phi(z; \mathtt a_i, \mathtt b_i)=\begin{pmatrix} a_{i,1} & 0 & \cdots &0&  b_{i,1}/z \\
    b_{i,2} & a_{i,2} & 0&&0  \\
   0& \ddots & \ddots & \ddots  & \vdots\\
  \vdots& \ddots& \ddots & a_{i,p-1} & 0\\
    0&\cdots &0 &b_{i,p}& a_{i,p}
\end{pmatrix},
$$
for the symbols with even index, 
and
\begin{equation}\label{eq:oddA}
A_{2i+1}(z)=\psi(z)=\frac{1}{1-1/z}\begin{pmatrix} 1 & 1/z &  \hdots& 1/z \\
    \vdots & \ddots & \ddots & \vdots \\
     \vdots & \ddots & \ddots & 1/z \\

  1 & \hdots & \hdots& 1 
\end{pmatrix},
\end{equation}
for the symbols with odd index.

It is convenient to use the notation 
$$
s(z)=\begin{pmatrix}
    0 & \cdots & \cdots & 0 & 1/z \\
   1& \ddots &  & &0 \\
   0&\ddots&\ddots& &\vdots\\
    \vdots&\ddots&\ddots&\ddots &\vdots\\
    0&\dots&0&1& 0
\end{pmatrix}, 
$$
and observe that the shift matrix $S$ is the block Toeplitz matrix with symbol $s(z)$. 

The matrix $G=V S^{-n}$ is then the block Toeplitz matrix $M(\eta(z))$ with symbol
$$
\eta(z)=\phi(z;\mathtt a_{0}, \mathtt b_{0}) \psi(z) \cdots \phi(z;\mathtt a_{n-1}, \mathtt b_{n-1}) \psi(z) s(z)^{-n},
$$
which, using the fact that $s(z)\psi s(z)^{-1}= \psi$,  we can rewrite as 
\begin{equation} \label{eq:defeta}
    \eta(z) =\prod_{i=0}^{n-1} s(z)^{i} \phi(z;\mathtt a_{i}, \mathtt b_{i}) s(z)^{-1-i}  \psi(z).
\end{equation}
Note that the factor $\psi(z)$ and its inverse are analytic outside the circle, and $\psi$ has a pole on the unit circle. Thus, in a block LU-decomposition, we would like to have the factors $M(\psi)$ as part of $L$. Now also observe that each factor 
$$
s(z)^i \phi(z;\mathtt a_{i}, \mathtt b_{i})s(z)^{-i-1}= 
\begin{pmatrix}b_{i,1-i}& a_{i,1-i} & 0 & \cdots &0\\
   0& b_{i,2-i} & a_{i,2-i} & 0&  \\
   0&0& \ddots & \ddots   & \vdots\\
  \vdots&& \ddots& \ddots & a_{i,p-1-i} \\
   a_{i,p-i}z & 0&\cdots &0 &b_{i,p-i}
\end{pmatrix},
$$
is analytic in the unit circle, and its determinant is linear:
$$
\det\left( s(z)^i \phi(z;\mathtt a_{i}, \mathtt b_{i})s(z)^{-i-1}\right)=z a_{i,1}\cdots a_{i,p}- b_{i,1}\cdots b_{i,p}.
$$
The assumption \eqref{eq:assumption1periodic} shows that its zero is outside the unit circle and thus each factor $s(z)^i \phi(z;\mathtt a_{i}, \mathtt b_{i})s(z)^{-i-1}$ and its inverse are analytic inside and on the unit circle. This is the importance of \eqref{eq:assumption1periodic}. 

We see that for a block LU-factorization we would like to reorganize the factors in $\eta(z)$ such that the factors $s(z)^i \phi(z;\mathtt a_{i}, \mathtt b_{i})s(z)^{-i-1}$ are at the right, and the factors $\psi(z)$ are all on the left.  If $p=1$, all terms commute and this is a triviality. If $p>1$ this is not a triviality at all, and this is where the refactorization procedure of Section \ref{sec:generalLU} comes into play. Note that in this procedure we update the parameters using the maps \eqref{eq:hata} and \eqref{eq:hatb}, and it is crucial that these maps preserve the condition \eqref{eq:assumption1periodic}. That this is indeed the case, follows by the fact that $\det \phi(z;\mathtt a_i, \mathtt b_{i})=c\det \phi(z;\hat{\mathtt a}_i, \hat {\mathtt b}_i)$, and thus the location of its zero is preserved. This also directly proves Lemma \ref{lem:hatmappreserving}  in the periodic setting.

\subsection{Doubly periodic weights}

Let us now assume in addition that the weights are also periodic in the horizontal direction, and let $q$ be the smallest integer such that 
$$A_{q+i}=A_{i}, $$
and hence $M_{q+i}=M(A_{q+i})=M(A_i)=M_i$. For simplicity, we will assume that the total number of transition matrices is $qn$ (instead of $n)$. Then the product of all transfer matrices has the symbol
$$
(B(z))^n, \quad \text{with }B(z)=\prod_{j=0}^{q-1}A_j(z). 
$$
The idea that was introduced in \cite{BD19}, and used in \cite{Ber19} and \cite{BD22}, is to first find a Wiener-Hopf factorization for $B=B_+(z)B_-(z)$ (and similarly $B=\tilde B_-(z) \tilde B_+(z)$), and then continue with the symbol
$$
  (B(z))^n=B_+(z) (B_-(z)B_+(z))^{n-1} B_-(z).
$$
Now set $\hat B(z)=B_-(z)B_+(z)$ and try to find Wiener-Hopf factorization for $\hat B(z)=\hat B_+(z) \hat B_-(z)$, and so forth until there are no factors left. This gives a discrete dynamical system on the space of symbols, by computing first a factorization and then swapping the order of that factorization. In special situations, the dynamics is periodic and this helps in finding a Wiener-Hopf factorization in compact and explicit form, which greatly helps in the asymptotic study. Generally, however, it will not be periodic.  Recently, it was shown for the biased two-periodic Aztec diamond  \cite{BD22} that the dynamical system can be linearized by passing to the Jacobian of the spectral curve of $B(z)$ (which is an invariant for the flow).  In a work in progress, this is worked out in a more general situation \cite{BB22}.

\begin{rem} \label{rem:periodicitylost}
    We emphasize that the restriction that all transition matrices with odd index are given by \eqref{eq:oddA} means that we do not cover all doubly periodic  models. Although it is always possible to change the edge weights such that the face weights do not change and such that the transition matrices with odd index are the doubly infinite Toeplitz matrix with symbol  \eqref{eq:oddA}, it is not necessarily true that the corresponding transition matrices at even steps are doubly infinite Toeplitz matrices. In other words, the gauge transformations needed to turn the weights of the desired edges to $1$ do not necessarily preserve double periodicity. See also Remark \ref{rem:periodicitylostpre}.
\end{rem}

\section{The inverse of $(W_n^{\mathrm{Az}})^{-1}$} \label{sec:inversesbyKas}

In this section, we give a recurrence for $(W_n^{\mathrm{Az}})^{-1}$  using the domino shuffle.  This is a generalization of the computation which originally appeared in~\cite{CY14}.   

\subsection{Recurrence for entries of $(W_n^{\mathrm{Az}})^{-1}$} 

In order to give the recurrence for $(W_n^{\mathrm{Az}})^{-1}$, we first need to consider the partition function of the square move and introduce an additional graphical transformation.  

For a finite graph, label $Z_{\mathrm{Old}}$ to be the partition function before applying the square move to a single face with edge weights $a,b,c,$ and $d$, and label $Z_{\mathrm{New}}$ to be the partition function after applying the square move.  Then, it is easy to see that $Z_{\mathrm{Old}}=\Delta Z_{\mathrm{New}}$.  

The final graphical transformation that we need is removal of pendant edges: if a vertex is incident to exactly one edge which has weight 1, then the edge and its incident edges can be removed from the graph since the vertex must be covered by a dimer.  This transformation does not alter the partition function. 

For notational simplicity in stating the result and its proof, for  each even face with center $(2i+1,2j+1)$  of the Aztec diamond of graph of size $n$, introduce $r_1(i,j), r_2(i,j), r_3(i,j)$, and $r_4(i,j)$ to be the edge weights where the labelling proceeds clockwise starting with the north-east edge. From our choice of edge weights, we initially have $r_1(i,j)=1$, $r_2(i,j)=1$, $r_3(i,j)=b_{i,j}$ and $r_4(i,j)=a_{i,j}$, but these will change under applying the square move. We denote $F^{(n)}_{2i,j}$ to be the face weight of the Aztec diamond of size $n$ at face $(2i+1,2j+1)$ and $F^{(n)}_{2i+1,j}$ to be the face weight at face $(2i+2,2j+2)$, where the superscript marks the size of the Aztec diamond.  This has the same convention as given in \cref{subsec:shuffle}, and we write $F^{(n)}$ to be the collection of face weights of an Aztec diamond of size $n$. We remind the reader that the face weight at the face whose center is given by $(2i+1,2j+1)$ equals
\begin{equation}
F^{(n)}_{2i,j}=\frac{r_2(i,j)r_4(i,j)}{r_1(i,j)r_3(i,j)} \label{eq:Aztecfaceeven}
\end{equation}
for $0\leq i,j \leq n-1$
while the face whose center is given by $(2i+2,2j+2)$ equals
\begin{equation}
F^{(n)}_{2i+1,j}=\frac{r_1(i,j)r_3(i+1,j+1)}{ r_2(i+1,j)r_4(i,j+1)}
\end{equation}
for $0\leq i,j \leq n-2$. 

To preempt the use of the domino shuffle, we write  $(W_n^{\mathrm{Az}})^{-1} (i,j,F^{(n)})$ instead of $(W_n^{\mathrm{Az}})^{-1} (i,j)$ for $1 \leq i,j \leq n$ and $(K_n^{\mathrm{Az}})^{-1}((x,y),F^{(n)})$ instead of $(K_n^{\mathrm{Az}})^{-1}(x,y)$ for $x \in \mathtt{W}_n^{\mathrm{Az}}$ and $y \in \mathtt{B}_n^{\mathrm{Az}}$. Here, the third argument in each case marks the face weights as these will change under successive iterations of the domino shuffle. 

 For the purposes of the proof of the result below, set, for $1\leq i,j\leq n$, $\Delta(i,j)=r_1(i,j)r_3(i,j)+r_2(i,j)r_4(i,j)$, $Z_n(F^{(n)})$ to be the partition function of the Aztec diamond  and 
\[
Z_n(i,j,F^{(n)})=\left\{ \begin{array}{ll}
| (K_n^{\mathrm{Az}})^{-1}((2i-1,0),(0,2j-1),F^{(n)})\det  K_n^{\mathrm{Az}}| & \mbox{for }1\leq i,j\leq n\\
0 & \mbox{otherwise.} \end{array} \right.
\]

From~\eqref{thm:eqAztecLGV} and a computation of the $\mathrm{sgn} ((K_n^{\mathrm{Az}})^{-1}((2i-1,0),(0,2j-1),F^{(n)}))$ given in ~\cite[Lemma 3.6]{CY14}, we have that
\begin{equation}\label{eq:WntoKnAz}
    \mathrm{i}^{i+j-1}(W_n^{\mathrm{Az}})^{-1} (i,j,F^{(n)})= (K_n^{\mathrm{Az}})^{-1}((2i-1,0),(0,2j-1),F^{(n)})
\end{equation}
for $1 \leq i,j \leq n$.
The vertices $(2i-1,0)$ and $(0,2j-1)$ are on the boundary of the Aztec diamond and so each term on the right side of the above equation comes with the same sign. 
We are now in the position to give a recurrence for the entries of $(W_n^{\mathrm{Az}})^{-1}$. The main ideas originate from~\cite{CY14}, in particular see Lemma 3.2 in that paper for simplest case, and so we give a shortened proof.  
\begin{prop} 
    For $1\leq i,j\leq n$, we have
\begin{equation}
\begin{split}
&(W_n^{\mathrm{Az}})^{-1} (i,j,F^{(n)}) =   
\sum_{k,l \in\{0,1\}} 
\frac{r_1(0,j-1)^{l} F^{(n)}_{0,j}}{r_2(0,j-1)^l r_4(0,j-1)(1+F^{(n)}_{0,j})}
\\&\times\frac{r_1(i-1,0)^k F_{2i+2,0}^{(n)}}{r_4(i-1,0)^kr_2(i-1,0) (1+ F^{(n)}_{2i+2,0})}
(W_{n-1}^{\mathrm{Az}})^{-1} (i-k,j-l,F^{(n-1)})\\
&+\frac{r_1(0,0)}{r_1(0,0)r_3(0,0)+r_2(0,0)r_4(0,0)}\mathbbm{I}[(i,j)=(1,1)]
\end{split}
\end{equation}
where for $0\leq i,j \leq n-3$
\begin{equation}
\label{prop:brec:eq1}
    F_{2i+1,j}^{(n-1)}=\frac{1}{F_{2i+1,j+1}^{(n)}}
\end{equation}
and for $0 \leq i,j\leq n-2$
\begin{equation}
\label{prop:brec:eq2}
        F_{2i,j}^{(n-1)}=\frac{F_{2i+1,j}^{(n)}\big(1+F_{2i+2,j+1}^{(n)}\big)\big(1+F_{2i,j}^{(n)}\big)}{\big(1+ (F_{2i,j+1}^{(n)})^{-1}\big) \big(1+ (F_{2i+2,j}^{(n)})^{-1}\big)}.
\end{equation}

\end{prop}

\begin{proof}
    We apply the domino shuffle on all the even faces and remove the pendant edges. This gives an Aztec diamond of size $n-1$; see Fig.~\ref{fig:Aztec4shuffle}.  We apply a shift so that the bottom left most vertex of the Aztec diamond of size $n-1$ has coordinates $(1,0)$.  The change of weights is given by \cref{prop:squaremoveall} which can be shown to equal~\eqref{prop:brec:eq1} and~\eqref{prop:brec:eq2} after the shift.  The change in the partition function is given by 
\begin{equation}\label{proppf:brec:pf}
Z_n(F^{(n)}) = \prod_{0\leq p,q\leq n-1} \Delta(p,q) Z_{n-1}(F^{(n-1)}),
\end{equation}
where the product is over all the even faces in the Aztec diamond of size $n$.   

	Next notice that $Z_n(i,j,F^{(n)})$ is equivalent to adding a pendant edge to $(2i-1,0)$ and another pendant edge to $(0,2j-1)$ for $1 \leq i,j\leq n$.  We apply the domino shuffle on the even faces of the graph, and we obtain 
\begin{equation}
\begin{split}
&Z_n(i,j,F^{(n)}) =\prod_{0\leq p,q\leq n-1} \Delta(p,q) \bigg( \sum_{k,l \in\{0,1\}}   \frac{r_1(0,j-1)^{1-l}r_2(0,j-1)^{l}}{\Delta(0,j-1)} \\ 
&\frac{r_1(i-1,0)^{1-k}r_4(i-1,0)^{k}}{\Delta(i-1,0)} Z_{n-1}(i-k,j-l,F^{(n-1)})\\
&+\frac{r_1(0,0)}{\Delta(1,1)}\mathbbm{I}[(i,j)=(1,1)]Z_{n-1}(F^{(n-1)})\bigg).
\end{split}
\end{equation}
    Fig.~\ref{fig:Aztec4shufflependant} shows an example of adding pendant edges and applying the domino shuffle. 
    \begin{center}
\begin{figure}
    \includegraphics[height=5cm]{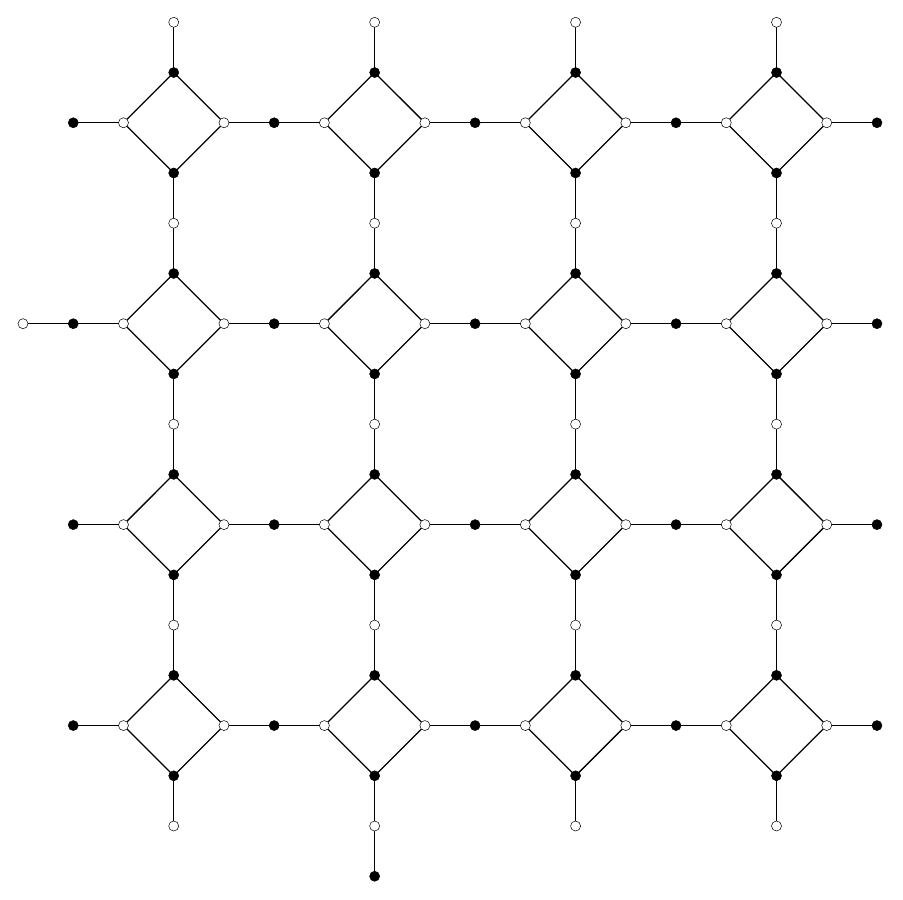}\,\,\,
  \includegraphics[height=5cm]{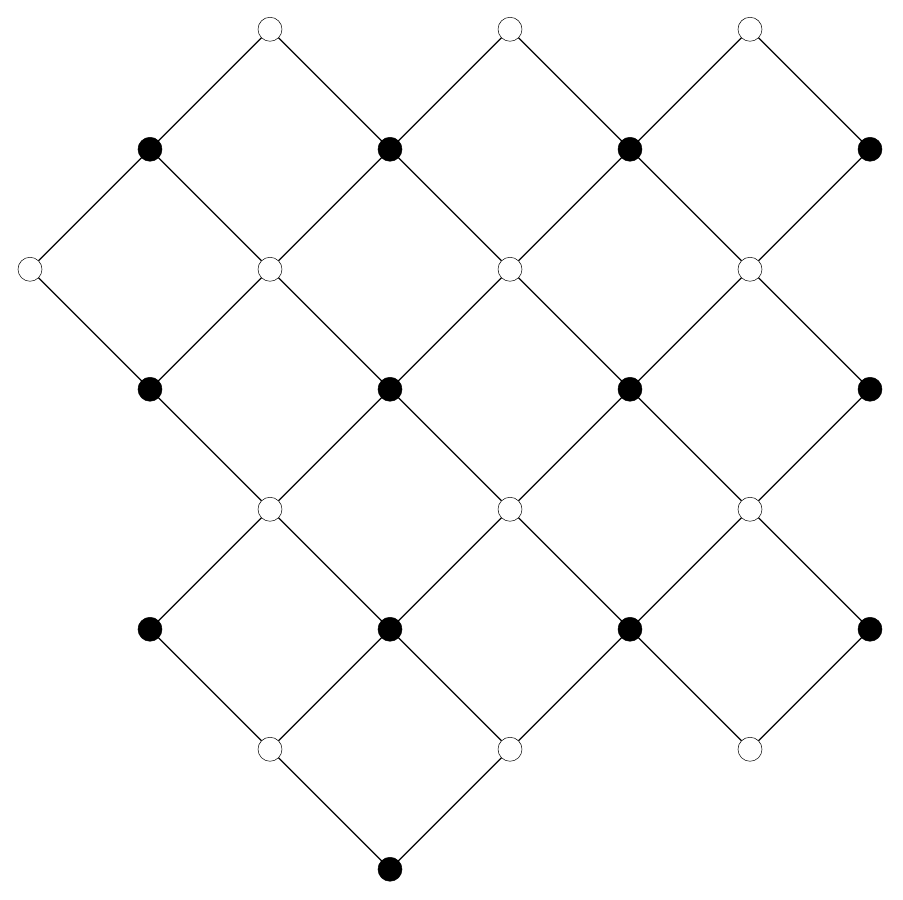}
    \caption{ The left figure shows adding two pendant edges incident to $(0,5)$ and $(3,0)$ and then applying the domino shuffle.  The figure on the right shows the resulting graph after contracting two-valent vertices and removing pendant edges. 
  }
    \label{fig:Aztec4shufflependant}
\end{figure}
\end{center}

We divide the above equation by~\eqref{proppf:brec:pf} to obtain
\begin{equation}\label{proppf:brec:pf2}
\begin{split}
&\frac{Z_n(i,j,F^{(n)})}{Z_n(F^{(n)})} =  \bigg( \sum_{k,l \in\{0,1\}}   \frac{r_1(0,j-1)^{l}r_2(0,j-1)^{1-l}}{\Delta(0,j-1)} \\ 
&\frac{r_1(i-1,0)^{k}r_4(i-1,0)^{1-k}}{\Delta(i-1,0)} \frac{Z_{n-1}(i-k,j-l,F^{(n-1)})}{Z_{n-1}(F^{(n-1)})}\\
&+\frac{r_1(0,0)}{\Delta(0,0)}\mathbbm{I}[(i,j)=(1,1)]\bigg).
\end{split}
\end{equation}
Using~\eqref{proppf:brec:pf2}, the above equation and~\eqref{eq:WntoKnAz} we arrive at
\begin{equation}\label{proppf:brec:pf3}
\begin{split}
    &(W_n^{\mathrm{Az}})^{-1} (i,j,F^{(n)}) =  \bigg( \sum_{k,l \in\{0,1\}} \\&  \frac{r_1(0,j-1)^{l}r_2(0,j-1)^{1-l}}{\Delta(0,j-1)}  
\frac{r_1(i-1,0)^{k}r_4(i-1,0)^{1-k}}{\Delta(i-1,0)} \frac{Z_{n-1}(i-k,j-l,F^{(n-1)})}{Z_{n-1}(F^{(n-1)})}\\
&+\frac{r_1(0,0)}{\Delta(1,1)}\mathbbm{I}[(i,j)=(1,1)]\bigg).
\end{split}
\end{equation}
Rearranging the above equation and using~\eqref{eq:Aztecfaceeven} gives the result. 

\end{proof}

\subsection{Remaining entries}

We will not give a computation to obtain the rest of the entries of $(K_n^{\mathrm{Az}})^{-1}$ for the Aztec diamond but will give two tractable approaches. The first is to notice that $(K_n^{\mathrm{Az}})^{-1} \cdot K_n^{\mathrm{Az}}=  K_n^{\mathrm{Az}} \cdot( K_n^{\mathrm{Az}})^{-1} =\mathbb{I}$ entry-wise is a recurrence relation.  Applying iteratively, we can view these equations as moving the white and black vertices to the boundary, that is we can express $(K_n^{\mathrm{Az}})^{-1}(x,y)$ as a linear combination of $(W_n^{\mathrm{Az}})^{-1} (i,j,F^{(n)})$ for $1 \leq i,j\leq n$.  Another approach is to use Theorem~\ref{thm:AztecLGV} and directly invert $D_n$ in the statement of this theorem.  Due to the choice of labelling, $D_n$ is a triangular matrix which means that its inverse is theoretically computable.  {Note that both of these approaches  extend to computing entries $(K_{n,p}^{\mathrm{Tow}})^{-1}$ (and thus entries $(K_{n}^{\mathrm{Az}})^{-1}$) from $(\tilde{W}_{n,p}^{\mathrm{Tow}})^{-1}$. }

\appendix

\section{LGV Theorem} \label{sec:app:lgv}
 
Let $G=(V,E)$ be a  directed acyclic graph and $w:E \mapsto \mathbb R$ a weight function. For any path $\pi$ in the graph $G$ we define the weight $w(\pi)$ as 
$$
    w(\pi)= \prod_{e \in \pi} w(e).
$$
For a vector $\vec \pi=(\pi_1,\ldots,\pi_N)$ we define its weight
$$
    w(\vec \pi)= \prod_{j=1}^n w(\pi_j).
$$ 

For any $a, b \in V$, let $\Pi(a,b)$ be the set of all paths from $a$ to $b$.  We then define a matrix $M: V \times V \to \mathbb R$  by setting 
$$
 M_{a,b}=\sum_{\pi \in \Pi(a,b)} w(\pi).
$$
For $N\in \mathbb N$, $\vec a=(a_1,\ldots,a_N) \in V^N$ and $\vec b=(b_1,\ldots,b_N)\in V^N$, let $\Pi(\vec a,\vec b)$ be the set of tuple of paths $(\pi_1,\ldots,\pi_n)$ that $\pi_j$ connects $a_j$ to $b_j$. Let $\Pi_{n.i}(\vec a, \vec b)$ be the set of all  such tuples such that no two paths have a vertex in common. For $\sigma \in S_n$ we define $\vec b_\sigma=(b_{\sigma(1)},\ldots,b_{\sigma(n)})$.
\begin{thm}\cite{GesVie85,Lin73}\label{thm:lgv}
    Let $G=(V,E)$ be a weighted directed acyclic graph, $N\in \mathbb N$, $\vec a=(a_1,\ldots,a_N) \in V^N$ and $\vec b=(b_1,\ldots,b_N)\in V^N$. Then 
    $$
        \det \left(M_{a_i,b_j}\right)_{i,j=1}^N= \sum_{\sigma \in S_n} \mathrm{sign} (\sigma) \sum_{\pi \in \Pi_{n.i.}(\vec a,\vec b_{\sigma})}  w(\pi).
    $$
    In particular, if $\Pi_{n.i}(\vec a, \vec b_\sigma)\neq \emptyset$ if and only if $\sigma=id$, then
    $$
    \det \left(M_{a_i,b_j}\right)_{i,j=1}^N= \sum_{\pi \in \Pi_{n.i}(\vec a, \vec b)} w(\pi).
    $$
\end{thm}

\section{Processes from products of block Toeplitz minors} \label{appendix:BD}

Let $p,n,N \in \mathbb N$ and set $x_j^0=x_j^N=-j$ for $j=1,\ldots,pN$. Then consider the point processes on $\{1,2,\ldots,n-1\} \times \mathbb Z$ defined by the probability of having a point configuration $\{(m,x_j^m)\}_{m,j=1}^{n-1,pN} \subset \{1,\ldots,n-1\} \times \mathbb Z$ is given by 
$$
\prod_{m=0}^{N-1} \det \left(M_m(x_j^m,x_k^{m+1})\right)_{j,k=1}^n,
$$ 
where $M_m$ is the block Toeplitz 
$$
\left[M_m(px+r,py+s)\right]_{r,s=0}^{p-1}=\frac{1}{2 \pi \mathrm{i}} \oint_{|z|=1} \phi_m(z) \frac{dz}{z^{{y-x}+1}},
$$
with symbol $\phi_m(z)$ such that $M_m$ is a totally non-negative matrix, and the entries of $\phi_m(z)$ are rational functions of $z$. 

The Eynard-Mehta Theorem  \cite{EM97} tells us that this probability measure on $ \{1,\ldots,n-1\} \times \mathbb Z$ is a determinantal point process with kernel given by 
\begin{multline}\label{eq:knbeforelimit}
    K_N(m_1,x_1,m_2,x_2)=- \mathbbm{I}[m_1 >m_2] M_{m_2+1}\cdots M_{m_1}(x_1,x_2)\\
    + \sum_{r,s=1}^{pN}  (M_{m_1}\cdots M_{2n-1})(x_1,-s) (W^{-1})_{s,r} (M_{0}\cdots M_{m_2-1})(-r,x_2).
    \end{multline}
    where 
    $$
     W_{r,s}=(M_0 \ldots M_{n-1})_{-r,-s}, \qquad r,s=1,\ldots,pN.
    $$
    Now $M_m=M(\phi_m)$ and \eqref{eq:factortoeplitz} give that  $$M_0 \ldots M_{n-1}=M(\phi_0\cdots \phi_{n-1})=M(\phi).$$
    Hence, $W$ is a submatrix of the doubly infinite Toeplitz matrix $M(\phi)$.  (Note that compared with Theorem ~\ref{thm:em} we have, without loss of generality, included a shift so that all parts end at the same height that they started at.)

    Below, we present a theorem that was first proved in \cite[Thm 3.1]{BD19}. It shows that when $N\to \infty$ there are two limiting processes for which the correlation functions can be computed in terms of double integrals with matrix-valued integrands.  One process focuses on the top of the paths, and the other on the bottom. The hope is that these double integral formulas can be used for asymptotic analysis, as $n \to \infty$. 
    
    The method of proof in \cite{BD19} was using a Riemann-Hilbert analysis for the matrix orthogonal polynomials introduced by \cite{DK21}. Here we will present a different proof that is more in line with the  more standard computation of the correlation kernel for Schur process (this is the special case of $p=1$ and $N\to \infty$), as for instance given in \cite{Joh17}.

\begin{thm}\label{thm:bd19} \cite[Thm 3.1]{BD19}
    Assume that 
$$
\phi=\phi_0\ldots \phi_{n-1},
$$
is analytic in an annulus $\{z \mid 1-2\varepsilon < |z|<1+2\varepsilon\}$ for some $\varepsilon>0$, and that there exists  factorizations
$$
\phi=\phi_+\phi_-=\tilde \phi_- \tilde \phi_+,
$$
such that 
\begin{itemize}
    \item $\phi_+^{\pm 1}(z),\tilde \phi_+^{\pm 1}(z)$ are analytic for $|z|>1-$.
    \item  $\phi_-^{\pm 1}(z),\tilde \phi_-^{\pm 1}(z)$ are analytic for $|z|<1+$.
    \item $\tilde \phi_-(z),\phi_-(z)\sim I_p$, as $z \to \infty$.
\end{itemize}
    Then, as $N \to \infty$, the point process for the top paths converges and the limit is the determinantal point process  defined by the kernel 
    \begin{multline} \label{eq:limittop}
   \lim_{N\to \infty}\left[ K_N(m_1,p y_1+j_1 ,m_2,p y_2+j_2)\right]_{j_1,j_2=0}^{p-1}\\=- \mathbbm{I}[m_1 >m_2] \oint_{|z|=1} \phi_{m_2}(z) \cdots \phi_{m_1-1}(z)\frac{dz}{z^{y_2-y_1+1}} \\
    +\frac{1}{(2 \pi \mathrm{i})^2} \oint_{|w|=1} \oint_{|z|=1+\varepsilon} \left(\prod_{k=m_1}^{n-1} \phi_{k}(w)\right) \tilde \phi_+^{-1}(w) \tilde \phi_-^{-1}(z) \left(\prod_{k=0}^{m_2-1} \phi_k(z)\right)  \frac{w^{y_1} dz dw }{z^{y_2+1}(z-w)},
    \end{multline}
    and, the point process for the bottom paths converges and the limit is the determinantal point process defined by the kernel 
    \begin{multline} \label{eq:limitbottom}
        \lim_{N\to \infty} \left[ K_N(m_1,p (-N+y_1)+j_1 ,m_2,p (-N+y_2)+j_2)\right]_{j_1,j_2=0}^{p-1}\\=- \mathbbm{I}[m_1 >m_2] \oint_{|z|=1} \phi_{m_2}(z) \cdots \phi_{m_1-1}(z)\frac{dz}{z^{y_2-y_1+1}} \\
         -\frac{1}{(2 \pi \mathrm{i})^2} \oint_{|w|=1+\varepsilon} \oint_{|z|=1} \left(\prod_{k=m_1}^{n-1} \phi_{k}(w)\right) \phi_-^{-1}(w)\phi_+^{-1}(z) \left(\prod_{k=0}^{m_2-1} \phi_k(z)\right)  \frac{w^{y_1} dz dw }{z^{y_2+1}(z-w)}.
         \end{multline}
\end{thm}
\begin{proof}
  By the condition this has the (block) LU- and UL-decompositions
        $$
            M(\phi)=M(\tilde \phi_-)M(\tilde \phi_+)= M(\phi_+)M(\phi_-).
        $$
        We are thus in the setting of Section \ref{sec:wid} with 
        $$
        L^{(1)}=M(\tilde \phi_-),\ U^{(1)}=M(\tilde \phi_+),\
        L^{(2)}=M(\phi_-), \ U^{(1)}=M(\phi_+),
        $$
        and 
        $$
        \Lambda^{(1)}=M(\tilde \phi_-^{-1}), \ \Upsilon^{(1)}=M(\tilde \phi_+^{-1}),
       \  \Lambda^{(2)}=M(\phi_-^{-1}), \ \Upsilon^{(1)}=M(\phi_+^{-1}).
        $$
        This is exactly the case for which Proposition \ref{prop:asymptoticinverse} was already proved in \cite{Wid74}.

        We will proceed with proving \eqref{eq:limittop} first. By \eqref{eq:approximateinverse2}, \eqref{eq:G22iinv1}, \eqref{eq:expontialdecay1} and \eqref{eq:expontialdecay2}, we see that there exists $0<\rho<1$ such that
        \begin{multline}
            (W^{-1})_{s,r}= \sum_{\ell=1}^{pN} \left(M(\tilde \phi_+^{-1})\right)_{-s,-\ell}\left(M(\tilde \phi_-^{-1})\right)_{-\ell,-r}\\- \sum_{\ell=pN+1}^{\infty} \left(M (\phi_-^{-1})\right)_{-s,-\ell}\left(M(\phi_+^{-1})\right)_{-\ell,-r}+\mathcal O(\rho^N),
        \end{multline}
        as $N \to \infty$.
    
     By inserting this into \eqref{eq:knbeforelimit} and using \eqref{eq:expontialdecay1} and \eqref{eq:expontialdecay2}, we see that for fixed $x_1,x_2$   the main contribution in the limit $N\to \infty$ comes from small values for $s$ and $r$, and
    \begin{equation}
\begin{split}
        & \lim_{N\to \infty} K_n(m_1,x_1,m_2,x_2)=- \mathbbm{I}[m_1 >m_2] M(\phi_{m_2}\cdots \phi_{m_1-1})_{x_1,x_2}\\
        &+ \sum_{\ell,r,s=1}^{\infty}  (M(\phi_{m_1} \cdots \phi_{n-1}))(x_1,-s)\left(M(\tilde \phi_+^{-1})\right)_{-s,-\ell}\left(M(\tilde \phi_-^{-1})\right)_{-\ell,-r} \\ &\times (M(\phi_0 \cdots \phi_{m_2-1})(-r,x_2).
\end{split}
        \end{equation} 
        Now note that the fact that $M(\tilde \phi_+^{-1})$ is (block) upper triangular, we have that $\left(M(\tilde \phi_+^{-1})\right)_{-s,-\ell}=0$ for $\ell>1$ and $s\leq 0$, and thus we can sum $s$ from $-\infty$ to $\infty$. Similarly, we let $r$ range from $-\infty$ to $\infty$.  This gives
        \begin{multline}\label{eq:almostmmm}
            \lim_{N\to \infty} K_n(m_1,x_1,m_2,x_2)=- \mathbbm{I}[m_1 >m_2] M(\phi_{m_2}\cdots \phi_{m_1-1})_{x_1,x_2}\\
            + \sum_{\ell=1}^{\infty}  \left(M(\phi_{m_1}\cdots \phi_{n-1}\tilde \phi_+^{-1})\right)_{x_1,-\ell}\left(M(\tilde \phi_-^{-1} \phi_{0}\cdots \phi_{m_2-1})\right)_{-\ell,x_2}.
            \end{multline}
        Next we turn this into block form. By setting $x_i=py_i+j_i$, and setting $-\ell=- p \lambda+v)$, we find
        \begin{multline}\label{eq:closetoalmost}
            \sum_{\ell=1}^{\infty}  \left(M(\phi_{m_1}\cdots \phi_{n-1}\tilde \phi_+^{-1})\right)_{x_1,-\ell}\left(M(\tilde \phi_-^{-1} \phi_{0}\cdots \phi_{m_2-1})\right)_{-\ell,x_2}
            \\=\sum_{\lambda=1}^{\infty} \sum_{v=0}^{p-1} \left(M(\phi_{m_1}\cdots \phi_{n-1}\tilde \phi_+^{-1})\right)_{py_1+j_1,-\lambda p+v}\left(M(\tilde \phi_-^{-1} \phi_{0}\cdots \phi_{m_2-1})\right)_{-\lambda p +v,py_2+j_2}
        \end{multline}
        Moreover,
            \begin{multline}
               \left[ \sum_{\lambda=1}^{\infty} \sum_{v=0}^{p-1} \left(M(\phi_{m_1}\cdots \phi_{n-1}\tilde \phi_+^{-1})\right)_{py_1+j_1,-\lambda p+v}\left(M(\tilde \phi_-^{-1} \phi_{0}\cdots \phi_{m_2-1})\right)_{-\lambda p +v,py_2+j_2}\right]_{j_1,j_2=0}^{p-1}\\
                =
               \sum_{\lambda=1}^\infty  \frac{1}{2 \pi \mathrm{i}}\oint_{|w|=1} \phi_{m_1}(w)\cdots \phi_{n-1}(w)\tilde \phi_+^{-1}(w) w^{x_1+\lambda} \frac{dw}{w}\\ \times \frac{1}{2 \pi \mathrm{i}} \oint_{|z|=1}\tilde \phi_-^{-1}(z) \phi_0(z)\cdots \phi_{m_2-1}(z)  \frac{dz}{z^{y_2+\lambda+1}}.
            \end{multline}
        By changing the sum and the integrals, and using 
        $$
        \sum_{\lambda=1}^\infty \left(\frac{w}{z}\right)^\lambda=\frac{w}{z-w},
        $$
        which converges for $|z|>|w|$, we find
        \begin{multline}
            \left[ \sum_{\lambda=1}^{\infty} \sum_{v=0}^p \left(M(\phi_{m_1}\cdots \phi_{n-1}\tilde \phi_+^{-1})\right)_{py_1+j_1,-\lambda p+v}\left(M(\tilde \phi_-^{-1} \phi_{0}\cdots \phi_{m_2-1})\right)_{-\lambda p +v,py_2+j_2}\right]_{j_1,j_2=0}^{p-1}\\
            \frac{1}{(2 \pi \mathrm{i})^2} \oint_{|w|=1} \oint_{|z|=1+\varepsilon} \left(\prod_{k=m_1}^{n-1} \phi_{k}(w)\right) \tilde \phi_+^{-1}(w) \tilde \phi_-^{-1}(z) \left(\prod_{k=0}^{m_2-1} \phi_k(z)\right)  \frac{w^{x_1} dz dw }{z^{x_2+1}(z-w)}.
        \end{multline}
        By combining this with \eqref{eq:almostmmm} and \eqref{eq:closetoalmost} we find \eqref{eq:limittop}.

        Next, we prove \eqref{eq:limitbottom}. The proof is very similar to the proof of \eqref{eq:limittop}. We start by changing the summation variable in \eqref{eq:knbeforelimit} and write
        \begin{multline}
            K_N(m_1,x_1,m_2,x_2)=- \mathbbm{I}[m_1 >m_2] M_{m_2+1}\cdots M_{m_1}(x_1,x_2)\\
            + \sum_{r,s=0}^{pN-1}  \left(M(\phi_{m_1}\ldots \phi_{n-1})\right)_{x_1,-pN+s} (W^{-1})_{-pN+s,-pN+r}  \left(M(\phi_{0}\ldots \phi_{m_2-1})\right)_{pN+r,x_2}.
            \end{multline}
            Since $x_i=p(-N+y_i)+j_i$ and by \eqref{eq:expontialdecay1} and \eqref{eq:expontialdecay2}, the main contribution in this sum comes from  finite values of $s,r$. Arguing as above for the case \eqref{eq:limittop}, we now find that there exists $0<\rho<1$ such that
            $$
            (W^{-1})_{-pN+s,-pN+r}\sum_{\ell=0}^{pN-1} \left(M( \phi_-^{-1})\right)_{-pN+s,-pN+\ell}\left(M(\phi_+^{-1})\right)_{-pN+\ell,-pN+r}+\mathcal O(\rho^N),
        $$
        and
            \begin{multline}\label{eq:almostmmmbottom}
              K_n(m_1,-pN+x_1,m_2,-pN+x_2)=- \mathbbm{I}[m_1 >m_2] M(\phi_{m_2}\cdots \phi_{m_1-1})_{-pN+x_1,-pN+x_2}\\
                 + \sum_{\ell=0}^{pN-1}  \left(M(\phi_{m_1}\cdots \phi_{n-1} \phi_-^{-1})\right)_{-pN+x_1,-pN+\ell}\left(M(\phi_+^{-1} \phi_{0}\cdots \phi_{m_2-1})\right)_{-pN+\ell,-pN+x_2}+\mathcal O(\rho^N),
            \end{multline}
            as $N\to \infty$.
            It remains to put  this expression in block form and take the limit $N\to \infty$. To this end, first note that  $M(\phi_{m_2}\cdots \phi_{m_1-1})_{-pN+x_1,-pN+x_2}=  M(\phi_{m_2}\cdots \phi_{m_1-1})_{x_1,x_2}$ does not depend on $N$. For the second term on the right-hand side of \eqref{eq:almostmmmbottom} we note that (after setting $\ell=p\lambda+v$ and $x_i=pN+y_i+j_i$)
            \begin{multline*}
                \left[\sum_{\ell=0}^{pN-1}  \left(M(\phi_{m_1}\cdots \phi_{n-1} \phi_-^{-1})\right)_{-pN+py_1+j_1,-pN+\ell}\right.\\ \times \left.\left(M(\phi_+^{-1} \phi_{0}\cdots \phi_{m_2-1})\right)_{-pN+\ell,-pN+py_2+j_2}\right]_{j_1,j_2=0}^{p-1}\\
                =\sum_{\lambda=0}^{N-1}  \frac{1}{2 \pi \mathrm{i}}\oint_{|w|=1} \phi_{m_1}(w)\cdots \phi_{n-1}(w) \phi_-^{-1}(w)  \frac{w^{y_1} dw}{w^{\lambda+1}}\\ \times \frac{1}{2 \pi \mathrm{i}} \oint_{|z|=1}\phi_+^{-1}(z) \phi_0(z)\cdots \phi_{m_2-1}(z)  \frac{z^{\lambda}dz}{z^{y_2+1}}.
            \end{multline*}
            Now take the sum under the integrals and use that $\sum_{\lambda=0}^\infty(z/w)^\lambda=w/(w-z)$. After that, a simple limit $N\to \infty$ proves \eqref{eq:limitbottom}.
\end{proof}

\bibliographystyle{alpha}
\bibliography{Biblio}

\end{document}